\documentclass[aop,MSNbibl,dvips]{arximspdf}

\doi{10.1214/12-AOP769} 
\volume{41}
\issue{4}
\pubyear{2013}
\firstpage{2820}
\lastpage{2879}

\makeatletter

\newtheorem{theorem}{Theorem}[section]
\newtheorem{lemma}[theorem]{Lemma}
\newtheorem{corollary}[theorem]{Corollary}
\newtheorem{proposition}[theorem]{Proposition}

\newproclaim{definition}[theorem]{Definition}
\newproclaim{remark}[theorem]{Remark}

\newcommand{\rrVert}{\Vert}
\newcommand{\rrvert}{\vert}
\newcommand{\llVert}{\Vert}
\newcommand{\llvert}{\vert}

\newcommand{\cal}{\mathcal}

\newcommand{\Ten}{\,\widetilde{\otimes}\,}
\newcommand{\hten}{{\mathfrak H}}
\newcommand{\law}{\stackrel{\cal L}{\longrightarrow}}

\makeatother

\begin{document}
\begin{frontmatter}

\title{Central limit theorem for a Stratonovich integral with Malliavin calculus}
\runtitle{CLT for a Stratonovich integral}

\begin{aug}
\author[A]{\fnms{Daniel} \snm{Harnett}\corref{}\ead[label=e1]{dharnett@math.ku.edu}}
\and
\author[A]{\fnms{David} \snm{Nualart}\thanksref{t1}}
\runauthor{D. Harnett and D. Nualart}
\affiliation{University of Kansas}
\address[A]{Department of Mathematics\\
University of Kansas\\
405 Snow Hall\\
Lawrence, Kansas 66045-2142\\
USA\\
\printead{e1}} 
\end{aug}

\thankstext{t1}{Supported by NSF Grant DMS-09-04538.}

\received{\smonth{5} \syear{2011}}
\revised{\smonth{4} \syear{2012}}

%
\begin{abstract}
The purpose of this paper is to establish the convergence in law of the
sequence of ``midpoint'' Riemann sums for a stochastic process of the
form $f'(W)$, where $W$ is a Gaussian process whose covariance function
satisfies some technical conditions. As a consequence we derive a
change-of-variable formula in law with a second order correction term
which is an It\^o integral of $f''(W)$ with respect to a Gaussian
martingale independent of $W$. The proof of the convergence in law is
based on the techniques of Malliavin calculus and uses a central limit
theorem for $q$-fold Skorohod integrals, which is a multi-dimensional
extension of a result proved by Nourdin and Nualart [\textit{J.
Theoret. Probab.} \textbf{23} (2010) 39--64]. The results proved in
this paper are generalizations of previous work by Swanson
[\textit{Ann. Probab.} \textbf{35} (2007) 2122--2159] and Nourdin and
R\'eveillac [\textit{Ann. Probab.} \textbf{37} (2009) 2200--2230], who
found a similar formula for two particular types of bifractional
Brownian motion. We provide three examples of Gaussian processes $W $
that meet the necessary covariance bounds. The first one is the
bifractional Brownian motion with parameters $H \le1/2$, $HK = 1/4$.
The others are Gaussian processes recently studied by Swanson
[\textit{Probab. Theory Related Fields} \textbf{138} (2007) 269--304],
[\textit{Ann. Probab.} \textbf{35} (2007) 2122--2159] in connection
with the fluctuation of empirical quantiles of independent Brownian
motion. In the first example the Gaussian martingale is a Brownian
motion, and expressions are given for the other examples.
\end{abstract}

%

\pdfkeywords{60H05, 60H07, 60F05, 60G15,
Ito formula, Skorohod integral, Malliavin calculus,
fractional Brownian motion}

\end{frontmatter}

\section{Introduction}\label{sec1}
The aim of this paper is to obtain a change-of-variable formula in
distribution for a class of Gaussian stochastic processes $W = \{W_t,
t\ge0\}$ under certain conditions on the covariance function. These
conditions are
in the form of upper bounds on the covariance of process increments.
For example, the variance on the increment on an interval of length $s$
is bounded by $C\sqrt{s}$, and the covariance between the increments in
the intervals $[t-s,t]$, and $[r-s,r]$ is bounded by
\[
s^2 |t-r|^{-\alpha}( r - s)^{-\beta} +
s^2|t-r|^{-{3/2}},
\]
if $0<2s\le r<t$ and $|t-r|\ge2s$, where $1<\alpha\le\frac32$ and
$\alpha+ \beta=\frac32$.

For this process and a suitable function $f$, we study the behavior of
the ``midpoint'' Riemann sum
\[
\Phi_n(t):= \sum_{j=1}^{ \lfloor{nt/2} \rfloor}
f'(W_{({2j-1})/{n}}) (W_{{2j}/{n}} - W_{({2j-2})/{n}} ).
\]
The limit of this sum as $n$ tends to infinity is the Stratonovich
midpoint integral, denoted by $\int_0^t f'(W_s)^\circ \,dW_s$.
We show that the couple of processes
$ \{ (W_t, \Phi_n(t) ), t\ge0 \}$ converges in
distribution in the Skorohod space $({\mathbb D}[0,\infty))^2$ to
$
\{ (W_t, \Phi(t)), t\ge0 \}$, where
\[
\Phi(t) = f(W_t) - f(W_0) - \frac{1}{2}\int
_0^t f''(W_s)
\,dB_s,
\]
and $B=\{B_t, t\ge0\}$ is a Gaussian martingale independent of $W$
with variance $\eta(t)$, depending on the covariance properties of $W$.
This limit theorem can be reformulated by saying that the following It\^
o formula in distribution holds:
%
\begin{equation}
\label{MainResult}
f(W_t) \stackrel{\cal L} {=} f(W_0) +
\int_0^t f'(W_s)^\circ
\,dW_s + \frac{1}{2}\int_0^t
f''(W_s) \,dB_s.
\end{equation}

The above mentioned convergence is proven by showing the stable
convergence of a $d$-dimensional vector $(\Phi_n(t_1),\ldots, \Phi
_n(t_d))$ and a tightness argument. To show the convergence in law of
the finite-dimensional distributions, we show first, using the
techniques of Malliavin calculus, that $\Phi_n(t)$ is asymptotically
equivalent to a sequence of iterated Skorohod integrals involving
$f''(W_t)$. We then apply our $d$-dimensional version of the central
limit theorem for multiple Skorohod integrals proved by Nourdin and
Nualart in~\cite{NoNu}.

Recent papers by Swanson~\cite{Swanson}, Nourdin and R\'eveillac \cite
{NoRev} and Burdzy and Swanson~\cite{Burdzy} presented results
comparable to (\ref{MainResult}) for a specific stochastic process. In
\cite{Swanson}, a change-of-variable form was found for a process
equivalent to the bifractional Brownian motion with parameters $H = K =
1/2$, arising as the solution to the one-dimensional stochastic heat
equation with an additive space--time white noise. This result was
proven mostly by martingale methods. In~\cite{Burdzy} and~\cite{NoRev},
the respective authors considered fractional Brownian motion with Hurst
parameter $1/4$. In~\cite{Burdzy}, the authors covered integrands of the
form $f(t,W_t)$, which can be applied to fBm on $[\varepsilon, \infty
)$. The authors of~\cite{NoRev} proved a change-of-variable formula
that holds on $[0,\infty)$ in the sense of marginal distributions. The
proof in~\cite{NoRev} uses Malliavin calculus; several similar methods
were used in the present paper. More recently, Nourdin, R\'eveillac and
Swanson~\cite{NoRevSwan} studied the case of fractional Brownian motion
with $H = 1/6$. In that paper, weak convergence was proven in the
Skorohod space, and the Riemann sums are based on the trapezoidal approximation.

It happens that the conditions on the process $W$ are satisfied by a
bifractional Brownian motion with parameters $H \le1/2, HK = 1/4$.
In this case, $\eta(t)=Ct$ and the process $B$ is a Brownian motion.
This includes both cases studied in~\cite{NoRev} and~\cite{Swanson},
and extends to a larger class of processes. For another example, we
consider a class of centered Gaussian processes with
twice-differentiable covariance function of the form
\[
{\mathbb E} [W_r W_t ] =r\phi\biggl(\frac tr \biggr),\qquad
t\ge r,
\]
where $\phi$ is a bounded function on $[1,\infty)$ such that
\[
\phi'(x)= \frac{\kappa} {\sqrt{x-1}} + \frac{\psi(x)}{\sqrt{x}},
\]
and $\psi$ is bounded, differentiable and $|\psi'(x)| \le C
(x-1)^{-1/2}$. This class of Gaussian processes includes the
process arising as the limit of the median of a system of independent
Brownian motions studied by Swanson in~\cite{Swanson07}. For this process,
\[
\phi(x) = \sqrt{x} \arctan\biggl(\frac{1}{\sqrt{x-1}} \biggr).
\]
It is surprising to remark that in this case $\eta(t)=Ct^2$. This is
related to the fact that the variance of the increments of $W$ on the
interval $[t-s,t]$ behaves as $C\sqrt{s}$, when $s$ is small, although
the variance of
$W(t)$ behaves as $Ct$. Our third example is another Gaussian process
studied by Swanson in~\cite{Swanson10}. This process also arises from
the empirical quantiles of a system of independent Brownian motions.
Let $B=\{ B(t), t\ge0\}$ be a Brownian motion, where $B(0)$ is a
random variable with density $f\in{\cal C}^\infty$. Given certain
growth conditions on $f$, Swanson proves there is a Gaussian process $F
= \{F(t), t\ge0\}$ with covariance given by
\[
{\mathbb E} \bigl[ F(r) F(t) \bigr] = \rho(r,t) = \frac{{\mathbb
P}
( B(r) \le q(r), B(t) \le q(t) ) -\alpha^2}{u(q(r),r) u(q(t),t)},
\]
where $\alpha\in(0,1)$ and $q(t)$ are defined by ${\mathbb P}(B(t)
\le q(t)) = \alpha$. It is shown that this family of processes
satisfies the required conditions, where $\eta(t)$ is determined by $f$
and $\alpha$.

The outline of this paper is as follows: In Section~\ref{sec2}, we introduce the
basic environment and recall some aspects of Malliavin calculus that
will be used. In Section~\ref{sec3}, a multi-dimensional version of a central
limit theorem that appears in~\cite{NoNu} is given. In Section~\ref{sec4}, the
theorem is applied to prove convergence of $\Phi_n(t)$. Section
\ref{sec5}
discusses three examples of suitable process families. Finally, Section
\ref{sec6} contains proofs of three of the longer lemmas from Section~\ref{sec4}. Most of
the notation in this paper follows that of~\cite{NoNu}.

\section{Preliminaries and notation}\label{sec2}
Let $W = \{W(t), t \ge0\}$ be a centered Gaussian process defined on
a probability space $( \Omega, {\cal F}, P )$ with continuous
covariance function
\[
{\mathbb E}\bigl[W(t) W(s) \bigr] = R(t,s).
\]
We will always assume that ${\cal F}$ is the $\sigma$-algebra generated
by $W$. Let $\cal E$ denote the set of step functions on $[0, T]$ for
$T >0$; and let $\hten$ be the Hilbert space defined as the closure of
$\cal E$ with respect to the scalar product
\[
\langle{\mathbf1}_{[0,t]}, {\mathbf1}_{[0,s]}
\rangle_\hten= R(t,s).
\]
The mapping ${\mathbf1}_{[0,t]} \mapsto W(t)$ can be extended to a
linear isometry between $\hten$ and the Gaussian space spanned by $W$.
We denote this isometry by $h \mapsto W(h)$. In this way, $\{ W(h), h
\in\hten\}$ is an isonormal Gaussian process. For integers $q \ge1$,
let $\hten^{\otimes q}$ denote the $q$th tensor product of $\hten$. We
use $\hten^{\odot q}$ to denote the symmetric tensor product.

For integers $q \ge1$, let ${\cal H}_q$ be the $q$th Wiener chaos
of $W$, that is, the closed linear subspace of $L^2(\Omega)$ generated
by the random variables $\{ H_q(W(h)), h \in\hten, \|h \|_\hten= 1 \}
$, where $H_q(x)$ is the $q$th Hermite polynomial, defined as
\[
H_q (x) = {(-1)^q}e^{{x^2}/{2}}\,\frac{d^q}{dx^q}e^{-{x^2}/{2}}.
\]
For $q \ge1$, it is known that the map
%
\begin{equation}
\label{Hmap}
I_q\bigl(h^{\otimes q}\bigr) = H_q
\bigl(W(h)\bigr)
\end{equation}
provides an isometry between the symmetric product space $\hten^{\odot
q}$ (equipped with the modified norm $\frac{1}{\sqrt{q!}}\| \cdot\|
_{\hten^{\otimes q}}$) and ${\cal H}_q$. By convention, ${\cal H}_0 =
\mathbb{R}$ and $I_0(x) = x$.

\subsection{Elements of Malliavin calculus}\label{sec2.1}
Following is a brief description of some identities that will be used
in the paper. The reader may refer to~\cite{NoNu} for a brief survey,
or to~\cite{Nualart} for detailed coverage of this topic. Let $\cal S$
be the set of all smooth and cylindrical random variables of the form
$F = g(W(\phi_1),\ldots, W(\phi_n))$, where $n \ge1$; $g\dvtx  {\mathbb
R}^n \to{\mathbb R}$ is an infinitely differentiable function with
compact support, and $\phi_i \in\hten$. The Malliavin derivative of
$F$ with respect to $W$ is the element of $L^2(\Omega, \hten)$
defined as
\[
DF = \sum_{i=1}^n \frac{\partial g}{\partial w_i}
\bigl(W(\phi_1),\ldots, W(\phi_n)\bigr)
\phi_i.
\]
In particular, $DW(h) = h$. By iteration, for any integer $q >1$, we
can define the $q$th derivative $D^qF$, which is an element of
$L^2(\Omega, \hten^{\odot q})$. For example, if $F = g(W(t))$, then
$D^2F = g''(W(t)) {\mathbf1}_{[0,t]}^{\otimes2}$.

For any integer $q \ge1$ and real number $p \ge1$, let ${\mathbb
D}^{q,p}$ denote the closure of $\cal S$ with respect to the norm $\|
\cdot\|_{{\mathbb D}^{q,p}}$ defined as
\[
\| F \|_{{\mathbb D}^{q,p}}^p = {\mathbb E} \bigl[ |F|^p
\bigr] + \sum_{i=1}^q {\mathbb E} \bigl[
\bigl\| D^iF \bigr\|_{\hten^{\otimes i}}^p \bigr].
\]

We denote by $\delta$ the Skorohod integral, which is defined as the
adjoint of the operator $D$. This operator is also referred to as the
divergence operator in~\cite{Nualart}. A~random element $u \in
L^2(\Omega, \hten)$ belongs to the domain of $\delta$,
$\operatorname{Dom} \delta$, if
and only if
\[
\bigl\llvert{\mathbb E} \bigl[ \langle DF, u \rangle_\hten\bigr]
\bigr\rrvert\le c_u \sqrt{E\bigl[F^2\bigr]}
\]
for any $F \in{\mathbb D}^{1,2}$, where $c_u$ is a constant which
depends only on $u$. If $u \in \operatorname{Dom} \delta$, then the random variable
$\delta(u) \in L^2(\Omega)$ is defined for all $F \in{\mathbb
D}^{1,2}$ by the duality relationship,
\[
{\mathbb E} \bigl[ F\delta(u) \bigr] = {\mathbb E} \bigl[ \langle DF, u
\rangle_\hten\bigr].
\]
This is sometimes called the Malliavin integration by parts formula. We
iteratively define the multiple Skorohod integral for $q \ge1$ as
$\delta(\delta^{q-1}(u))$, with $\delta^0(u) = u$. For this definition
we have
\[
{\mathbb E} \bigl[ F\delta^q(u) \bigr] = {\mathbb E} \bigl[ \bigl
\langle D^qF, u \bigr\rangle_{\hten^{\otimes q}} \bigr],
\]
where $u \in \operatorname{Dom} \delta^q$ and $F \in{\mathbb D}^{q,2}$. Moreover,
if $h \in\hten^{\odot q}$, then we have $\delta^q (h) = I_q(h)$.

For $f, g \in\hten^{\otimes p}$, the following integral multiplication
formula holds:
%
\begin{equation}
\label{multi} \delta^p (f) \delta^p(g) = \sum
_{r=0}^p r! \pmatrix{p
\cr
r} \delta^{2p-2r}(f
\otimes_r g),
\end{equation}
where $\otimes_r$ is the contraction operator; see, for example, \cite
{NoNu}, Section 2.

We will use the Meyer inequality for the Skorohod integral; see, for
example, Proposition 1.5.7 of~\cite{Nualart}. Let ${\mathbb
D}^{k,p}(\hten^{\otimes k})$ denote the corresponding Sobolev space of
$\hten^{\otimes k}$-valued random variables. Then for $p \ge1$ and
integers $k \ge q \ge1$, we have
%
\begin{equation}
\label{Meyer}\bigl\| \delta^q(u) \bigr\|_{{\mathbb D}^{k-q, p}} \le c_{k,p}
\| u \|_{{\mathbb D}^{k,p}(\hten^{\otimes q})}
\end{equation}
for all $u \in{\mathbb D}^{k,p}(\hten^{\otimes k})$ and some constant
$c_{k,p}$.

The following three results will be used in the proof of Theorem~\ref{th4.3}.
The reader may refer to~\cite{NoNu} and~\cite{Nualart} for details.
%
\begin{lemma}\label{le2.1}
Let $q \ge1$ be an integer.

\begin{longlist}[(3)]
\item[(1)] Assume $F \in{\mathbb D}^{q,2}$, $u$ is a symmetric element
of $\operatorname{Dom} \delta^q$, and $ \langle D^rF$,  $\delta^j (u)
\rangle_{\hten
^{\otimes r}} \in L^2(\Omega, \hten^{\otimes q-r-j})$ for all $0 \le
r+j \le q$. Then $ \langle D^rF, u \rangle_{\hten^{\otimes
r}} \in \operatorname{Dom} \delta^r$ and
\[
F \delta^q(u) = \sum_{r=0}^q
\pmatrix{q
\cr
r} \delta^{q-r} \bigl( \bigl\langle D^rF, u
\bigr\rangle_{\hten^{\otimes r}} \bigr).
\]

\item[(2)] Suppose that $u$ is a symmetric element of ${\mathbb
D}^{j+k,2}(\hten^{\otimes j})$. Then we have
\[
D^k \delta^j (u) = \sum_{i=0}^{j \wedge k}
\pmatrix{k
\cr
i} \pmatrix{j
\cr
i} i! \delta^{j-i} \bigl(D^{k-i}u
\bigr).
\]

\item[(3)] Let $u, v$ be symmetric functions in ${\mathbb D}^{2q,
2}(\hten^{\otimes q})$. Then
\[
{\mathbb E}\bigl[\delta^q(u) \delta^q(v) \bigr] = \sum
_{i=0}^q \pmatrix{q
\cr
i}^2 {
\mathbb E} \bigl[ \bigl\langle D^{q-i}u, D^{q-i}v \bigr
\rangle_{\hten
^{\otimes
(2q-i)}} \bigr].
\]
In particular,
\[
\bigl\llVert\delta^q (u) \bigr\rrVert^2_{L^2(\Omega)}
= {\mathbb E} \bigl[ \delta^q(u)^2 \bigr] = \sum
_{i=0}^q \pmatrix{q
\cr
i}^2 {\mathbb E}
\bigl[ \bigl\llVert D^{q-i}u\bigr\rrVert_{\hten^{\otimes(2q-i)}}^2
\bigr].
\]
\end{longlist}
\end{lemma}

\textit{Proof of} (1). This is proved in~\cite{NoNu}; see Lemma
2.1. It follows by induction from the relation $F\delta(u) = \delta(Fu)
+ \langle DF, u \rangle_\hten$; see~\cite{Nualart},
Proposition 1.3.3.\vspace*{5pt}

\textit{Proof of} (2). This follows from repeated application of
the relation $D \delta(u) = u + \delta(Du)$; see~\cite{Nualart},
Proposition 1.3.2.\vspace*{5pt}

\textit{Proof of} (3). This follows from repeated application of
the duality property; see~\cite{NoNu}, equation (2.12).

\section{A central limit theorem for multiple Skorohod integrals}\label{sec3}
Let $X = \{ X(h), h \in\hten\}$ be an isonormal Gaussian process
associated with a real-separable Hilbert space $\hten$, defined on a
probability space $(\Omega, {\cal F}, P)$. We assume that ${\cal F}$ is
generated by $X$. The purpose of this section is to prove a
multi-dimensional version of a theorem proved in~\cite{NoNu}; see
Theorem 3.1. We begin by defining the notion of stable convergence.
%
\begin{definition}\label{de3.1} 
Assume $F_n$ is a sequence of $d$-dimensional random variables defined
on a probability space $(\Omega, {\cal F}, P)$, and $F$ is a
$d$-dimensional random variable defined on $(\Omega, {\cal G}, P)$,
where ${\cal F} \subset{\cal G}$. We say that $F_n$ \textit{converges
stably} to $F$ as $n \to\infty$, if, for any continuous and bounded
function $f\dvtx  {\mathbb R}^d \to{\mathbb R}$ and bounded, ${\mathbb
R}$-valued, ${\cal F}$-measurable random variable $Z$, we have
\[
\lim_{n\to\infty} {\mathbb E} \bigl(f(F_n) Z \bigr) = {\mathbb E}
\bigl(f(F) Z \bigr).
\]
\end{definition}

%
\begin{theorem}\label{th3.2}
Let $q \ge1$ be an integer, and suppose that $F_n$ is a sequence of
random variables in $\mathbb{R}^d$ of the form $F_n = \delta^q (u_n) =
( \delta^q (u_n^1),\ldots, \delta^q(u_n^d) )$, for a
sequence of $\mathbb{R}^d$-valued symmetric functions $u_n$ in
$\mathbb
{D}^{2q, 2q}(\hten^{\otimes q})$. Suppose that the sequence $F_n$ is
bounded in $L^1(\Omega, \hten)$ and that:
\begin{longlist}[(a)]
\item[(a)]$ \langle u_n^j, \bigotimes_{\ell=1}^m (D^{a_\ell}F_n^{j_\ell})
\otimes h \rangle_{\hten^{\otimes q}}$ converges to zero in
$L^1(\Omega
)$ for all integers $1 \le j, j_\ell\le d$, all integers $1 \le a_1,\ldots, a_m, r \le q-1$ such that $a_1 + \cdots+ a_m + r = q$; and all
$h \in\hten^{\otimes r}$.
\item[(b)] For each $1 \le i,j \le d$, $ \langle u_n^i, D^qF_n^j
\rangle_{\hten
^{\otimes q}}$ converges in $L^1(\Omega, \hten)$ to a random variable
$s_{ij}$, such that the matrix $\Sigma:= ( s_{ij} )_{d
\times d}$ is nonnegative definite (i.e., $\lambda^T \Sigma\lambda
\ge
0$ for all nonzero $\lambda\in{\mathbb R}^d$).
\end{longlist}
Then $F_n$ converges stably to a random variable in ${\mathbb R}^d$
with conditional Gaussian law ${\cal N} (0, \Sigma)$ given $X$.
\end{theorem}
%
\begin{remark}\label{re3.3}
Conditions (a) and (b) mean that for $q \ge1$, some combinations of
lower-order derivative products are negligible. For example, for $q=2$,
then the following scalar products will converge to zero in
$L^1(\Omega, \hten)$:
\begin{itemize}
\item$ \langle u_n^i, h_1 \otimes h_2 \rangle_{\hten
^{\otimes2}}$ for
all $h_1, h_2 \in\hten$;
\item$ \langle u_n^i, DF_n^j \otimes h \rangle_{\hten
^{\otimes2}} $ for
all $h \in\hten$ and all $j$ (including $i=j$);
\item$ \langle u_n^i, DF_n^j \otimes DF_n^k \rangle_{\hten
^{\otimes2}}$
for all $1 \le k,j \le d$.
\end{itemize}
Only the $q$th-order derivative products converge to a nontrivial
random variable. Usually (see Section~\ref{sec6}), the term $ \langle u_n^i,
D^qF_n^j \rangle_{\hten^{\otimes q}}$ has the same asymptotic behavior
as $ \langle u_n^i, u_n^j \rangle_{\hten^{\otimes q}}$.
\end{remark}
%
\begin{remark}\label{re3.4} It suffices to impose condition (a) for $h \in{\cal
S}_0$, where ${\cal S}_0$ is a total subset of $\hten^{\otimes r}$.
\end{remark}

\begin{pf*}{Proof of Theorem~\ref{th3.2}}
As in the one-dimensional case considered in~\cite{NoNu}, we will use
the conditional characteristic function. Given any $h_1, \ldots,\break h_m \in
\hten$, we want to show that the sequence
\[
\xi_n = \bigl(F_n^1,\ldots,
F_n^d, X(h_1),\ldots, X(h_m)
\bigr)
\]
converges in distribution to a vector $ (F_\infty^1, \ldots,
F_\infty^d, X(h_1),\ldots, X(h_m) )$, where, for any vector $\lambda
\in
{\mathbb R}^d$, $F_\infty$ satisfies
%
\begin{equation}
\label{chgen} {\mathbb E} \bigl(e^{i\lambda\cdot
F_\infty} |
X(h_1),\ldots, X(h_m) \bigr) = \exp\bigl(-
\tfrac{1}{2}\lambda^T \Sigma\lambda\bigr),
\end{equation}
where $\lambda\cdot F_n = \sum_{j=1}^d \lambda_j F_n^j$ denotes the
usual scalar product in ${\mathbb R}^d$, and we use this notation to
avoid confusion with the scalar product in $\hten$.

Since $F_n$ is bounded in $L^1(\Omega, \hten)$, the sequence
$\xi_n$ is tight in the sense that for any $\varepsilon> 0$, there is
a $K>0$ such that \mbox{$P ( F_n\in[-K,K]^d ) > 1-\varepsilon$},
which follows from Chebyshev inequality.
Dropping to a subsequence if necessary, we may assume that $\xi_n$
converges in distribution to a limit\vadjust{\goodbreak} $ (F_\infty^1, \ldots,
F_\infty^d,\break X(h_1), \ldots, X(h_m) )$. Let $Y:= g (X(h_1),\ldots,
X(h_m) )$, where $g \in{\cal C}^\infty_b ({\mathbb R}^m)$, and
consider $\phi_n(\lambda) = \phi(\lambda, \xi_n):= {\mathbb
E}
(e^{i \lambda\cdot F_n} Y )$ for $\lambda\in{\mathbb R}^d$. The
convergence in law of $\xi_n$ implies that for each $1\le j\le d$,
%
\begin{equation}
\label{partialg1} \lim_{n\to\infty} \frac
{\partial\phi_n}{\partial
\lambda_j} =
\lim_{n\to\infty} i {\mathbb E} \bigl(F_n^j
e^{i
\lambda
\cdot F_n} Y \bigr) = i {\mathbb E} \bigl(F_\infty^j
e^{i \lambda
\cdot
F_\infty} Y \bigr),
\end{equation}
where convergence in distribution follows from a truncation argument
applied to~$F_n^j$.

On the other hand, using the duality property of the Skorohod integral
and the Malliavin derivative,
%
\begin{eqnarray}\label{expand2g}\quad
\frac{\partial\phi_n}{\partial\lambda_j} &=& i {\mathbb E} \bigl(\delta^q
\bigl(u_n^j\bigr) e^{i\lambda\cdot F_n} Y \bigr) = i{\mathbb
E} \bigl( \bigl\langle u_n^j,D^q
\bigl(e^{i\lambda\cdot F^n}Y \bigr) \bigr\rangle_{\hten
^{\otimes
q}} \bigr)
\nonumber\\[-1pt]
&=& i \sum_{a=0}^q \pmatrix{q
\cr
a} {
\mathbb E} \bigl( \bigl\langle u_n^j, D^a
\bigl(e^{i\lambda\cdot F_n} \bigr) \Ten D^{q-a}Y \bigr\rangle_{\hten
^{\otimes
q}}
\bigr)
\\[-1pt]
&=& i \Biggl\{ {\mathbb E} \bigl\langle
u_n^j, Y D^q e^{i\lambda
\cdot F_n} \bigr
\rangle_{\hten^{\otimes q}} + \sum_{a=0}^{q-1}
\pmatrix{q
\cr
a} {\mathbb E} \bigl\langle u_n^j,
D^a e^{i \lambda\cdot F_n} \Ten D^{q-a} Y \bigr
\rangle_{\hten^{\otimes q}} \Biggr\}.\nonumber
\end{eqnarray}

By condition (a), we have that $ \langle u_n^j, D^a e^{i\lambda
\cdot
F_n} \Ten D^{q-a}Y \rangle_{\hten^{\otimes q}}$ converges to zero in
$L^1(\Omega)$ when $a < q$, so the sum term vanishes as $n \to\infty$,
and this leaves
\begin{eqnarray*}
&&
\lim_{n \to\infty} i {\mathbb E} \bigl\langle u_n^j,
YD^qe^{i \lambda
\cdot
F_n} \bigr\rangle_{\hten^{\otimes q}} \\[-1pt]
&&\qquad=
\lim_{n\to\infty} i \sum_{k=1}^d {
\mathbb E} \bigl( i\lambda_k e^{i \lambda\cdot F_n} \bigl\langle
u_n^j, YD^q F_n^k
\bigr\rangle_{\hten^{\otimes q}} \bigr)
\\[-1pt]
&&\qquad= - \sum_{k=1}^d {\mathbb E} \bigl(
\lambda_k e^{i \lambda\cdot
F_\infty
} s_{kj} Y \bigr),
\end{eqnarray*}
because the lower-order derivatives in $D^q e^{i\lambda\cdot F_n}$ also
vanish by condition (a). Combining this with (\ref{partialg1}), we obtain
\[
i {\mathbb E} \bigl(F_{\infty}^j e^{i \lambda\cdot F_\infty} Y \bigr) =
-\sum_{k=1}^d \lambda_k {
\mathbb E} \bigl(e^{i \lambda\cdot
F_\infty
} s_{kj}Y \bigr).
\]

This leads to the PDE system,
\begin{eqnarray*}
&&
\frac{\partial}{\partial\lambda_j} {\mathbb E} \bigl(e^{i \lambda
\cdot F_\infty} | X(h_1),\ldots, X(h_m) \bigr) \\[-1pt]
&&\qquad= - \sum_{k=1}^d
\lambda_k s_{kj} {\mathbb E} \bigl(e^{i \lambda\cdot F_\infty} |
X(h_1),\ldots, X(h_m) \bigr),
\end{eqnarray*}
which has unique solution (\ref{chgen}).\vadjust{\goodbreak}
\end{pf*}

\section{Central limit theorem for the Stratonovich integral}\label{sec4}
Suppose that $W = \{W_t, t \ge0 \}$ is a centered Gaussian process, as
in Section~\ref{sec2}, that meets conditions (i) through (v), below, for any
$T>0$, where the constants $C_i$ may depend on $T$:
\begin{longlist}
\item For any $0<s\le t\le T$, there is a constant $C_1$ such that
\[
{\mathbb E} \bigl[ (W_t - W_{t-s} )^2 \bigr]
\le C_1s^{{1/2}}.
\]
\item For any $s > 0$ and $2s \le r, t \le T$ with $|t-r| \ge2s$,
\begin{eqnarray*}
&&
\bigl\llvert{\mathbb E} \bigl[ (W_t - W_{t-s})
(W_r - W_{r-s}) \bigr]\bigr\rrvert\\
&&\qquad\le
C_1s^2 |t-r|^{-\alpha}(t\wedge r -
s)^{-\beta} + s^2|t-r|^{-
{3/2}}
\end{eqnarray*}
for positive constants $\alpha, \beta, \gamma$, such that $1 <
\alpha
\le\frac{3}{2}$ and $\alpha+ \beta= \frac{3}{2}$.
%
\item For $0< t \le T$ and $0<s\le r \le T$,
\begin{eqnarray*}
&&
\bigl\llvert{\mathbb E} \bigl[ W_t (W_{r+s} -
2W_r + W_{r-s} ) \bigr]\bigr\rrvert\\
&&\qquad\le\cases{
C_2s^{{1/2}}, &\quad if $r < 2s$ or $|t-r|<2s$,
\cr
C_2s^2 \bigl( (r-s)^{-{3/2}} + |t-r|^{-{3/2}}
\bigr), &\quad if $r \ge2s$ and $|t-r|\ge2s$,}
\end{eqnarray*}
for some positive constant $C_2$.
\item For any $0<s\le t\le T-s$,
\begin{eqnarray*}
&&
\bigl\llvert{\mathbb E} \bigl[ W_t (W_{t+s} -
W_{t-s}) \bigr]\bigr\rrvert\\
&&\qquad\le\cases{ C_3s^{{1/2}},
&\quad if $t< 2s$,
\cr
C_3s (t-s)^{-{1/2}}, &\quad if $t \ge2s$,}
\end{eqnarray*}
and for each $0<s\le r\le T$,
\begin{eqnarray*}
&&
\bigl\llvert{\mathbb E} \bigl[ W_r (W_{t+s} -
W_{t-s}) \bigr]\bigr\rrvert\\
&&\qquad\le\cases{ C_3s^{{1/2}},
&\quad if $t< 2s$ or $|t-r|<2s$,
\cr
C_3s (t-s)^{-{1/2}}+
C_3s|t-r|^{-{1/2}}, &\quad if $t \ge2s$ and $|t-r|\ge2s$,}
\end{eqnarray*}
for some positive constant $C_3$. In addition, for $t > 2s$,
\[
\bigl\llvert{\mathbb E} \bigl[ W_s (W_t -
W_{t-s}) \bigr]\bigr\rrvert\le C_3s^{{1/2}+\gamma}(t-2s)^{-\gamma}
\]
for some $\gamma> 0$.
\item Consider a uniform partition of $[0, \infty)$ with increment
length $1/n$. Define for integers $j,k \ge0$ and $n \ge1$,
\[
\beta_n(j,k) = {\mathbb E} \bigl[ (W_{({j+1})/{n}} -
W_{
{j}/{n}} ) (W_{({k+1})/{n}} - W_{{k/n}} ) \bigr].
\]
Next, define
\begin{eqnarray*}
\eta_n^+(t) &=&\sum_{j,k = 1}^{ \lfloor{nt/2} \rfloor}
\beta_n(2j-1,2k-1)^2 + \beta_n(2j-2,2k-2)^2;
\\
\eta_n^-(t) &=&\sum_{j,k = 1}^{ \lfloor{nt/2} \rfloor}
\beta_n(2j-2,2k-1)^2 + \beta_n(2j-1,2k-2)^2.
\end{eqnarray*}
Then for each $t\ge0$,
\[
\lim_{n \to\infty}\eta_n^+(t) = \eta^+(t) \quad\mbox{and}\quad
\lim_{n
\to\infty}\eta_n^-(t) = \eta^-(t)
\]
both exist, where $\eta^+(t), \eta^-(t)$ are nonnegative and
nondecreasing functions.
\end{longlist}

Consider a real-valued function $f \in{\cal C}^9 ({\mathbb R})$, such
that $f$ and all its derivatives up to order 9 have at most exponential
growth, that is,
\[
\bigl\llvert f^{(k)}(x)\bigr\rrvert< K_1 \exp
\bigl(K_2|x|^\alpha\bigr),\qquad x \in{\mathbb R}, \alpha< 2,
\]
for $k = 0,\ldots, 9$, and positive constants $K_1, K_2$. We will
refer to this as condition~$(0)$.

In the following, the term $C$ represents a generic positive constant,
which may change from line to line. The constant $C$ may depend on $T$
and the constants in conditions (0) and (i)--(v), listed above.

The results of the next lemma follow from conditions (i) and (ii).
%
\begin{lemma}\label{le4.1}
Using the notation described above, for integers $0 \le a
< b$ and integers $r, n \ge1$, we have the estimate
\[
\sum_{j,k = a}^b \bigl|\beta_n(j,k)\bigr|^r
\le C(b-a+1)n^{-{r/2}}.
\]
\end{lemma}
\begin{pf} Suppose first that $r=1$. Let $I =\{ (j,k)\dvtx  a \le j,k\le b,
|k-j|\ge2, j\wedge k \ge2\}$ and $J = \{ (j,k)\dvtx  a \le j,k\le b, (j,k)
\notin I\}$. Consider the decomposition
\[
\sum_{j,k=a}^b \bigl\llvert
\beta_n(j,k)\bigr\rrvert= \sum_{(j,k)\in I}
\bigl\llvert\beta_n(j,k)\bigr\rrvert+ \sum
_{(j,k) \in J} \bigl\llvert\beta_n(j,k)\bigr\rrvert.
\]
Then by condition (ii), the first sum is bounded by
\[
\sum_{(j,k)\in I} n^{-{1/2}}|j-k|^{-\alpha} \le
Cn^{-{1/2}} (b-a+1),
\]
and the second sum, using condition (i) and Cauchy--Schwarz, is bounded
by $Cn^{-{1/2}}(b-a+1)$.
For the case $r > 1$, condition (i) implies $\llvert\beta_n(j,k)\rrvert
\le C_1n^{-{1/2}}$ for all $j,k$. It follows that we can write
\begin{eqnarray*}
\sum_{j,k=a}^b \bigl\llvert
\beta_n(j,k)\bigr\rrvert^r &\le& C_1n^{-({r-1})/{2}}
\sum_{j,k=a}^b \bigl\llvert
\beta_n(j,k)\bigr\rrvert\\
&\le& C(b-a+1)n^{-{r/2}}.
\end{eqnarray*}
\upqed
\end{pf}

%
\begin{corollary}\label{co4.2}
Using the notation of Lemma~\ref{le4.1}, for each integer $r
\ge1$,
\begin{eqnarray*}
&&\sum_{j,k=1}^{ \lfloor{nt/2} \rfloor} \bigl(\bigl\llvert
\beta_n(2j-1,2k-1)\bigr\rrvert^r + \bigl\llvert
\beta_n(2j-1,2k-2)\bigr\rrvert^r \\
&&\hspace*{26.5pt}{}+ \bigl\llvert
\beta_n(2j-2,2k-1)\bigr\rrvert^r + \bigl\llvert
\beta_n(2j-2, 2k-2)\bigr\rrvert^r \bigr)
\\
&&\qquad\le C \biggl\lfloor\frac{nt}{2} \biggr\rfloor n^{-{r/2}}.
\end{eqnarray*}
\end{corollary}
\begin{pf} Note that
\begin{eqnarray*}
&&\sum_{j,k=1}^{ \lfloor{nt/2} \rfloor} \bigl(\bigl\llvert
\beta_n(2j-1,2k-1)\bigr\rrvert^r + \bigl\llvert
\beta_n(2j-1,2k-2)\bigr\rrvert^r \\
&&\hspace*{27pt}{}+ \bigl\llvert
\beta_n(2j-2,2k-1)\bigr\rrvert^r + \bigl\llvert
\beta_n(2j-2, 2k-2)\bigr\rrvert^r \bigr)
\\
&&\qquad = \sum_{j,k=0}^{2 \lfloor{nt}/{2} \rfloor-1} \bigl\llvert
\beta_n(j,k)\bigr\rrvert^r.
\end{eqnarray*}
\upqed
\end{pf}

Consider a uniform partition of $[0, \infty) $ with increment length
$1/n$. The Stratonovich midpoint integral of $f'(W)$ will be defined as
the limit in distribution of the sequence (see~\cite{Swanson})
%
\begin{equation}
\label{Phin} \Phi_n(t):= \sum
_{j=1}^{ \lfloor{nt/2} \rfloor} f'(W_{({2j-1})/{n}})
(W_{{2j}/{n}} - W_{({2j-2})/{n}} ).
\end{equation}

We introduce the following notation, as used in~\cite{NoNu}:
$\varepsilon_t:= {\mathbf1}_{[0,t]}$; and $\partial_{j/n}:=
{\mathbf
1}_{ [{j/n},({j+1})/{n} ]} $.\vspace*{1pt}

The following is the major result of this section.
%
\begin{theorem}\label{th4.3} 
Let $f$ be a real function satisfying condition (0), and let $W = \{
W_t, t \ge0\}$ be a Gaussian process satisfying conditions \textup{(i)} through
\textup{(v)}. Then
\[
\bigl( W_t, \Phi_n(t) \bigr) \law\biggl(
W_t, f(W_t) - f(W_0) - \frac
{1}{2}\int
_0^t f''(W_s)
\,dB_s \biggr)
\]
as $n \to\infty$ in the Skorohod space $ ({\mathbb D}[0, \infty
) )^2$, where $\eta(t) = \eta^+(t) - \eta^-(t)$ for the functions
defined in condition \textup{(v)}; and $B = \{B_t, t\ge0\}$ is scaled Brownian
motion, independent of $W$, and with variance ${\mathbb E}
[B_t^2 ] = 2\eta(t)$.
\end{theorem}

The rest of this section consists of the proof of Theorem~\ref{th4.3},
and is presented in a series of lemmas. The proofs of Lemmas
\ref{le4.4},~\ref{le4.5} and~\ref{le4.9}, which are rather technical, are
deferred to Section~\ref{sec6}. We begin with an expansion of $f(W_t)$,
following the methodology used in \cite {Swanson}. Consider the
telescoping series
\begin{eqnarray*}
f(W_t) &=& f(W_0) + \sum_{j=1}^{ \lfloor{nt/2} \rfloor}
\bigl[f(W_{{2j}/{n}}) - f(W_{({2j-2})/{n}}) \bigr] \\
&&{}+ f(W_t) -
f(W_{({2}/{n})\lfloor{nt}/{2} \rfloor}),
\end{eqnarray*}
where the sum is zero by convention if $ \lfloor\frac
{nt}{2}
\rfloor=0$. Using a Taylor series expansion of order 2, we
obtain
\begin{eqnarray*}
\Phi_n(t) &=& f(W_t) - f(W_0) \\
&&{}-
\frac{1}{2}\sum_{j=1}^{ \lfloor{nt/2} \rfloor}
f''(W_{
({2j-1})/{n}}) \bigl( \Delta
W_{{2j}/{n}}^2 - \Delta W_{
({2j-1})/{n}}^2 \bigr)
\\
&&{} - \sum_{j=1}^{ \lfloor{nt/2} \rfloor} R_0(
W_{{2j}/{n}}) + \sum_{j=1}^{ \lfloor{nt/2} \rfloor}
R_1 ( W_{({2j-2})/{n}}) \\
&&{}- \bigl( f(W_t) -
f(W_{({2}/{n})\lfloor
{nt}/{2} \rfloor}) \bigr),
\end{eqnarray*}
where $R_0, R_1$ represent the third-order remainder terms in the
Taylor expansion, and can be expressed in integral form as
%
\begin{equation}
\label{R0}R_0(W_{{2j}/{n}}) = \frac{1}{2}
\int_{W_{
({2j-1})/{n}}}^{W_{{2j}/{n}}} (W_{{2j}/{n}} -
u)^2 f^{(3)}(u) \,du
\end{equation}
and
%
\begin{equation}
\label{R1}R_1(W_{({2j-2})/{n}}) = -\frac{1}{2}
\int_{W_{
({2j-2})/{n}}}^{W_{({2j-1})/{n}}} (W_{({2j-2})/{n}} -
u)^2 f^{(3)}(u) \,du.
\end{equation}
By condition (0) we have for any $T > 0$ that
\[
\lim_{n \to\infty} {\mathbb E} \sup_{0\le t \le T} \bigl\llvert
f(W_t) - f(W_{({2}/{n}) \lfloor{nt}/{2} \rfloor}) \bigr\rrvert= 0,
\]
so this term vanishes uniformly on compacts in probability (ucp), and
may be neglected. Therefore, it is sufficient to work with the term
%
\begin{equation}
\label{DefDelta}\Delta_n(t):= f(W_t) -
f(W_0) - \tfrac{1}{2}\Psi_n(t) + R_n(t),
\end{equation}
where
\[
\Psi_n(t) = \sum_{j=1}^{ \lfloor{nt/2} \rfloor}
f''(W_{({2j-1})/{n}}) \bigl( \Delta
W_{{2j}/{n}}^2 - \Delta W_{({2j-1})/{n}}^2 \bigr)
\]
and
\[
R_n(t) = \sum_{j=1}^{ \lfloor{nt/2} \rfloor}
\bigl(R_1(W_{({2j-2})/{n}}) - R_0(W_{
{2j}/{n}})
\bigr).
\]
We will first decompose the term $\Psi_n(t)$, using a Skorohod integral
representation. Using (\ref{Hmap}) and the second Hermite polynomial, one
can write $\Delta W^2(h) = H_2 (\Delta W(h) ) + 1 = \delta^2
(h^{\otimes2}) +1$ for any $h \in\hten$ with $\| h \|_\hten= 1$. It
follows that
\[
\Psi_n(t) = \sum_{j=1}^{ \lfloor{nt/2} \rfloor}
f''(W_{({2j-1})/{n}}) \delta^2 \bigl(
\partial_{({2j-1})/{n}}^{\otimes2} - \partial_{
({2j-2})/{n}}^{\otimes2}
\bigr).
\]
From Lemma~\ref{le2.1}, we have for random variables $u, F$,
\[
F \delta^2(u) = \delta^2 (Fu) + 2\delta\bigl( \langle
DF,u \rangle_\hten\bigr) + \bigl\langle D^2F, u \bigr
\rangle_{\hten^{\otimes2}},
\]
so we can write
\begin{eqnarray*}
\Psi_n(t)& = &\sum_{j=1}^{\lfloor{nt}/{2}\rfloor}
\delta^2 \bigl( f''(W_{({2j-1})/{n}})
\bigl(\partial_{({2j-1})/{n}}^{\otimes2} - \partial_{
({2j-2})/{n}}^{\otimes2}
\bigr) \bigr)
\\
&&{} +\sum_{j=1}^{\lfloor{nt}/{2}\rfloor} 2 \delta
\bigl(f^{(3)}(W_{({2j-1})/{n}}) \bigl\langle\varepsilon_{({2j-1})/{n}},
\partial_{({2j-1})/{n}}^{\otimes2}-\partial_{
({2j-2})/{n}}^{\otimes2}
\bigr\rangle_\hten\bigr)
\\
&&{} +\sum_{j=1}^{\lfloor{nt}/{2}\rfloor} f^{(4)}(W_{
({2j-1})/{n}})
\bigl( \langle\varepsilon_{({2j-1})/{n}}, \partial_{
({2j-1})/{n}}
\rangle_\hten^2\\
&&\hspace*{112pt}{} - \langle\varepsilon_{({2j-1})/{n}},
\partial_{({2j-2})/{n}} \rangle_\hten^2 \bigr)
\\
:\!&=& F_n(t) + B_n(t) + C_n(t).
\end{eqnarray*}
Hence, we have $\Delta_n(t) = f(W_t) - f(W_0) - \frac{1}{2} (
F_n(t) + B_n(t) +C_n(t) ) + R_n(t)$. In the next two lemmas, we
show that the terms $B_n(t), C_n(t)$ and $R_n(t)$ converge to zero in
probability as $n \to\infty$. The proofs of these lemmas are deferred
to Section~\ref{sec6}.

%
\begin{lemma}\label{le4.4}
Let $0 \le r < t \le T$. Using the notation defined above,
\[
{\mathbb E} \bigl[ \bigl(R_n(t) - R_n(r)
\bigr)^2 \bigr] \le C \biggl( \biggl\lfloor\frac{nt}{2} \biggr
\rfloor- \biggl\lfloor\frac{nr}{2} \biggr\rfloor\biggr)n^{-{3/2}}
\]
for some positive constant $C$, which may depend on $T$. It follows
that for any $0 \le t \le T$, $R_n(t)$ converges to zero in probability
as $n \to\infty$.
\end{lemma}

%
\begin{lemma}\label{le4.5}
Let $0 \le r < t \le T$. Using the above notation, there
exist constants $C_B, C_C$ such that
\[
{\mathbb E} \bigl[ \bigl(B_n(t) - B_n(r)
\bigr)^2 \bigr] \le C_B \biggl( \biggl\lfloor
\frac{nt}{2} \biggr\rfloor- \biggl\lfloor\frac
{nr}{2} \biggr\rfloor
\biggr)n^{-{3/2}}
\]
and
\[
{\mathbb E} \bigl[ \bigl(C_n(t) - C_n(r)
\bigr)^2 \bigr] \le C_C \biggl( \biggl\lfloor
\frac{nt}{2} \biggr\rfloor- \biggl\lfloor\frac
{nr}{2} \biggr\rfloor
\biggr)n^{-{3/2}}.
\]
It follows that for any $0 \le t \le T$, $B_n(t)$ and $C_n(t)$ converge
to zero in probability as $n \to\infty$.
\end{lemma}
%
\begin{corollary}\label{co4.6}
Let $Z_n(t):= R_n(t) - \frac{1}{2}B_n(t) - \frac
{1}{2}C_n(t)$. Then given $0 \le t_1 < t <t_2 \le T$, there exists a
positive constant $C$ such that
\[
{\mathbb E} \bigl[ \bigl|Z_n(t) - Z_n(t_1)\bigr| \bigl|Z_n(t_2)
- Z_n(t)\bigr| \bigr] \le C(t_2 - t_1)^{{3/2}}.
\]
\end{corollary}
\begin{pf} By Lemmas~\ref{le4.4} and~\ref{le4.5},
\begin{eqnarray*}
{\mathbb E} \bigl[ \bigl(Z_n(t_2) - Z_n(t_1)
\bigr)^2 \bigr] &\le& 3{\mathbb E} \bigl[ \bigl(R_n(t_2)
- R_n(t_1) \bigr)^2 \bigr]\\
&&{} + 2 {\mathbb E}
\bigl[ \bigl(B_n(t_2) - B_n(t_1)
\bigr)^2 \bigr]
\\
&&{} + 2{\mathbb E} \bigl[ \bigl(C_n(t_2) -
C_n(t_1) \bigr)^2 \bigr]
\\
&\le& C \biggl( \biggl\lfloor\frac{nt_2}{2} \biggr\rfloor- \biggl\lfloor
\frac{nt_1}{2} \biggr\rfloor\biggr)n^{-{3/2}}.
\end{eqnarray*}
Then by the Cauchy--Schwarz inequality,
\begin{eqnarray*}
&&
{\mathbb E} \bigl[ \bigl|Z_n(t) - Z_n(t_1)\bigr| \bigl|Z_n(t_2)
- Z_n(t)\bigr| \bigr]\\
&&\qquad\le\bigl( {\mathbb E} \bigl[ \bigl(Z_n(t)
- Z_n(t_1) \bigr)^2 \bigr] {\mathbb E}
\bigl[ \bigl(Z_n(t) - Z_n(t_1)
\bigr)^2 \bigr] \bigr)^{{1/2}}
\\
&&\qquad\le C \biggl( \biggl\lfloor\frac{nt_2}{2} \biggr\rfloor- \biggl\lfloor
\frac{nt_1}{2} \biggr\rfloor\biggr)^{{3/2}}n^{-{3/2}}.
\end{eqnarray*}
This estimate implies the required bound $C(t_2 - t_1)^{{3/2}}$;
see, for example,~\cite{Billingsley}, page~156.
\end{pf}

Next, we will develop a comparable estimate for differences of the form
$F_n(t) - F_n(r)$. In order to prove this estimate, we need a technical
lemma which will be used here and also in Section~\ref{sec6}.

%
\begin{lemma}\label{le4.7}
Suppose $a, b$ are nonnegative integers such that $a + b
\le9$. For fixed $T > 0$ and interval $[t_1, t_2] \subset[0,T]$, let
\[
g_a = \sum_{\ell= \lfloor{nt_1}/{2} \rfloor+1}^{\lfloor
{nt_2}/{2} \rfloor}
f^{(a)}(W_{(2\ell-1)/{n}}) \bigl(\partial
_{(2\ell-1)/{n}}^{\otimes2}
- \partial_{(2\ell-2)/{n}}^{\otimes
2} \bigr).
\]
Then we have for $1 \le p < \infty$
\[
{\mathbb E} \bigl[ \bigl\| D^b g_a \bigr\|^p_{\hten^{\otimes2+b}}
\bigr] \le C \biggl( \biggl\lfloor\frac{nt_2}{2} \biggr\rfloor- \biggl
\lfloor
\frac
{nt_1}{2} \biggr\rfloor\biggr)^{{p/2}}n^{-{p/2}}.
\]
\end{lemma}
\begin{pf} We may assume $t_1 = 0$ with $t_2 \le T$. For each $b$ we
can write
\begin{eqnarray*}
&&{\mathbb E} \bigl[ \bigl(\bigl\| D^b g_a
\bigr\|^2_{\hten
^{\otimes
2+b}} \bigr)^{{p/2}} \bigr]
\\
&&\qquad ={\mathbb E} \Biggl[ \Biggl(\sum_{\ell, m = 1}^{ \lfloor
{nt_2}/{2} \rfloor} f^{(a+b)}(W_{(2\ell
-1)/{n}})f^{(a+b)}(W_{(2m-1)/{n}}) \\
&&\qquad\hspace*{56.2pt}{}\times
\bigl\langle\varepsilon_{(2\ell-1)/{n}}^{\otimes b},
\varepsilon_{(2m-1)/{n}}^{\otimes b} \bigr\rangle_{\hten^{\otimes b}}\\
&&\qquad\hspace*{56.2pt}{}\times
\bigl\langle\partial_{(2\ell-1)/{n}}^{\otimes2} -
\partial_{(2\ell-2)/{n}}^{\otimes2}, \partial_{(2m -1)/{n}}^{\otimes2}
- \partial_{(2m -2)/{n}}^{\otimes2} \bigr\rangle_{\hten^{\otimes2}}
\Biggr)^{{p/2}} \Biggr]
\\
&&\qquad \le{\mathbb E} \Bigl[\sup_{0 \le s \le t} \bigl\llvert f^{(a+b)}(W_s)
\bigr\rrvert^p \Bigr] \Bigl(\sup_{\ell,m} \bigl\llvert\langle
\varepsilon_{(2\ell-1)/{n}}, \varepsilon_{(2m-1)/{n}} \rangle_\hten
\bigr\rrvert^b \Bigr)^{{p/2}} \\
&&\quad\qquad{}\times\Biggl( \sum
_{\ell, m = 1}^{ \lfloor{nt_2}/{2} \rfloor
}\bigl\llvert\bigl\langle
\partial_{(2\ell-1)/{n}}^{\otimes2} - \partial_{(2\ell-2)/{n}}^{\otimes2}, \partial_{(2m -1)/{n}}^{\otimes2} -
\partial_{(2m -2)/{n}}^{\otimes2} \bigr\rangle_{\hten^{\otimes
2}}\bigr\rrvert\Biggr)^{{p/2}}.
\end{eqnarray*}
Recall that condition (0) holds for $f$ and its first 9 derivatives, so
the first two terms are bounded. For the last term, note that by
Corollary~\ref{co4.2} with $r=2$,
\begin{eqnarray*}
&&\sum_{\ell, m = 1}^{ \lfloor{nt_2}/{2}
\rfloor}\bigl\llvert\bigl\langle
\partial_{(2\ell-1)/{n}}^{\otimes2} - \partial_{(2\ell
-2)/{n}}^{\otimes2}, \partial_{(2m -1)/{n}}^{\otimes2} -
\partial_{(2m -2)/{n}}^{\otimes2} \bigr\rangle_{\hten^{\otimes
2}}\bigr\rrvert
\\
&&\qquad = \sum_{\ell, m=1}^{ \lfloor{nt_2}/{2}
\rfloor
}\bigl\llvert
\beta_n(2\ell-1,2m-1)^2 -\beta_n(2
\ell-1,2m-2)^2 \\
&&\qquad\quad\hspace*{27.6pt}{}- \beta_n(2\ell-2,2m-1)^2 +
\beta_n(2\ell-2,2m-2)^2\bigr\rrvert
\\
&&\qquad \le C \biggl\lfloor\frac{nt_2}{2} \biggr\rfloor n^{-1}.
\end{eqnarray*}
\upqed
\end{pf}

%
\begin{lemma}\label{le4.8}
For $0 \le s < t \le T$, write
\[
F_n(t) - F_n(s) = \sum_{j=\lfloor{ns}/{2} \rfloor+1}^{
\lfloor{nt/2} \rfloor} \delta^2 \bigl( f''(W_{({2j-1})/{n}})
\bigl(\partial_{ ({2j-1})/{2}}^{\otimes2} -
\partial_{({2j-2})/{n}}^{\otimes2} \bigr) \bigr).
\]
Then given $0\le t_1 < t < t_2 \le T$, there exists a positive constant
$C$ such that
%
\begin{equation}
\label{tight} {\mathbb E} \bigl[ \bigl|F_n(t) - F_n(t_1)
\bigr|^2 \bigl|F_n(t_2) - F_n(t)
\bigr|^2 \bigr] \le C (t_2 - t_1)^2.
\end{equation}
\end{lemma}
\begin{pf}
First, for each $n \ge1$, we want to show that there is a $C$ such that
\[
{\mathbb E} \bigl[ \bigl( F_n(t_2) -
F_n(t_1) \bigr)^4 \bigr] \le C \biggl(
\biggl\lfloor\frac{nt_2}{2} \biggr\rfloor- \biggl\lfloor\frac
{nt_1}{2}
\biggr\rfloor\biggr)^2n^{-2}.
\]

By the Meyer inequality (\ref{Meyer}) there exists a constant $c_{2,4}$
such that
\[
{\mathbb E} \bigl\llvert\bigl(\delta^2(u_n)
\bigr)^4 \bigr\rrvert\le c_{2,4} \| u_n
\|^4_{{\mathbb D}^{2,4}(\hten^{\otimes2})},
\]
where in this case
\[
u_n = \sum_{j=\lfloor\frac{nt_1}{2} \rfloor+1}^{\lfloor
{nt_2}/{2} \rfloor}
f''(W_{({2j-1})/{n}}) \bigl( \partial_{(2j-1)/{n}}^{\otimes2}
- \partial_{({2j-2})/{n}}^{\otimes2} \bigr)
\]
and
\[
\|u_n\|_{{\mathbb D}^{2,4}(\hten^{\otimes2})}^4 = {\mathbb E}\|
u_n \|_{\hten^{\otimes2}}^4 +{\mathbb E}\| Du_n
\|_{\hten^{\otimes
3}}^4+{\mathbb E}\bigl\| D^2u_n
\bigr\|_{\hten^{\otimes4}}^4.
\]
From Lemma~\ref{le4.7} we have ${\mathbb E}\| u_n\|^4_{\hten^{\otimes2}},
{\mathbb E}\|Du_n\|^4_{\hten^{\otimes3}}, {\mathbb E}\|D^2u_n \|_{\hten
^{\otimes4}}^4 \le C ( \lfloor\frac{nt_2}{2}
\rfloor- \lfloor\frac{nt_1}{2} \rfloor)^2 n^{-2}$,
and so it follows that
\[
{\mathbb E} \bigl[ \bigl( \delta^2 (u_n)
\bigr)^4 \bigr] \le C \biggl( \biggl\lfloor\frac{nt_2}{2} \biggr
\rfloor- \biggl\lfloor\frac{nt_1}{2} \biggr\rfloor\biggr)^2
n^{-2}.
\]

From this result, given $0 \le t_1 < t < t_2$, it follows from the H\"
{o}lder inequality that
\begin{eqnarray*}
&&
{\mathbb E} \bigl[ \bigl|F_n(t) - F_n(t_1)
\bigr|^2 \bigl|F_n(t_2) - F_n(t)\bigr|^2
\bigr] \\
&&\qquad\le \bigl( {\mathbb E} \bigl[ \bigl|F_n(t) - F_n(t_1)\bigr|^4
\bigr] \bigr)^{{1/2}} \bigl( {\mathbb E} \bigl[ \bigl|F_n(t_2)
- F_n(t)\bigr|^4 \bigr] \bigr)^{{1/2}}
\\
&&\qquad\le C \biggl( \biggl\lfloor\frac{nt_2}{2} \biggr\rfloor-
\biggl\lfloor \frac{nt_1}{2} \biggr\rfloor\biggr)^2n^{-2}.
\end{eqnarray*}
As in Corollary~\ref{co4.6}, this implies the required bound $C(t_2 - t_1)^2$.
\end{pf}

By Corollary~\ref{co4.6} and Lemma~\ref{le4.8}, it follows that $\Delta_n(t) = f(W_t) -
f(W_0) -\frac{1}{2}F_n(t) + Z_n(t)$ is tight, since both sequential
parts $F_n(t), Z_n(t)$ are tight. Further, we have that $Z_n(t)$ tends
to zero in probability, and $F_n(t)$ is in a form suitable for Theorem
\ref{th3.2}. In the next lemma, we show that the conditions of Theorem
\ref{th3.2} are
satisfied by $F_n(t)$ evaluated at a finite set of points.

%
\begin{lemma}\label{le4.9}
Fix $0=t_0 < t_1 <t_2 < \cdots< t_d$. Set $F_n^i = F_n(t_i) -
F_n(t_{i-1})$ for $i = 1,\ldots, d$, and let $F_n = (F_n^i,\ldots,
F_n^d)$. Then under conditions (0), and~\textup{(i)--(v)}, $F_n$ satisfies\vadjust{\goodbreak}
conditions \textup{(a)} and \textup{(b)} of Theorem~\ref{th3.2}, and so given $W$, $F_n$
converges stably as $n \to\infty$ to a random variable $\xi=
(\xi_1,\ldots, \xi_d)$ with distribution ${\cal N}(0, \Sigma)$, where
$\Sigma$ is a diagonal $d \times d$ matrix with the following entries:
\[
s_i^2 = \int_{t_{i-1}}^{t_i}
f''(W_s)^2\eta(ds),
\]
where $\eta(t) = \eta^+(n) - \eta^-(t)$ is as defined in condition \textup{(v)}.
\end{lemma}
%
\begin{remark}\label{re4.10} As we will see later, $\eta(t)$ is continuous,
nonnegative and nondecreasing.
\end{remark}

It follows from the structure of $\Sigma$ that, given $W$, $F_n$
converges stably to a $d$-dimensional vector with conditionally
independent components of the form
\[
F_\infty^i = \zeta_i \sqrt{\int
_{t_{i-1}}^{t_i} f''(W_s)^2
\eta(ds)},
\]
where each $\zeta_i \sim{\cal N}(0,1)$. Thus, we may conclude that for
each $i$,
\[
F_n^i \law\int_{t_{i-1}}^{t_i}
f''(W_s) \,dB_s
\]
for a scaled Brownian motion $B = \{B_t, t\ge0\}$ that is independent
of $W_t$, with ${\mathbb E} [ B_t^2 ] = \eta(t)$.

\begin{pf*}{Proof of Theorem~\ref{th4.3}}
To prove Theorem~\ref{th4.3}, it is enough to show that for any finite set of
times $0 = t_0<t_1 < t_2 < \cdots< t_d$, we have
\begin{eqnarray*}
&&
\bigl(\Delta_n(t_1), \Delta_n(t_2)
- \Delta_n(t_1),\ldots, \Delta_n(t_d)-
\Delta_n(t_{d-1}) \bigr) \\
&&\qquad\law\bigl(\Delta(t_1),
\Delta(t_2) - \Delta(t_1),\ldots, \Delta(t_d)-
\Delta(t_{d-1}) \bigr)
\end{eqnarray*}
as $n \to\infty$; and that $\Delta_n(t)$ satisfies the tightness condition
%
\begin{equation}
\label{tightness} {\mathbb E} \bigl[ \bigl\llvert\Delta_n(t) -
\Delta_n(t_1)\bigr\rrvert^\gamma\bigl\llvert
\Delta_n(t_2) - \Delta_n(t)\bigr
\rrvert^\gamma\bigr] \le C(t_2-t_1)^\alpha
\end{equation}
for $0\le t_1 < t < t_2 < \infty$, $\gamma> 0$ and $\alpha> 1$.

For $\Delta_n(t) = f(W_t) - f(W_0) - \frac{1}{2}F_n(t) + Z_n(t)$, we
have shown in Lemmas~\ref{le4.4} and~\ref{le4.5} that
\[
Z_n(t) = R_n(t) -\tfrac{1}{2} \bigl(
B_n(t) + C_n(t) \bigr) \stackrel{\cal P} {
\longrightarrow} 0
\]
for each $0 \le t\le T$, and hence $Z_n(t_i) - Z_n(t_{i-1}) \stackrel
{\cal P}{\longrightarrow} 0$ for each $t_i$, $1 \le i \le d$. By Lemma
\ref{le4.9}, the pair $(W, F_n)$ converges in law to $(W, F_\infty)$, where
$F_\infty$ is a $d$-dimensional random vector with conditional Gaussian
law and whose covariance matrix is diagonal with entries
\[
s_i^2 = \int_{t_{i-1}}^{t_i}
f''(W_s)^2 \eta(ds).\vadjust{\goodbreak}
\]
It follows that, conditioned on $W$, each component may be expressed as
an independent Gaussian random variable, equivalent in law to
\[
\int_{t_{i-1}}^{t_i} f''(W_s)
\,dB_s,
\]
where $B=\{B_t, t \ge0\}$ is a scaled Brownian motion independent of
$W$ with ${\mathbb E} [B_t^2 ] = \eta(t)$. Finally, tightness
follows from Lemma~\ref{le4.8} and Corollary~\ref{co4.6}. Theorem~\ref{th4.3} is proved.
\end{pf*}

\section{Examples}\label{sec5}

\subsection{Bifractional Brownian motion}\label{sec5.1}
The bifractional Brownian motion is a generalization of fractional
Brownian motion, first introduced by Houdr\'e and Villa~\cite{Houdre}.
It is defined as a centered Gaussian process $B^{H,K} = \{B^{H,K}(t),\break t
\ge0\}$, with covariance defined by
\[
{\mathbb E} \bigl[B_t^{H,K} B^{H,K}_s
\bigr] = \frac{1}{2^K} \bigl( t^{2H} + s^{2H}
\bigr)^K + \frac{1}{2^K}|t-s|^{2HK},
\]
where $H \in(0,1)$, $K\in(0,1]$. (Note that the case $K=1$ corresponds
to fractional Brownian motion with Hurst parameter $H$.) The reader
may refer to~\cite{Lei} and its references for further discussion of
properties.

In this section, we show that the results of Section~\ref{sec4} are valid for
bifractional Brownian motion with parameter values $H$, $K$ such that
$H \le1/2$ and $2HK = 1/2$. In particular, this includes the endpoint
cases $H = 1/4$, $K=1$ studied in~\cite{NoRev}, and $H=1/2, K=1/2$
studied in~\cite{Swanson}.
%
\begin{proposition}\label{pr5.1}
Let $ \{B^{H,K}_t, t \ge0 \} $ denote a bifractional Brownian motion.
The covariance conditions \textup{(i)--(iv)} of Section~\ref{sec4} are
satisfied for values of $0 < H \le1/2 $ and $0 < K \le1$ such that $2HK
= 1/2$.
\end{proposition}
\begin{pf}
\textit{Condition} (i).
\begin{eqnarray*}
&&
{\mathbb E} \bigl[ \bigl(B_t^{H,K} - B_{t-s}^{H,K}
\bigr)^2 \bigr] \\
&&\qquad= t^{2HK} + \frac{2}{2^K}(t-s)^{2HK}
- \bigl[t^{2H} + (t-s)^{2H} \bigr]^K -
\frac{2}{2^K}s^{2HK}
\\
&&\qquad\le \biggl[ \biggl\llvert\sqrt{t} - \frac{1}{2^K} \bigl(
t^{2H} +(t-s)^{2H} \bigr)^K\biggr\rrvert\\
&&\qquad\hspace*{15.5pt}{} +
\biggl\llvert\sqrt{t-s} - \frac{1}{2^K} \bigl( t^{2H}
+(t-s)^{2H} \bigr)^K\biggr\rrvert+ \frac{1}{2^K}s^{{1/2}}
\biggr]
\\
&&\qquad\le Cs^{{1/2}},
\end{eqnarray*}
where we used the inequality $a^m - b^m \le(a-b)^m$ for $a > b > 0$
and $m<1$.

\textit{Condition} (ii).
%
\begin{eqnarray*}
&&{\mathbb E} \bigl[ \bigl(B^{H,K}_t - B^{H,K}_{t-s}
\bigr) \bigl(B^{H,K}_r - B^{H,K}_{r-s}
\bigr) \bigr]
\\
&&\qquad = \frac{1}{2^K} \bigl( \bigl[t^{2H} + r^{2H}
\bigr]^K - \bigl[t^{2H} + (r-s)^{2H}
\bigr]^K\\
&&\qquad\quad\hspace*{18.5pt}{} - \bigl[(t-s)^{2H} + r^{2H}
\bigr]^K + \bigl[(t-s)^{2H} + (r-s)^{2H}
\bigr]^K \bigr)
\\
&&\qquad\quad{} +\frac{1}{2^K} \bigl(|t-r+s|^{2HK} - 2 |t-r|^{2HK} +
|t-r-s|^{2HK} \bigr).
\end{eqnarray*}

This can be interpreted as the sum of a position term, $\frac
{1}{2^K}\varphi(t, r, s)$, and a distance term, $\frac{1}{2^K}\psi(t-r,
s)$, where
\begin{eqnarray*}
\varphi(t, r, s) &=& \bigl[t^{2H} + r^{2H}
\bigr]^K - \bigl[t^{2H} + (r-s)^{2H}
\bigr]^K- \bigl[(t-s)^{2H} + r^{2H}
\bigr]^K\\
&&{} + \bigl[(t-s)^{2H} + (r-s)^{2H}
\bigr]^K
\end{eqnarray*}
and
\[
\psi(t-r, s) = |t-r+s|^{2HK} - 2 |t-r|^{2HK} +
|t-r-s|^{2HK}.
\]
We begin with the position term. Note that if $K=1$, then $\varphi
(t,r,s) = 0$, so we may assume $K < 1$ and $H > \frac{1}{4}$. Assume
$0< s \le r \le t$, and let $p:= t-r$. By the fundamental theorem of
calculus, we can write $\varphi(t, t-p, s)$ as
\begin{eqnarray*}
&&
2HK \int_0^s \bigl[ t^{2H} + (t-p-
\xi)^{2H} \bigr]^{K-1} (t-p-\xi)^{2H-1} \\
&&\quad{}- \bigl[
(t-s)^{2H} + (t-p-\xi)^{2H} \bigr]^{K-1} (t-p-
\xi)^{2H-1} \,d\xi
\\
&&\qquad=\int_0^s \int_0^s
4H^2K(1-K) \bigl[(t-\eta)^{2H} + (t-p-\xi
)^{2H} \bigr]^{K-2}\\
&&\qquad\hspace*{38pt}{}\times (t-\eta)^{2H-1}(t-p-
\xi)^{2H-1}\,d\xi\,d\eta
\\
&&\qquad\le4H^2K(1-K)s^2 \bigl[ (t-r)^{2H} +
(r-s)^{2H} \bigr]^{K-2} (t-r)^{2H-1}(r-s)^{2H-1}
\\
&&\qquad\le Cs^2(t-r)^{2HK-2H-1}(r-s)^{2H-1}.
\end{eqnarray*}
This implies condition (ii) for the position term taking $\alpha=
\frac
{1}{2} + 2H >1$ and $\beta= 1-2H$.

Next, consider the distance term $\psi(t-r, s)$. Without loss of
generality, assume $r<t$.
Again using an integral representation, we have
\begin{eqnarray*}
\psi(t-r,s)&=& |t-r+s|^{2HK} - 2|t-r|^{2HK} +
|t-r-s|^{2HK}
\\
&=&\int_0^s 2HK \bigl[ (t-r+
\xi)^{2HK-1} - (t-r-\xi)^{2HK-1} \bigr]\,d\xi
\\
&=&\int_0^s \int_{-\xi}^{\xi}
2HK(2HK-1) [ t-r+\eta]^{2HK-2} \,d\eta\,d\xi
\\
&\le& Cs^2(t-r-s)^{2HK-2} \le Cs^2|t-r|^{-{3/2}},
\end{eqnarray*}
since $|t-r| \ge2s$ implies $(t-r-s)^{-{3/2}} \le2^{{3/2}}|t-r|^{-{3/2}}$.

\textit{Condition} (iii).
\begin{eqnarray*}
&&\bigl\llvert{\mathbb E} \bigl[ B^{H,K}_t
\bigl(B_{r+s}^{H,K} - 2B_r^{H,K} +
B_{r-s}^{H,K}\bigr) \bigr] \bigr\rrvert
\\
&&\qquad = \frac{1}{2^K}\biggl| \bigl[ t^{2H} + (r+s)^{2H}
\bigr]^K - 2\bigl[ t^{2H} + r^{2H}
\bigr]^K + \bigl[ t^{2H} + (r-s)^{2H}
\bigr]^K
\\
&&\qquad\quad\hspace*{41.5pt}{} - \frac{1}{2^K} \bigl[ |t-r+s|^{2HK} -2 |t-r|^{2HK} +
|t-r-s|^{2HK} \bigr] \biggr|.
\end{eqnarray*}
Take first the term, $\varphi(t,r,s)$. If $r < 2s$, then
\[
\bigl| \bigl[ t^{2H} + (r+s)^{2H} \bigr]^K - 2
\bigl[ t^{2H} + r^{2H} \bigr]^K + \bigl[
t^{2H} + (r-s)^{2H} \bigr]^K\bigr| \le
Cs^{2HK} = Cs^{{1/2}},
\]
based on the inequality $a^K - b^K \le(a-b)^K$ for $a>b>0$ and $K<1$.
Hence, we will assume $r \ge2s$.
If $K=1$, then $H=\frac{1}{4}$, and we have
\begin{eqnarray*}
\llvert\sqrt{r+s}-2\sqrt{r}+\sqrt{r-s}\rrvert&=&\biggl\llvert\int
_0^s \frac{1}{2\sqrt{r+x}}\,dx - \int
_0^s \frac{1}{2\sqrt{r-s+x}}\,dx\biggr\rrvert
\\
&=&\frac{1}{4}\int_0^s \int
_0^s \frac{1}{(r-s+x+y)^{{3/2}}}\,dy\,dx
\\
&\le&\frac{1}{4}s^2(r-s)^{-{3/2}};
\end{eqnarray*}
and if $K < 1$,
\begin{eqnarray*}
&&\bigl\llvert\varphi(t,r,s)\bigr\rrvert
\\[-1pt]
&&\qquad =\biggl\llvert\int_0^s 2HK
\bigl[t^{2H} + (r+x)^{2H}\bigr]^{K-1}(r+x)^{2H-1}\,dx\\[-1pt]
&&\qquad\hspace*{12.5pt}{}
- \int_0^s 2HK\bigl[t^{2H} +
(r-s+x)^{2H}\bigr]^{K-1}(r-s+x)^{2H-1}\,dx\biggr
\rrvert
\\[-1pt]
&&\qquad \le\biggl\llvert\int_0^s \int
_0^s 4H^2K(K-1)\\[-1pt]
&&\qquad\quad\hspace*{30.5pt}{}\times
\bigl[t^{2H} + (r-s+x+y)^{2H}\bigr]^{K-2}(r-s+x+y)^{4H-2}\,dy\,dx
\biggr\rrvert
\\[-1pt]
&&\qquad\quad{} + \biggl\llvert\int_0^s \int
_0^s 2H(2H-1)K\bigl[t^{2H} +
(r-s+x+y)^{2H}\bigr]^{K-1}\\[-1pt]
&&\qquad\quad\hspace*{117.5pt}{}\times (r-s+x+y)^{2H-2}\,dy\,dx
\biggr\rrvert
\\[-1pt]
&&\qquad \le4H^2K(1-K)s^2(r-s)^{2HK-2} +
2H(1-2H)Ks^2(r-s)^{2HK-2} \\[-1pt]
&&\qquad\le Cs^2(r-s)^{-{3/2}}.
\end{eqnarray*}
This bound for $\varphi(t,r,s)$ also holds in the case $|t-r| < 2s$, so
the bound of $Cs^{{1/2}}$ is valid for this case.
Next for the second term. Note that if $|t-r| < 2s$, then
\[
\biggl\llvert\frac{1}{2^K} \bigl(|t-r+s|^{2HK}-2|t-r|^{2HK}
+ |t-r-s|^{2HK} \bigr)\biggr\rrvert\le2(3s)^{2HK} \le
Cs^{{1/2}}.
\]
If $|t-r|\ge2s$, then we have
\begin{eqnarray*}
&&
\bigl\llvert\sqrt{|t-r|+s} - 2\sqrt{|t-r|} + \sqrt{|t-r|-s}\bigr\rrvert\\[-1pt]
&&\qquad= \biggl
\llvert
\int_0^s \frac{1}{2\sqrt{|t-r|+x}}\,dx - \int
_0^s \frac
{1}{2\sqrt{|t-r|-s+x}}\,dx\biggr\rrvert
\\[-1pt]
&&\qquad=\int_0^s \int_0^s
\frac{1}{(|t-r|-s+x+y)^{3/2}}\,dy\,dx
\\[-1pt]
&&\qquad\le\frac{s^2}{4(|t-r|-s)^{{3/2}}} \le\frac
{s^2}{2|t-r|^{{3/2}}},
\end{eqnarray*}
using the inequality $\frac{1}{|t-r|-s} \le\frac{2}{|t-r|}$ for $|t-r|
\ge2s$. This bound for $\psi(t-r,s)$ holds even in the case $r < 2s$,
so the bound of $Cs^{{1/2}}$ when $r < 2s$ is verified as well.

\textit{Condition} (iv).
For the first part, we have for all $t \ge s$,
\begin{eqnarray*}
&&\bigl\llvert{\mathbb E} \bigl[ B_t^{H,K}
\bigl(B_{t+s}^{H,K} - B_{t-s}^{H,K} \bigr)
\bigr] \bigr\rrvert\\[-1pt]
&&\qquad=\biggl\llvert\frac{1}{2^K} \bigl[ t^{2H} +
(t+s)^{2H} \bigr]^K - \frac
{1}{2^K} \bigl[
t^{2H} + (t-s)^{2H} \bigr]^K\biggr\rrvert.
\end{eqnarray*}
This is bounded by $Cs^{{1/2}}$ if $t < 2s$. On the other hand,
if $t \ge2s$,
\begin{eqnarray*}
&&
\biggl\llvert\frac{1}{2^K} \bigl[ t^{2H} + (t+s)^{2H}
\bigr]^K - \frac
{1}{2^K} \bigl[ t^{2H} +
(t-s)^{2H} \bigr]^K\biggr\rrvert\\[-1pt]
&&\qquad = \biggl\llvert
\frac{1}{2^K} \int_{-s}^s 2HK \bigl[
t^{2H} + (t+x)^{2H} \bigr]^{K-1}(t+x)^{2H-1}\,dx
\biggr\rrvert
\\[-1pt]
&&\qquad\le Cs(t-s)^{2HK-1}\\[-1pt]
&&\qquad=Cs(t-s)^{-{1/2}}.
\end{eqnarray*}
For $0 < s \le r \le T$ with $t \ge2s$ and $|t-r| \ge2s$,
\begin{eqnarray*}
&&
\bigl\llvert{\mathbb E} \bigl[ B_r^{H,K}
\bigl(B_{t+s}^{H,K} - B_{t-s}^{H,K} \bigr)
\bigr] \bigr\rrvert\\[-1pt]
&&\qquad \le\biggl\llvert\frac{1}{2^K} \bigl[
r^{2H} + (t+s)^{2H} \bigr]^K -
\frac
{1}{2^K} \bigl[ r^{2H} + (t-s)^{2H}
\bigr]^K\biggr\rrvert
\\[-1pt]
&&\qquad\quad{} +\biggl\llvert\frac{1}{2^K}|r-t+s|^{2HK} -
\frac
{1}{2^K}|r-t-s|^{2HK}\biggr\rrvert
\\[-1pt]
&&\qquad\le Cs(t-s)^{-{1/2}} + Cs|r-t|^{-{1/2}}.
\end{eqnarray*}
If $t < 2s$ or $|t-r| < 2s$, then we have an upper bound of $Cs^{
{1/2}}$ by condition (i) and Cauchy--Schwarz.

For the third bound, if $t > 2s$,
\begin{eqnarray*}
&&
\bigl\llvert{\mathbb E} \bigl[ B_s^{H,K}
\bigl(B_{t}^{H,K} - B_{t-s}^{H,K} \bigr)
\bigr]\bigr\rrvert\\[-2pt]
&&\qquad\le\biggl\llvert\frac{1}{2^K} \bigl[ s^{2H}
+ t^{2H} \bigr]^K - \frac
{1}{2^K} \bigl[
s^{2H} + (t-s)^{2H} \bigr]^K\biggr\rrvert
\\[-2pt]
&&\qquad\quad{} +\biggl\llvert\frac{1}{2^K}(t-s)^{2HK} - \frac
{1}{2^K}(t-2s)^{2HK}
\biggr\rrvert
\\[-2pt]
&&\qquad\le\frac{2}{2^K}\int_0^s HK \bigl[
s^{2H}+(t-s+x)^{2H} \bigr]^K(t-s+x)^{2H-1}\,dx
\\[-2pt]
&&\qquad\quad{} + \frac{1}{2^{K+1}}\int_0^s
(t-2s+x)^{-{1/2}}\,dx
\\[-2pt]
&&\qquad\le Cs(t-2s)^{-{1/2}} = Cs^{{1/2}+\gamma
}(t-2s)^{-\gamma}
\end{eqnarray*}
for $\gamma= \frac{1}{2}$.\vspace*{-2pt}
\end{pf}
%
\begin{proposition}\label{pr5.2}
Let $B^{H,K}$ be a bifractional Brownian motion with parameters $H \le
1/2$ and $HK = 1/4$. Then condition \textup{(v)} of Section~\ref{sec4} holds, with the
functions $\eta^+(t) = 2C_K^+t$ and $\eta^-(t) = 2C_K^-t$, where
\begin{eqnarray*}
C_K^+ &=& \frac{1}{4^K} \Biggl( 2+\sum
_{m=1}^\infty(\sqrt{2m+1}-2\sqrt{2m}+\sqrt{2m-1}
)^2 \Biggr),
\\[-2pt]
C_K^- &=& \frac{(2-\sqrt{2})^2}{2^{2K}}+\frac{1}{4^K}\sum
_{m=1}^\infty(\sqrt{2m+2}-2\sqrt{2m+1}+\sqrt{2m}
)^2.\vspace*{-2pt}
\end{eqnarray*}
\end{proposition}
\begin{pf}
As in Proposition~\ref{pr5.1}, we use the decomposition
\begin{eqnarray*}
\beta_n (j,k) &=& \frac{1}{2^K}\varphi\biggl(\frac{j}{n},
\frac{k}{n}, \frac{1}{n} \biggr) + \frac{1}{2^K}\psi\biggl(
\frac{j-k}{n}, \frac
{1}{n} \biggr) \\[-2pt]
&=& 2^{-K}n^{-{1/2}}
\varphi(j,k,1) +2^{-K}n^{-{1/2}}\psi(j-k,1).
\end{eqnarray*}
The first task is to show that
%
\begin{equation}
\label{phito0} \lim_{n \to\infty} \sum
_{j,k=1}^{\lfloor nt \rfloor} n^{-1} \varphi(j,k,1)^2
=0.\vspace*{-2pt}
\end{equation}

\textit{Proof of} (\ref{phito0}). We consider two cases, based on the
value of $H$. First, assume $H < \frac{1}{2}$. Then
\begin{eqnarray*}
\varphi(j,k,1) &=& \bigl[ (j+1)^{2H} + (k+1)^{2H}
\bigr]^K - \bigl[ (j+1)^{2H} + k^{2H}
\bigr]^K
\\[-2pt]
&&{} - \bigl[ j^{2H} + (k+1)^{2H} \bigr]^K +
\bigl[ j^{2H} + k^{2H} \bigr]^K
\\[-2pt]
&=& \int_0^1 2HK \bigl[ (j+1)^{2H}
+ (k+x)^{2H} \bigr]^{K-1}(k+x)^{2H-1} \,dx
\\[-2pt]
&&{} -\int_0^1 2HK \bigl[ j^{2H} +
(k+x)^{2H} \bigr]^{K-1}(k+x)^{2H-1} \,dx
\\[-2pt]
&=& \int_0^1\int_0^1
{4H^2K(1-K)} \bigl[ (j+y)^{2H} + (k+x)^{2H}
\bigr]^{K-2}\\[-2pt]
&&\hspace*{28.1pt}{}\times(k+x)^{2H-1}(j+y)^{2H-1}\,dy\,dx
\\[-2pt]
&\le& Ck^{2HK-2H-1}j^{2H-1} = Ck^{-{1/2} - 2H}j^{2H-1}.
\end{eqnarray*}
With this bound, it follows that
\begin{eqnarray*}
\frac{1}{n}\sum_{j,k=1}^{\lfloor nt \rfloor}
\varphi(j,k,1)^2 &\le& \frac
{C}{n}\sum
_{j=1}^{\lfloor nt \rfloor}j^{4H-2}\sum
_{k=1}^\infty k^{-1-4H}
\\[-2pt]
&\le&\frac{C}{n}\lfloor nt \rfloor^{4H-1} \\[-2pt]
&\le&
Ctn^{4H-2},
\end{eqnarray*}
which tends to zero as $n \to\infty$ because $H < \frac{1}{2}$.

Next, we have the case $H = \frac{1}{2}$. Note that this implies $K =
\frac{1}{2}$, and we have
\begin{eqnarray*}
\bigl\llvert\varphi(j,k,1)\bigr\rrvert&=& \llvert\sqrt{j+k+2} - 2\sqrt{j+k+1}+
\sqrt{j+k}\rrvert\\[-2pt]
&\le& C(j+k)^{-{3/2}}.
\end{eqnarray*}
So with this bound,
\begin{eqnarray*}
\sum_{j,k=1}^{\lfloor nt \rfloor} n^{-1}
\varphi(j,k,1)^2 &\le&\frac
{C}{n}\sum
_{j,k=1}^{\lfloor nt \rfloor} (j+k)^{-3}
\\[-2pt]
&\le&\frac{C}{n}\sum_{j=1}^{\lfloor nt \rfloor}
\sum_{m=j+1}^\infty m^{-3} \\[-2pt]
&\le&
\frac{C}{n}\sum_{j=1}^{\lfloor nt \rfloor}
j^{-2},
\end{eqnarray*}
which tends to zero as $n \to\infty$ because $j^{-2}$ is summable.
Hence, (\ref{phito0}) is proved.

From (\ref{phito0}), it follows that 
to investigate the limit behavior of $\eta_n^+(t),\eta_n^-(t)$, it is
enough to consider
\[
\frac{1}{n}\sum_{j,k=1}^{ \lfloor{nt/2} \rfloor}
\psi(2j-2k,1)^2 + \psi(2j-2k,1)^2 = \frac{2}{n}\sum
_{j,k=1}^{ \lfloor{nt/2} \rfloor}\psi(2j-2k,1)^2\vadjust{\goodbreak}
\]
and
\begin{eqnarray*}
&&
\frac{1}{n}\sum_{j,k=1}^{ \lfloor{nt/2} \rfloor}
\psi(2j-2k+1,1)^2 + \psi(2j-2k-1,1)^2\\
&&\qquad =
\frac{2}{n}\sum_{j,k=1}^{ \lfloor{nt/2} \rfloor}
\psi(2j-2k+1,1)^2;
\end{eqnarray*}
since the sums of $\psi(2j-2k+1,1)^2$ and $\psi(2j-2k-1,1)^2$ are equal
by symmetry. We start with
\begin{eqnarray*}
&&\frac{1}{n}\sum_{j,k=1}^{ \lfloor{nt/2} \rfloor}
\psi(2j-2k,1)^2
\\
&&\qquad = \frac{1}{4^Kn}\sum_{j,k=1}^{ \lfloor{nt/2} \rfloor} \bigl(
\sqrt{|2j-2k+1|} - 2\sqrt{|2j-2k|} + \sqrt{|2j-2k-1|} \bigr)^2
\\
&&\qquad = \frac{1}{4^Kn}\sum_{j=1}^{ \lfloor{nt/2} \rfloor} 4 \\
&&\qquad\quad{} +
\frac{2}{4^Kn}\sum_{j=1}^{ \lfloor{nt/2} \rfloor}\sum
_{k=1}^{j-1} (\sqrt{2j-2k+1}-2\sqrt{2j-2k} +
\sqrt{2j-2k-1} )^2
\\
&&\qquad = \frac{4 \lfloor{nt}/{2} \rfloor}{4^Kn} + \frac{2}{4^Kn}\sum
_{j=1}^{ \lfloor{nt/2} \rfloor}
\sum_{m=1}^{j-1} (\sqrt{2m+1}-2\sqrt{2m} +
\sqrt{2m-1} )^2
\\
&&\qquad = \frac{4 \lfloor{nt}/{2} \rfloor}{4^Kn} + \frac{2}{4^Kn}\sum
_{j=1}^{ \lfloor{nt/2} \rfloor}
\sum_{m=1}^{\infty} (\sqrt{2m+1}-2\sqrt{2m} +
\sqrt{2m-1} )^2
\\
&&\qquad\quad{} - \frac{2}{4^Kn}\sum_{j=1}^{ \lfloor{nt/2} \rfloor}\sum
_{m=j}^{\infty} (\sqrt{2m+1}-2\sqrt{2m} +
\sqrt{2m-1} )^2,
\end{eqnarray*}
where the last term tends to zero since
\[
\sum_{m=j}^{\infty} (\sqrt{2m+1}-2\sqrt{2m} +
\sqrt{2m-1} )^2 \le\sum_{m=j}^{\infty}(2m-1)^{-3}
\le C(2j-1)^{-2}
\]
and
\[
\frac{C}{n}\sum_{j=1}^{ \lfloor{nt/2} \rfloor}
(2j-1)^{-2} \longrightarrow0
\]
as $n \to\infty$. We therefore conclude that
\begin{eqnarray*}
\eta^+(t) &=& \lim_{n\to\infty}\sum_{j,k=1}^{
\lfloor{nt}/{2} \rfloor
}
\bigl(\beta_n(2j-1,2k-1)^2 + \beta_n(2j-2,2k-2)^2
\bigr)
\\
& = &\lim_{n\to\infty} \frac{2}{n}\sum_{j,k=1}^{ \lfloor{nt/2} \rfloor}
\psi(2j-2k,1)^2 = 2C_K^+t,
\end{eqnarray*}
where
\[
C_K^+ 
= \frac{1}{4^K} \Biggl( 2 + \sum
_{m=1}^\infty(\sqrt{2m+1}-2\sqrt{2m}+
\sqrt{2m-1} )^2 \Biggr).
\]
For the other term,
\begin{eqnarray*}
&&\frac{1}{n}\sum_{j,k=1}^{ \lfloor{nt/2} \rfloor}
\psi(2j-2k+1,1)^2
\\
&&\qquad = \frac{1}{4^Kn}\sum_{j=1}^{ \lfloor{nt/2} \rfloor} (2-
\sqrt{2})^2\\
&&\qquad\quad{} +\frac
{2}{4^Kn}\sum_{j=1}^{ \lfloor{nt/2} \rfloor}
\sum_{k=1}^{j-1} (\sqrt{2j-2k+2}-2\sqrt
{2j-2k+1}-\sqrt{2j-2k} )^2.
\end{eqnarray*}
Hence, by a similar computation,
\[
\eta^-(t) = \lim_{n\to\infty}\sum_{j,k=1}^{ \lfloor{nt/2} \rfloor}
\beta_n(2j-1,2k-2)^2 + \beta_n(2j-2,2k-1)^2
= 2C_K^-t,
\]
where
\[
C_K^- = \frac{(2-\sqrt{2})^2}{2^{2K}}+\frac{1}{4^K}\sum
_{m=1}^\infty(\sqrt{2m+2}-2\sqrt{2m+1}+\sqrt{2m}
)^2.
\]
\upqed
\end{pf}

As a concluding remark, it is easy to show that $C_K^+ > C_K^-$, and in
general we have $\eta^+(t) \ge\eta^-(t)$.

\subsection{A Gaussian process with differentiable covariance function}\label{sec5.2}
Consider the following class of Gaussian processes. Let $\{F_t, 0 \le t
\le T\}$ be a mean-zero Gaussian process with covariance defined by
%
\begin{equation}
\label{CovFdef} {\mathbb E} [ F_r F_t ] =
r\phi\biggl(\frac
{t}{r} \biggr),\qquad t \ge r,
\end{equation}
where $\phi\dvtx  [1, \infty) \to{\mathbb R}$ is twice-differentiable on
$(1,\infty)$ and satisfies the following:
\begin{longlist}[($\phi.1$)]
\item[($\phi.1$)] $\|\phi\|_\infty:= {\sup_{x\ge1}}|\phi(x)| \le c_{\phi,0} <
\infty$.
\item[($\phi.2$)] For $1 < x < \infty$,
\[
\bigl| \phi'(x) \bigr| \le\frac{c_{\phi,1}}{\sqrt{x-1}}.
\]
\item[($\phi.3$)] For $1 < x < \infty$,
\[
\bigl| \phi''(x) \bigr| \le c_{\phi,2}x^{-{1/2}}(x-1)^{-{3/2}},
\]
%
\end{longlist}
where $c_{\phi,j}$, $j = 0,1,2$ are nonnegative constants.

%
\begin{proposition}\label{pr5.3}
The process $\{ F_t, 0\le t\le T\}$ described above satisfies
conditions \textup{(i)--(iv)} of Section~\ref{sec4}.
\end{proposition}
\begin{pf}
\textit{Condition} (i). By conditions $(\phi.1)$ and $(\phi.2)$,
\begin{eqnarray*}
{\mathbb E} \bigl[ (F_t - F_{t-s} )^2 \bigr]&=&t
\phi(1) + (t-s)\phi(1)-2(t-s)\phi\biggl(1+\frac{s}{t-s} \biggr)
\\
&\le&2(t-s)\biggl\llvert\phi\biggl(1+\frac{s}{t-s} \biggr)-\phi
(1)\biggr
\rrvert+ s\bigl\llvert\phi(1)\bigr\rrvert
\\
&\le&2(t-s)\biggl\llvert\int_1^{1+{s}/({t-s})}
\phi'(x)\,dx\biggr\rrvert+ s\| \phi\|_\infty
\\
&\le&2(t-s)\int_1^{1+{s}/({t-s})} \frac{c_{\phi,1}}{\sqrt{x-1}}\,dx + s
\|\phi\|_\infty
\\
&\le& Cs^{{1/2}}\sqrt{t-s}+ s\|\phi\|_\infty
\\
&\le& Cs^{{1/2}},
\end{eqnarray*}
where the constant $C$ depends on $\max\{ \sqrt{T}, \|\phi\|_\infty\}$.

\textit{Condition} (ii). For $2s \le r \le t-2s$ we have by the
mean value theorem
\begin{eqnarray*}
&&
\bigl\llvert{\mathbb E} [ F_tF_r -
F_{t-s}F_r - F_t F_{r-s} +
F_{t-s}F_{r-s} ] \bigr\rrvert\\
&&\qquad= \biggl\llvert r \biggl[
\phi\biggl(\frac{t}{r} \biggr) - \phi\biggl(\frac
{t-s}{r} \biggr)
\biggr] - (r-s) \biggl[ \phi\biggl(\frac
{t}{r-s} \biggr) - \phi\biggl(
\frac{t-s}{r-s} \biggr) \biggr]\biggr\rrvert
\\
&&\qquad\le s\sup_{ [({t-s})/{r}, {t}/({r-s}) ]} \bigl\llvert\phi''(x)
\bigr\rrvert\biggl( \frac{t}{r-s} - \frac{t-s}{r} \biggr)
\\
&&\qquad\le c_{\phi,2}s \biggl(\frac{t-s}{r} \biggr)^{-{1/2}} \biggl(
\frac
{t-s}{r} -1 \biggr)^{-{3/2}} \biggl(\frac{ts}{r(r-s)} \biggr)
\\
&&\qquad\le\frac{C\sqrt{T}s^2}{(t-r)^{{3/2}}} = C\sqrt{T}s^2
|t-r|^{-{3/2}}.
\end{eqnarray*}

\textit{Condition} (iii). By symmetry we can assume $r\le t$.
Consider the following cases: First, suppose $2s \le r \le t-2s$. Then
we have
\begin{eqnarray*}
&&
\bigl\llvert{\mathbb E} \bigl[ F_t (F_{r+s} -
2F_r +F_{r-s}) \bigr] \bigr\rrvert\\
&&\qquad= \biggl\llvert(r+s)
\phi\biggl(\frac{t}{r+s} \biggr) - 2r\phi\biggl(\frac
{t}{r} \biggr)
+ (r-s)\phi\biggl(\frac{t}{r-s} \biggr)\biggr\rrvert
\\
&&\qquad= \biggl\llvert(r+s) \biggl[\phi\biggl(\frac{t}{r+s} \biggr) -\phi
\biggl(\frac
{t}{r} \biggr) \biggr] - (r-s) \biggl[\phi\biggl(
\frac{t}{r} \biggr) -\phi\biggl(\frac{t}{r-s} \biggr) \biggr]\biggr
\rrvert
\\
&&\qquad\le\frac{st}{r} \sup_{ [ {t}/({r+s}), {t}/({r-s})
]}\bigl\llvert\phi''(x)
\bigr\rrvert\biggl( \frac{t}{r-s} - \frac
{t}{r+s} \biggr)
\\
&&\qquad\le\frac{2s^2t^2c_{\phi,2}}{r(r-s)(r+s)} \biggl(\frac{r+s}{t} \biggr
)^{{1/2}} \biggl(
\frac{r+s}{t-r-s} \biggr)^{{3/2}}
\\
&&\qquad\le\frac{Cs^2t^{{3/2}}}{r(t-r)^{{3/2}}}.
\end{eqnarray*}
There are two possibilities, depending on the value of $r$. If $r \ge
\frac{t}{2}$, then $\frac{t}{r} \le2$, and we have a bound of
\[
Cs^2 \biggl(\frac{t}{r} \biggr) \biggl(\frac{\sqrt{T}}{(t-r)^{{3/2}}}
\biggr) \le2C\sqrt{T}s^2|t-r|^{-{3/2}}.
\]
On the other hand, if $r < \frac{t}{2}$, then $\frac{t}{t-r} \le2$ and
$r< t-r$. Then the bound is
\[
Cs^2 \biggl(\frac{t}{t-r} \biggr) \biggl(\frac{\sqrt{T}}{r\sqrt{t-r}}
\biggr) \le{2C\sqrt{T}s^2} \bigl[ (r-s)^{-{3/2}} +
|t-r|^{-{3/2}} \bigr].
\]

For the case $|t-r| < 2s$, assume that $t = r+ks$ for some $0 \le k <
2$. Then
\begin{eqnarray*}
&&\bigl\llvert{\mathbb E} \bigl[ F_t (F_{r+s}-2F_r+F_{r-s}
) \bigr]\bigr\rrvert
\\
&&\qquad =\biggl\llvert\bigl(t \wedge(r_s) \bigr)\phi\biggl(
\frac{t\vee
(r+s)}{t\wedge(r+s)} \biggr)  - 2 r \phi\biggl(\frac{t}{r} \biggr) + (r-s)
\phi\biggl(\frac{t}{r-s} \biggr)\biggr\rrvert
\\
&&\qquad =\biggl\llvert\bigl(t \wedge(r_s) \bigr)\phi\biggl(
\frac{t\vee
(r+s)}{t\wedge(r+s)} \biggr) -(r+s)\phi(1)\\
&&\qquad\hspace*{12pt}{} - 2 r \phi\biggl(\frac
{t}{r}
\biggr) +2r\phi(1) + (r-s)\phi\biggl(\frac{t}{r-s} \biggr)- (r-s)\phi(1)
\biggr\rrvert
\\
&&\qquad \le3(r+s)\biggl\llvert\phi\biggl( 1 + \frac{(k+1)s}{r-s} \biggr) -
\phi(1)
\biggr\rrvert\\
&&\qquad\le3(r+s)\biggl\llvert\int_1^{1+{(k+1)s}/({r-s})}
\phi'(x)\,dx\biggr\rrvert
\\
&&\qquad \le3(r+s) \int_1^{1+{(k+1)s}/({r-s})} \frac{c_{\phi,1}}{\sqrt{x-1}}
\,dx \\
&&\qquad\le C\sqrt{T}s^{{1/2}}.
\end{eqnarray*}

For the last case, note that if $t\wedge r < 2s$, then we have an upper
bound of $Cs^{{1/2}}$, since ${\mathbb E} [ F_s F_t ]
\le s\|\phi\|_\infty$.

\textit{Condition} (iv). Take first the bound for ${\mathbb
E} [F_t(F_{t+s}-F_{t-s}) ]$. Note that if $t < 2s$, then an
upper bound of $Cs^{{1/2}}$ is clear, so we will assume $t \ge
2s$. We have
\begin{eqnarray*}
&&
\bigl\llvert{\mathbb E} [ F_tF_{t+s} -
F_tF_{t-s} ] \bigr\rrvert\\
&&\qquad = \biggl\llvert t \phi\biggl(
\frac{t+s}{t} \biggr) - (t-s)\phi\biggl( \frac{t}{t-s} \biggr)\biggr
\rrvert
\\
&&\qquad\le(t-s)\sup_{ [({t+s})/{t}, {t}/({t-s}) ]} \bigl\llvert\phi
'(x)\bigr\rrvert
\biggl\llvert\frac{t+s}{t} -\frac{t}{t-s}\biggr\rrvert+ s\biggl\llvert
\phi\biggl(\frac{t+s}{t} \biggr)\biggr\rrvert
\\
&&\qquad\le c_{\phi,1}\frac{s^2}{t}\sqrt{\frac{t}{t+s}} \sqrt{
\frac
{t}{s}} + c_{\phi,0}s\frac{\sqrt{T}}{\sqrt{t-s}}
\\
&&\qquad\le Cs\sqrt{T}(t-s)^{-{1/2}}.
\end{eqnarray*}
For the case $r \neq t$, first assume $r \le t-2s$. By condition $(\phi.2)$,
\begin{eqnarray*}
\bigl\llvert{\mathbb E} [ F_rF_{t+s} -
F_rF_{t-s} ]\bigr\rrvert&=& \biggl\llvert r\phi\biggl(
\frac{t+s}{r} \biggr) - r\phi\biggl(\frac
{t-s}{r} \biggr)\biggr\rrvert\\
&\le&2s\sup_{[({t-s})/{r},({t+s})/{r}]}\bigl|\phi'(x)\bigr|
\\
&\le&\frac{2s\sqrt{r}c_{\phi,1}}{\sqrt{t-r-s}} \le\frac{C\sqrt
{T}s}{\sqrt{t-r}}.
\end{eqnarray*}
If $r \ge t+2s$, then
\begin{eqnarray*}
\bigl\llvert{\mathbb E} [ F_rF_{t+s} -
F_rF_{t-s} ]\bigr\rrvert&=& \biggl\llvert(t+s)\phi
\biggl(\frac{r}{t+s} \biggr) - (t-s)\phi\biggl(\frac
{r}{t-s} \biggr)
\biggr\rrvert
\\
&\le& t\int_0^{2s} \biggl\llvert
\phi' \biggl(\frac{r}{t-s+x} \biggr)\biggr\rrvert\,dx + 2s\|\phi
\|_\infty
\\
&\le&\frac{2stc_{\phi,1}\sqrt{t+s}}{\sqrt{r-t}} +
\frac{2sc_{\phi,0}\sqrt{T}}{\sqrt{t-s}}
\\
&\le& Cs(r-t)^{-{1/2}} + Cs(t-s)^{-{1/2}}.
\end{eqnarray*}
For the case $t < 2s$ or $|r-t| < 2s$, the bound follows from condition
(i) and Cauchy--Schwarz.

For the third part of condition (iv), we have for $t > 2s$,
\begin{eqnarray*}
{\mathbb E} [ F_sF_t - F_sF_{t-s}]
&=& s\phi\biggl(\frac{t}{s} \biggr) - s \phi\biggl(\frac{t-s}{s}
\biggr)\\
&\le& s \sup_{ [ ({t-s})/{s},{t/s} ]}\bigl|\phi'(x)\bigr| \biggl(\frac{t}{s}-
\frac{t-s}{s} \biggr)
\le\frac{c_{\phi,1}s}{\sqrt{({t-s})/{s}-1}}\\
&\le& Cs^{{3/2}}(t-2s)^{-{1/2}}
= Cs^{{1/2}+\gamma}(t-2s)^{-\gamma},
\end{eqnarray*}
where $\gamma= \frac{1}{2}$.
\end{pf}

%
\begin{proposition}\label{pr5.4} Suppose $\phi(x)$ satisfies conditions $(\phi.1)$,
$(\phi.3)$, and in addition, $\phi(x)$ satisfies
\[
\mbox{$(\phi.4)$:\quad}  \phi'(x) = \frac{\kappa}{\sqrt{x-1}} + \frac{\psi
(x)}{\sqrt{x}},
\]
where $\kappa\in{\mathbb R}$ and $\psi\dvtx (1,\infty) \to{\mathbb R}$ is
a bounded differentiable function satisfying $|\psi'(1+x)| \le C_\psi
x^{-{1/2}}$ for some positive constant $C_\psi$. Then condition
\textup{(v)} of Section~\ref{sec4} is satisfied, with $\eta^+(t) = C_\beta^+ t^2$, and
$\eta^-(t) = C^-_\beta t^2$ for positive constants $C_\beta^+,
C_\beta^-$.
\end{proposition}
%
\begin{remark}\label{re5.5} Observe that condition $(\phi.4)$ implies $(\phi.2)$,
but not $(\phi.3)$.
\end{remark}
\begin{pf*}{Proof of Proposition~\ref{pr5.4}}
We want to show
%
\begin{eqnarray}
\label{etaodd} \sum_{j,k=1}^{ \lfloor{nt/2} \rfloor}
\beta_n(2j-1,2k-1)^2 &\longrightarrow& C_{\beta,1}
t^2;
\\
\label{etavar} \sum_{j,k=1}^{ \lfloor{nt/2} \rfloor}
\beta_n(2j-2,2k-2)^2 &\longrightarrow& C_{\beta,2}
t^2;
\\
\label{etacross} \sum_{j,k=1}^{ \lfloor{nt/2} \rfloor}
\beta_n(2j-1,2k-2)^2 &\longrightarrow& C_{\beta,3}
t^2,
\end{eqnarray}
so that $C_\beta^+ = C_{\beta,1}+C_{\beta,2}$ and $C_\beta^- =
2C_{\beta,3}$. We will show computations for (\ref{etaodd}), with the
others being similar. As in Proposition~\ref{pr5.2},
\begin{eqnarray*}
\sum_{j,k=1}^{ \lfloor{nt/2} \rfloor} \beta_n(2j-1,2k-1)^2
&=& \sum_{j=1}^{ \lfloor{nt/2} \rfloor} \beta_n(2j-1,2j-1)^2\\
&&{} + 2\sum_{j=1}^{ \lfloor{nt/2} \rfloor}\sum
_{k=1}^{j-1}\beta_n(2j-1,2k-1)^2,
\end{eqnarray*}
so it is enough to show
%
\begin{equation}
\label{Betaless}\lim_{n \to\infty} \sum
_{j=1}^{ \lfloor{nt/2} \rfloor} \sum_{k=1}^{j-1}
\beta_n(2j-1, 2k-1)^2 = C_1 t^2
\end{equation}
and
%
\begin{equation}
\label{Betaeq}\lim_{n \to\infty} \sum
_{j=1}^{ \lfloor{nt/2} \rfloor} \beta_n(2j-1,
2j-1)^2 = C_2 t^2.
\end{equation}

\textit{Proof of} (\ref{Betaless}). For $1 \le k \le j-1$, we have
\begin{eqnarray*}
\beta_n(2j-1,2k-1) &=& \frac{2k}{n} \biggl( \phi\biggl(
\frac
{2j}{2k} \biggr) - \phi\biggl(\frac{2j-1}{2k} \biggr) \biggr)\\
&&{} -
\frac
{2k-1}{n} \biggl( \phi\biggl(\frac{2j}{2k-1} \biggr) - \phi\biggl(
\frac
{2j-1}{2k-1} \biggr) \biggr)
\\
&=& \frac{2k}{n}\int_{({2j-1})/({2k})}^{{2j}/({2k})}
\phi'(x)\,dx \\
&&{}- \frac{2k-1}{n}\int_{({2j-1})/({2k-1})}^{{2j}/({2k-1})}
\phi'(x)\,dx.
\end{eqnarray*}
Using the change of index $j = k+m$ and a change of variable for the
two integrals, this becomes
%
\begin{eqnarray}
\label{Iphi1} \beta_n(2j-1,2k-1) &=& \frac{1}{n}
\int_{2m-1}^{2m} \phi' \biggl( 1+
\frac
{y}{2k} \biggr)\,dy \nonumber\\[-8pt]\\[-8pt]
&&{}- \frac{1}{n}\int_{2m}^{2m+1}
\phi' \biggl( 1+\frac
{y}{2k-1} \biggr)\,dy.\nonumber
\end{eqnarray}

With the decomposition of $(\phi.4)$, we will address (\ref{Iphi1}) in
two parts. Using the first term, we have
\begin{eqnarray*}
&&\frac{\kappa}{n}\int_{2m-1}^{2m} \sqrt{
\frac{2k}{y}}\,dy - \frac
{\kappa
}{n}\int_{2m}^{2m+1}
\sqrt{\frac{2k-1}{y}}\,dy
\\
&&\qquad=\frac{2\kappa}{n} \bigl[ \sqrt{2k} ( \sqrt{2m}-\sqrt{2m-1} ) - \sqrt
{2k-1} (
\sqrt{2m+1}-\sqrt{2m} ) \bigr].
\end{eqnarray*}
We are interested in the sum
%
\begin{equation}
\label{decsum1}\sum_{k=1}^{ \lfloor{nt/2} \rfloor}
\sum_{m=1}^{ \lfloor{nt}/{2}
\rfloor
- k}\frac{4\kappa^2}{n^2} \bigl[
\sqrt{2k} ( \sqrt{2m}-\sqrt{2m-1} ) - \sqrt{2k-1} (\sqrt{2m+1}-\sqrt
{2m} )
\bigr]^2.\hspace*{-30pt}
\end{equation}
We can write
\begin{eqnarray*}
&&\sqrt{2k} ( \sqrt{2m}-\sqrt{2m-1} ) - \sqrt{2k-1} (\sqrt{2m+1}-\sqrt
{2m} )
\\
&&\qquad=-\sqrt{2k-1} (\sqrt{2m+1}-2\sqrt{2m}+\sqrt{2m-1} ) \\
&&\qquad\quad{} + (\sqrt{2k}-\sqrt
{2k-1} ) (
\sqrt{2m}-\sqrt{2m-1} ).
\end{eqnarray*}
Observe that
\[
\bigl[ (\sqrt{2k}-\sqrt{2k-1} ) (\sqrt{2m}-\sqrt{2m-1} ) \bigr]^2
\le\frac{1}{(2k-1)(2m-1)}
\]
and so
\[
\frac{4\kappa^2}{n^2}\sum_{k=1}^{ \lfloor{nt/2} \rfloor} \sum
_{m=1}^{ \lfloor{nt}/{2} \rfloor- k} \frac
{1}{(2k-1)(2m-1)} \le
\frac{4\kappa^2}{n^2} \Biggl(\sum_{k=1}^{ \lfloor{nt/2} \rfloor}
\frac
{1}{2k-1} \Biggr)^2 \le\frac{C\log(nt)^2}{n^2}.
\]
Therefore the contribution of this term is zero, and it follows by
Cauchy--Schwarz that the only significant term is
\begin{eqnarray*}
&&\frac{4\kappa^2}{n^2}\sum_{k=1}^{ \lfloor{nt/2} \rfloor} \sum
_{m=1}^{ \lfloor{nt}/{2} \rfloor-
k}(2k-1) (\sqrt{2m+1}-2\sqrt{2m}+
\sqrt{2m-1} )^2
\\
&&\qquad = 4\kappa^2\sum_{m=1}^{ \lfloor{nt/2} \rfloor} (
\sqrt{2m+1}-2\sqrt{2m}+\sqrt{2m-1} )^2\sum
_{k=1}^{ \lfloor{nt}/{2} \rfloor
-m}\frac{2k-1}{n^2}
\\
&&\qquad = 4\kappa^2\sum_{m=1}^{ \lfloor{nt/2} \rfloor} (
\sqrt{2m+1}-2\sqrt{2m}+\sqrt{2m-1} )^2\frac{ ( \lfloor{nt}/{2}
\rfloor- m )^2}{n^2},
\end{eqnarray*}
which converges as $n \to\infty$ to
\[
\kappa^2t^2 \sum_{m=1}^\infty
(\sqrt{2m+1}-2\sqrt{2m}+\sqrt{2m-1} )^2.
\]
Next, we consider the term $\frac{1}{\sqrt{x}}\psi(x)$. The
contribution of this term to (\ref{Iphi1}) is
%
\begin{eqnarray}
\label{Iphi2}
&&\frac{1}{n} \int_{2m-1}^{2m}
\sqrt{\frac{2k}{2k+y}}\psi\biggl(1+\frac{y}{2k} \biggr)\,dy
\nonumber\\[-8pt]\\[-8pt]
&&\qquad{}-
\frac{1}{n} \int_{2m}^{2m+1} \sqrt{
\frac{2k-1}{2k-1+y}}\psi\biggl(1+\frac{y}{2k-1} \biggr)\,dy.\nonumber
\end{eqnarray}
We can bound (\ref{Iphi2}) by
\begin{eqnarray*}
&&\frac{1}{n} \biggl\llvert\int_{2m-1}^{2m}
\sqrt{\frac{2k}{2k+y}}\psi\biggl(1+\frac{y}{2k} \biggr)\,dy - \int
_{2m}^{2m+1} \sqrt{\frac
{2k-1}{2k-1+y}}\psi
\biggl(1+\frac{y}{2k-1} \biggr)\,dy\biggr\rrvert
\\
&&\qquad \le\frac{1}{n} \biggl[ \sup_{(1,\infty)}\bigl|\psi(x)\bigr|
\frac
{\sqrt{2k}-\sqrt{2k-1}}{\sqrt{2k+2m-1}}\\
&&\qquad\quad\hspace*{11.5pt}{} + \sqrt{\frac
{2k}{2k+2m-1}}\biggl\llvert\int
_{2m-1}^{2m} \psi\biggl(1 + \frac
{y}{2k}
\biggr)\,dy \\
&&\qquad\quad\hspace*{93.1pt}{} - \int_{2m}^{2m+1} \psi\biggl(1 +
\frac
{y}{2k-1} \biggr)\,dy\biggr\rrvert\biggr]
\\
&&\qquad = \frac{1}{n} ( A_{k,m} + B_{k,m} ).
\end{eqnarray*}
Since $|\psi(x)|$ is bounded, we have
%
\begin{equation}
\label{Psiu1}A_{k,m} \le\frac{C}{\sqrt{2k-1}\sqrt{2k+2m-1}} \le
\frac
{C}{\sqrt{2k-1}\sqrt{2m-1}}.
\end{equation}
For $B_{k,m}$ using that $|\psi'(x+1)| \le Cx^{-{1/2}}$,
\begin{eqnarray*}
&&\biggl\llvert\int_{2m-1}^{2m} \psi\biggl(1 +
\frac{y}{2k} \biggr)\,dy - \int_{2m}^{2m+1}
\psi\biggl(1 + \frac{y}{2k-1} \biggr)\,dy\biggr\rrvert
\\
&&\qquad=\biggl\llvert\int_{2m-1}^{2m} \psi\biggl(1+
\frac{u}{2k} \biggr) - \psi\biggl(1+\frac{u+1}{2k-1} \biggr)\,du\biggr
\rrvert
\\
&&\qquad \le\int_{2m-1}^{2m} \biggl\llvert\int
_{({u+1})/({2k-1})}^{
{u}/({2k})} \psi'(1+v) \,dv \biggr
\rrvert \,du
\\
&&\qquad\le C\int_{2m-1}^{2m} \int_{{u}/({2k})}^{({u+1})/({2k-1})}
v^{-{1/2}} \,dv\,du \\
&&\qquad\le\frac{C}{\sqrt{2k-1}} ( \sqrt{2m+1} - \sqrt{2m} )
\\
&&\qquad\le\frac{C}{\sqrt{2k-1}\sqrt{2m-1}}
\end{eqnarray*}
so that
%
\begin{equation}
\label{Psiu2} B_{k,m} \le\sqrt{\frac{2k}{2k+2m-1}}\cdot
\frac{C}{\sqrt{2k-1}\sqrt{2m-1}} \le\frac{C}{\sqrt{2k-1}\sqrt{2m-1}}.
\end{equation}
Hence, from (\ref{Psiu1}) and (\ref{Psiu2}), we obtain
\begin{eqnarray*}
&&\sum_{k=1}^{ \lfloor{nt/2} \rfloor} \sum
_{m=1}^{ \lfloor{nt}/{2} \rfloor- k} \frac
{C}{n^2} \biggl(
\frac{1}{\sqrt{2k-1}\sqrt{2m-1}} \biggr)^2
\\
&&\qquad \le\frac{C}{n^2}\sum_{k,m=1}^{ \lfloor{nt/2} \rfloor}
\frac{1}{(2m-1)(2k-1)} \le\frac{C \log(n)^2}{n^2},
\end{eqnarray*}
so the portion represented by (\ref{Iphi2}) tends to zero as $n \to
\infty$. Since this term is not significant, it follows by
Cauchy--Schwarz that the behavior of
\[
\sum_{j=1}^{ \lfloor{nt/2} \rfloor}\sum
_{k=1}^{j-1} \beta_n(2j-1,2k-1)^2
\]
is dominated by equation (\ref{decsum1}), and we have result (\ref
{Betaless}), with
\[
C_1 = \kappa^2\sum_{m=1}^\infty
(\sqrt{2m+1}-2\sqrt{2m}+\sqrt{2m-1} )^2.
\]

\textit{Proof of} (\ref{Betaeq}). For each $j$,
\begin{eqnarray*}
&&
\beta_n(2j-1,2j-1)^2 \\
&&\qquad= \biggl(\frac{2j}{n}
\phi(1) - 2\frac
{2j-1}{n}\phi\biggl(\frac{2j}{2j-1} \biggr) +
\frac{2j-1}{n}\phi(1) \biggr)^2
\\
&&\qquad=\frac{1}{n^2} \biggl[\phi(1) + (4j-2) \biggl( \phi(1) - \phi\biggl(1+
\frac
{1}{2j-1} \biggr) \biggr) \biggr]^2
\\
&&\qquad= \frac{\phi(1)^2}{n^2} + \frac{4(2j-1)\phi(1)}{n^2} \biggl(\phi(1) -
\phi\biggl(1+
\frac{1}{2j-1} \biggr) \biggr)
\\
&&\qquad\quad{} +\frac
{4(2j-1)^2}{n^2} \biggl(\phi(1) - \phi\biggl(1+\frac{1}{2j-1}
\biggr) \biggr)^2.
\end{eqnarray*}
Since $\llvert\phi(1) - \phi(1+\frac{1}{2j-1} )\rrvert
\le
\frac{c_{\phi,3}}{\sqrt{2j-1}}$ by $(\phi.3)$, we see that
\[
\sum_{j=1}^{ \lfloor{nt/2} \rfloor} \biggl[\frac
{\phi(1)^2}{n^2}
+ \frac{4(2j-1)\phi
(1)}{n^2}\biggl\llvert\phi(1) - \phi\biggl(1+\frac{1}{2j-1}
\biggr)\biggr\rrvert\biggr] \le Cn^{-{1/2}},
\]
which implies only the last term is significant in the limit. Again we
use $(\phi.4)$ to obtain
\begin{eqnarray*}
&&
\phi(1) - \phi\biggl(1+\frac{1}{2j-1} \biggr) \\
&&\qquad= -\int
_1^{1+
{1}/({2j-1})} \phi'(x)\,dx
\\
&&\qquad= -\kappa\int_1^{1+{1}/({2j-1})} \frac{1}{\sqrt{x-1}}\,dx -
\int_1^{1+{1}/({2j-1})} \frac{1}{\sqrt{x}}\psi(x)\,dx
\\
&&\qquad= -\frac{2\kappa}{\sqrt{2j-1}} + O \biggl(\frac{1}{2j-1} \biggr);
\end{eqnarray*}
hence
\[
\frac{4(2j-1)^2}{n^2} \biggl(\phi(1) - \phi\biggl(1+\frac
{1}{2j} \biggr)
\biggr)^2 = \frac{16\kappa^2(2j-1)^2}{n^2(2j-1)} + O \biggl(\frac
{j^{{1/2}}}{n^2}
\biggr)
\]
and taking $n \to\infty$,
\[
\lim_{n \to\infty}\sum_{j=1}^{ \lfloor{nt/2} \rfloor}
\frac{16\kappa^2(2j-1)}{n^2} + O \biggl(\frac{j^{{1/2}}}{n^2}
\biggr) = 4\kappa^2
t^2,
\]
which gives (\ref{Betaeq}). Thus (\ref{etaodd}) is proved with
$C_{\beta,1} = 4\kappa^2+2\kappa^2\sum_{m=1}^\infty(\sqrt{2m+1}-2\sqrt
{2m}+\sqrt{2m-1} )^2$.

By similar computations,
\[
C_{\beta,2} = 4\kappa^2 + 2\kappa^2\sum
_{m=1}^\infty(\sqrt{2m+1}-2\sqrt{2m}+\sqrt{2m-1}
)^2
\]
and
\[
C_{\beta,3} = 4\kappa^2 + 2\kappa^2\sum
_{m=1}^\infty(\sqrt{2m+2}-2\sqrt{2m+1}+\sqrt{2m}
)^2
\]
and so
\begin{eqnarray*}
C_\beta^+ &=& C_{\beta,1} + C_{\beta,2} =8
\kappa^2+ 4\kappa^2 \sum_{m=1}^\infty
(\sqrt{2m+1}-2\sqrt{2m}+\sqrt{2m-1} )^2,
\\
C_\beta^- &=& 2C_{\beta,3} = 8\kappa^2+4
\kappa^2\sum_{m=1}^\infty(
\sqrt{2m+2}-2\sqrt{2m+1}+\sqrt{2m} )^2.
\end{eqnarray*}
Note that $C_\beta^+ \ge C_\beta^-$, and it follows that $\eta(t) =
\eta^+(t) - \eta^-(t)$ is nonnegative, and strictly positive if $\kappa
\neq0$.
\end{pf*}

For a particular example, we consider a mean-zero Gaussian process $\{
F_t,\break t\ge0\}$, with covariance given by
\[
{\mathbb E} [ F_rF_t ] = \sqrt{rt}\sin^{-1}
\biggl(\frac
{r\wedge t}{\sqrt{rt}} \biggr).
\]
This process was studied by Swanson in a 2007 paper~\cite{Swanson07},
and it appears in the limit of normalized empirical quantiles of a
system of independent Brownian motions.
%
\begin{corollary}\label{co5.6}
The process $\{F_t, 0\le t\le T\}$, with the
covariance described above, satisfies the conditions of Section~\ref{sec4},
with $\eta(t) = (C_\beta^+ - C_\beta^- )t^2$, where
$C_\beta^+$, $C_\beta^-$ are as given in Proposition~\ref{pr5.4}, with $\kappa
^2 = 1/4$.
\end{corollary}
\begin{pf}
Assume $0 \le r < t\le T$. We can write
\[
\sqrt{rt}\sin^{-1} \biggl(\sqrt{\frac rt} \biggr) = \sqrt{rt}
\tan^{-1} \biggl(\sqrt{\frac{r}{t-r}} \biggr) = r\phi\biggl(
\frac tr \biggr),
\]
where
%
\begin{equation}
\label{arctanphi}\phi(x) = \cases{\displaystyle \sqrt{x}\tan^{-1}
\biggl(\frac{1}{\sqrt{x-1}} \biggr),&\quad if $x>1$,
\cr
\displaystyle \frac{\pi}{2}, &\quad if $x=1$.}
\end{equation}

Condition $(\phi.1)$ is clear by continuity and L'H\^opital. Conditions
$(\phi.2)$ and $(\phi.3)$ are easily verified by differentiation. For
$(\phi.4)$ we can write,
\[
\phi'(x) = -\frac{1}{2\sqrt{x-1}} +\frac{1}{2\sqrt{x}} \biggl(
\frac
{\sqrt{x}-1}{\sqrt{x-1}}-\tan^{-1} \biggl(\frac{1}{\sqrt{x-1}} \biggr)
\biggr)
\]
so that $\kappa= -1/2$, and
\[
\psi(x) = \frac12 \biggl(\frac{\sqrt{x}-1}{\sqrt{x-1}}-\tan^{-1} \biggl(
\frac{1}{\sqrt{x-1}} \biggr) \biggr)
\]
satisfies $(\phi.4)$.
\end{pf}

\subsection{Empirical quantiles of independent Brownian motions}\label{sec5.3}
For our last example, we consider a family of processes studied by
Swanson in~\cite{Swanson10}. Similar to~\cite{Swanson07}, this
Gaussian family arises from the empirical quantiles of independent
Brownian motions, but this case is more general, and does not have a
covariance representation (\ref{CovFdef}).

Let $B = \{B(t), t\ge0\}$ be a Brownian motion with random initial
position. Assume $B(0)$ has a density function $f \in{\cal C}^\infty
({\mathbb R})$ such that
\[
\sup_{x\in{\mathbb R}} \bigl(1+|x|^m\bigr)\bigl|f^{(n)}(x)\bigr| <
\infty
\]
for all nonnegative integers $m$ and $n$. It follows that for $t > 0$,
$B$ has density
\[
u(x,t) = \int_{\mathbb R} f(y)p(t,x-y)\,dy,
\]
where $p(t,x) = (2\pi t)^{-1/2}e^{-{{x^2}/({2t})}}$. For fixed
$\alpha\in(0,1)$, define the $\alpha$-quantile $q(t)$ by
\[
\int_{-\infty}^{q(t)} u(x,t)\,dx = \alpha,
\]
where we assume $f(q(0)) >0$. It is proved in~\cite{Swanson10} (Theorem
1.4) that there exists a continuous, centered Gaussian process $\{F(t),
t\ge0\}$ with covariance
%
\begin{equation}
\label{rhodef} {\mathbb E} [ F_r F_t ] =
\rho(r,t) = \frac
{{\mathbb P} ( B(r) \le q(r), B(t) \le q(t) ) -\alpha^2}{u(q(r),r)u(q(t),t)}.
\end{equation}

In~\cite{Swanson10}, the properties of $\rho$ are studied in detail,
and we follow the notation and proof methods given in Section 3 of that
paper. Swanson defines the following factors:
\[
\tilde{\rho}(r,t)={\mathbb P} \bigl( B(r) \le q(r), B(t) \le q(t) \bigr
) -
\alpha^2 \quad\mbox{and}\quad \theta(t)= \bigl(u\bigl(q(t),t\bigr)
\bigr)^{-1}
\]
so that $\rho(r,t) = \theta(r)\theta(t)\tilde{\rho}(r,t)$. For
fixed $T >0$ and $0
< r < t \le T$, the first partial derivatives of $\tilde{\rho}$ are
calculated in~\cite{Swanson10} [see equations (3.4), (3.7)]
%
\begin{eqnarray}
\frac{\partial}{\partial t} \tilde{\rho}(r,t)&=& q'(t)\int
_{-\infty
}^{q(r)}p\bigl(t-r,x-q(t)\bigr)u(x,r)\,dy\,dx
\nonumber
\\
&&{}-\frac12
p\bigl(t-r,q(r)-q(t)\bigr)u\bigl(q(r),r\bigr)\nonumber\\[-8pt]\\[-8pt]
&&{} + u\bigl(q(r),r
\bigr)q'(r)\int_{-\infty
}^{q(t)} p
\bigl(t-r,q(r)-y\bigr)\,dy
\nonumber\\
&&{}+ \frac12\int_{-\infty}^{q(t)}\int_{-\infty}^{q(r)}
p(t-r,x-y)\,\frac{\partial^2}{\partial x^2}u(x,r)\,dx\,dy;
\nonumber
\\
\label{rhosqr}\frac{\partial}{\partial r} \tilde{\rho}(r,t)&=&
\frac12 p\bigl(t-r,q(t)-q(r)\bigr)u\bigl(q(r),r\bigr).
\end{eqnarray}

\begin{lemma}\label{le5.7}
Let $0 < T$, and $0< r <t \le T$. Then there exist
constants $C_i, i=1,2,3,4$, such that:
\begin{longlist}[(a)]
\item[(a)]
\[
\biggl\llvert\frac{\partial}{\partial r} \rho(r,t)\biggr\rrvert\le C_1
|t-r|^{-1/2};
\]
\item[(b)]
\[
\biggl\llvert\frac{\partial^2}{\partial r^2} \rho(r,t)\biggr\rrvert
\le C_2
|t-r|^{-{3/2}};
\]
\item[(c)]
\[
\biggl\llvert\frac{\partial}{\partial t} \rho(r,t)\biggr\rrvert\le C_3
|t-r|^{-{1/2}};
\]
\item[(d)]
\[
\biggl\llvert\frac{\partial^2}{\partial t^2} \rho(r,t)\biggr\rrvert
\le C_4
|t-r|^{-{3/2}}.
\]
\end{longlist}
\end{lemma}
\begin{pf}
Results (a) and (c) are proved in Theorem 3.1 of~\cite{Swanson10}.
Bounds for (b) and (d) follow by differentiating the expressions for
$\partial_r\rho(r,t)$ and $\partial_t \rho(r,t)$ given in the proof of
that theorem.
\end{pf}
%
\begin{proposition}\label{pr5.8} Let $T>0$, $0< s < T\wedge1$ and $s \le r \le t
\le T$. Then $\rho(r,t)$ satisfies conditions
\textup{(i)}--\textup{(iv)} of Section~\ref{sec4}.\vadjust{\goodbreak}
\end{proposition}
\begin{pf}
Conditions (i) and (ii) are proved in~\cite{Swanson10} (Corollaries
3.2, 3.5 and Remark 3.6). For condition (iii), there are several cases
to consider.

\textit{Case} 1: $s\le r\le t- 2s$. Using Lemma~\ref{le5.7}(a),
\begin{eqnarray*}
\bigl\llvert{\mathbb E} \bigl[ F_t (F_{r+s}-2F_r+F_{r-s})
\bigr]\bigr\rrvert&\le&\bigl\llvert\rho(r+s,t) -\rho(r,t)\bigr\rrvert
+\bigl
\llvert\rho(r,t) -\rho(r-s,t)\bigr\rrvert
\\
&\le&\int_0^s \biggl\llvert
\frac{\partial}{\partial r}\rho(r+x,t)\biggr\rrvert\,dx+\int_{-s}^0
\biggl\llvert\frac{\partial}{\partial r}\rho(r+y,t)\biggr\rrvert\,dy
\\
&\le&2\int_0^s C_1|t-r-x|^{-{1/2}}\,dx
\le Cs^{1/2}.
\end{eqnarray*}

\textit{Case} 2: If $|t-r| < 2s$, the computation is similar to
case 1, where we use the fact that
\[
\int_0^s x^{-{1/2}}\,dx =
2s^{1/2}.
\]

\textit{Case} 3: For $r, t \ge2s$ and $|t-r|\ge2s$, the
results follow from Lemma~\ref{le5.7}(b) and (d) for $r<t$ and $r>t$, respectively.

Now to condition (iv). For the first part, we first assume $t \ge2s$.
Then using the above decomposition,
\begin{eqnarray*}
{\mathbb E} \bigl[ F_t (F_{t+s}-F_{t-s}) \bigr]
&=& \rho(t,t+s) - \rho(t,t-s)
\\
&=& \theta(t) \bigl[ \theta(t+s)\tilde{\rho}(t,t+s) - \theta(t-s)\tilde{
\rho}(t,t-s) \bigr]
\\
&=& \theta(t) \bigl[ \tilde{\rho}(t,t+s)\Delta\theta+ \theta(t-s)\Delta
\tilde{
\rho} \bigr],
\end{eqnarray*}
where $\Delta\theta= \theta(t)-\theta(t-s)$ and $\Delta\tilde
{\rho} =
\tilde{\rho}(t,t+s)-\tilde{\rho}(t,t-s)$.
First, note that
\[
\bigl\llvert u'\bigl(q(t),t\bigr)\bigr\rrvert= \biggl\llvert
\frac{\partial}{\partial
x}u\bigl(q(t),t\bigr)q'(t) + \frac{\partial}{\partial t}u
\bigl(q(t),t\bigr)\biggr\rrvert\le C,
\]
where we used that $q'(t)$ is bounded; see Lemma 1.1 of \cite
{Swanson10}. Since $u(q(t),t)$ is continuous and strictly positive on
$[0,T]$, it follows that $\theta(t)$ is bounded and
%
\begin{equation}
\label{thetadiff} \bigl\llvert\theta'(t)\bigr\rrvert
= \frac
{|u'(q(t),t)|}{u^2(q(t),t)} \le C,
\end{equation}
hence,
\[
\llvert\Delta\theta\rrvert\le\int_{-s}^s \bigl|
\theta'(t+x)\bigr|\,dx \le Cs.
\]
For $\Delta\tilde{\rho}$, we have
\begin{eqnarray*}
\llvert\Delta\tilde{\rho}\rrvert&=& \bigl\llvert{\mathbb P} \bigl(
B(t)\le
q(t), B(t+s)\le q(t+s) \bigr)\\
&&\hspace*{4pt}{} - {\mathbb P} \bigl( B(t)\le q(t),
B(t-s)\le q(t-s)
\bigr)\bigr\rrvert
\\
&=& \int_{-\infty}^{q(t)}\int_{-\infty}^{q(t+s)}
p(s,x-y)u(x,t)\,dy\,dx \\
&&{}- \int_{-\infty}^{q(t-s)}\int
_{-\infty}^{q(t)} p(s,x-y)u(x,t-s)\,dy\,dx
\\
&\le&\biggl\llvert\int_{-\infty}^{q(t-s)}\int
_{-\infty}^{q(t)} p(s,x-y)u(x,t) - p(s,x-y)u(x,t-s)\,dy\,dx
\biggr\rrvert+ Cs
\\
&\le&\int_{-\infty}^{q(t-s)} \bigl|u(x,t-s) - u(x,t)\bigr|\,dx+Cs
\\
&\le&\int_{-\infty}^\infty\biggl\llvert\int
_{t-s}^t \frac{\partial
}{\partial
r} u(x,r)\,dr\biggr\rrvert
\,dx + Cs \\
&=& \frac12 \int_{-\infty}^\infty\biggl\llvert
\int_{t-s}^t \frac{\partial^2}{\partial x^2}u(x,r)\,dr\biggr
\rrvert\,dx + Cs
\\
&\le&\frac12 \int_{-\infty}^\infty\int
_{-\infty}^\infty\int_{t-s}^t
\bigl\llvert f''(y)\bigr\rrvert p(r,x-y)\,dr\,dy\,dx +
Cs \le Cs.
\end{eqnarray*}
When $t < 2s$, we write
\begin{eqnarray*}
\bigl\llvert{\mathbb E} \bigl[ F_t(F_{t+s}-F_{t-s})
\bigr]\bigr\rrvert&\le&\bigl\llvert\rho(t,t+s) - \rho(t,t)\bigr
\rrvert+ \bigl
\llvert\rho(t,t) - \rho(t-s,t)\bigr\rrvert
\\
&\le&\int_0^s \biggl\llvert
\frac{\partial}{\partial t}\rho(t,t+x)\biggr\rrvert\,dx + \int_{-s}^0
\biggl\llvert\frac{\partial}{\partial r} \rho(t+y,t)\biggr\rrvert\,dy
\\
&\le& Cs^{1/2},
\end{eqnarray*}
using Lemma~\ref{le5.7} and the fact that
\[
\int_0^s x^{-{1/2}}\,dx =
2s^{1/2}.
\]

For the second part of condition (iv), we consider
\[
\bigl\llvert{\mathbb E} \bigl[ F_r(F_{t+s} -
F_{t-s}) \bigr]\bigr\rrvert\quad\mbox{and}\quad \bigl\llvert{\mathbb E}
\bigl[ F_s(F_t - F_{t-s}) \bigr]\bigr\rrvert.
\]
When $r<t-s$ (including $r=s$), an upper bound of $Cs|t-r|^{-{1/2}}$
is proved in~\cite{Swanson10}; see Corollary 3.4 and Remark 3.6. When
$r \ge t+2s$, or $|t-r| < 2s$, the bounds follow from Lemma~\ref{le5.7}.
\end{pf}

The rest of this section is dedicated to verifying condition (v). We
start with two useful estimates. As in Proposition~\ref{pr5.8}, suppose $0<s\le
r \le t \le T$. It follows from Lemma 1.1 of~\cite{Swanson10} that for
some positive constant $C$,
%
\begin{equation}
\label{qdiff} \bigl|q(t) - q(r)\bigr| \le C(t-r).
\end{equation}
Using this estimate and the fact that $e^{-a} - e^{-b} \le b-a$ for
$0\le a\le b$, we obtain
%
\begin{equation}
\label{expqdiff} \bigl\llvert e^{-
{(q(t)-q(r))^2}/({2(t-r)})} -
e^{-{(q(t)-q(r-s))^2}/({2(t-r+s)})} \bigr\rrvert\le Cs \le1.
\end{equation}

Recalling the definitions in condition (v), we can write for $t \in[0,T]$,
\begin{eqnarray*}
\eta_n^+(t) - \eta_n^-(t) 
&=& \sum_{\ell=1}^{2 \lfloor{nt}/{2} \rfloor}
\beta_n(\ell-1,\ell-1)^2 + 2\sum
_{k\le j-1} \beta_n(2k-1,2j-1)^2 \\
&&{} + 2\sum
_{k\le j-1} \beta_n(2k-2,2j-2)^2
- 2\sum_{k\le j-1} \beta_n(2k-2,2j-1)^2\\
&&{}
- 2\sum_{k\le j-1} \beta_n(2k-1,2j-2)^2.
\end{eqnarray*}
For the first sum, since $F_{{\ell}/{n}} - F_{({\ell
-1})/{n}}$ is
Gaussian, we have
\[
\beta_n(\ell-1,\ell-1)^2 = \bigl( {\mathbb E} \bigl[
(F_{{\ell}/{n}} - F_{({\ell-1})/{n}} )^2 \bigr]
\bigr)^2 = \tfrac13 {\mathbb E} \bigl[ (F_{{\ell}/{n}} -
F_{({\ell-1})/{n}} )^4 \bigr].
\]
By Theorem 3.7 of~\cite{Swanson10},
\[
\sum_{\ell=1}^{\lfloor nt\rfloor} (F_{{\ell}/{n}} -
F_{({\ell-1})/{n}} )^4 \longrightarrow\frac6\pi\int
_0^t \bigl(u\bigl(q(s),s\bigr)
\bigr)^{-2}\,ds
\]
in $L^2$ as $n \to\infty$. For the second sum, assume $1\le k<j$, and
we study the term
\begin{eqnarray*}
&&
\beta_n(2k-1,2j-1) \\
&&\qquad= \rho\biggl(\frac{2k}{n},
\frac{2j}{n} \biggr) - \rho\biggl(\frac{2k-1}{n}, \frac{2j}{n}
\biggr) - \rho\biggl(\frac{2k}{n}, \frac{2j-1}{n} \biggr) + \rho
\biggl(\frac{2k-1}{n}, \frac
{2j-1}{n} \biggr)
\\
&&\qquad= \theta\biggl(\frac{2j}{n} \biggr)\int_{(2k-1)/{n}}^{{2k}/{n}}
\biggl[\theta'(r)\tilde{\rho} \biggl(r,\frac{2j}{n} \biggr)
+ \theta(r)\partial_r \tilde{\rho} \biggl(r,\frac{2j}{n}
\biggr) \biggr]\,dr
\\
&&\qquad\quad{} -\theta\biggl(\frac{2j-1}{n} \biggr)\int_{(2k-1)/{n}}^{{2k}/{n}}
\biggl[\theta'(r)\tilde{\rho} \biggl(r,\frac{2j-1}{n} \biggr)
+ \theta(r)\partial_r \tilde{\rho} \biggl(r,\frac{2j-1}{n}
\biggr) \biggr]\,dr.
\end{eqnarray*}
We can write this as
%
\begin{eqnarray}
\label{betana}\qquad
&&\theta\biggl(\frac{2j}{n} \biggr)\int_{(2k-1)/{n}}^{{2k}/{n}}
\theta(r) \biggl(\partial_r\tilde{\rho} \biggl(r,\frac{2j}{n}
\biggr) -\partial_r \tilde{\rho} \biggl(r,\frac{2j-1}{n}
\biggr) \biggr)\,dr
\\
\label{betanb}
&&\qquad{} + \biggl[ \theta\biggl(\frac{2j}{n} \biggr)-\theta\biggl(
\frac
{2j-1}{n} \biggr) \biggr]\int_{(2k-1)/{n}}^{{2k}/{n}}
\theta(r) \biggl(\partial_r\tilde{\rho} \biggl(r,\frac{2j-1}{n}
\biggr) \biggr)\,dr
\\
\label{betanc}
&&\qquad{} + \int_{(2k-1)/{n}}^{{2k}/{n}} \theta'(r)
\biggl[ \theta\biggl(\frac{2j}{n} \biggr)\tilde{\rho} \biggl(r,
\frac{2j}{n} \biggr) - \theta\biggl(\frac{2j-1}{n} \biggr)\tilde{
\rho} \biggl(r,\frac
{2j-1}{n} \biggr) \biggr]\,dr.
\end{eqnarray}
The first task is to show that components (\ref{betanb}) and (\ref
{betanc}) have a negligible contribution to $\eta(t)$. For (\ref
{betanb}), it follows from (\ref{thetadiff}) that
%
\begin{equation}
\label{Deltatheta}\biggl\llvert\theta\biggl(\frac
{2j}{n}
\biggr) - \theta\biggl(\frac{2j-1}{n} \biggr)\biggr\rrvert\le
Cn^{-1},
\end{equation}
and using (\ref{rhosqr}), we have
\begin{eqnarray*}
&&
\int_{(2k-1)/{n}}^{{2k}/{n}} \theta(r)\partial_r
\tilde{\rho} \biggl(r,\frac{2j-1}{n} \biggr)\,dr \\
&&\qquad= \int_{(2k-1)/{n}}^{{2k}/{n}}
p \biggl(\frac{2j-1}{n}-r,q \biggl(\frac{2j-1}{n} \biggr)-q(r)
\biggr)\,dr \\
&&\qquad\le Cn^{-{1/2}}.
\end{eqnarray*}
Hence, the contribution of (\ref{betanb}) to the sum of $\beta
_n(2k-1,2j-1)^2$ is bounded by $C (n^{-{3/2}} )^2\cdot n^2
\le Cn^{-1}$.
We can write component (\ref{betanc}) as
\begin{eqnarray*}
&&
\int_{(2k-1)/{n}}^{{2k}/{n}} \theta'(r) \biggl[
\theta\biggl(\frac
{2j}{n} \biggr) \biggl( \tilde{\rho} \biggl( r,
\frac{2j}{n} \biggr)-\tilde{\rho} \biggl( r, \frac{2j-1}{n} \biggr)
\biggr) \\
&&\qquad\hspace*{45pt}{}+ \biggl(\theta\biggl(\frac
{2j}{n} \biggr) - \theta\biggl(
\frac{2j-1}{n} \biggr) \biggr)\tilde{\rho} \biggl(r, \frac{2j-1}{n}
\biggr) \biggr]\,dr.
\end{eqnarray*}
Using (\ref{rhosqr}), we have for each $r\in[\frac
{2k-1}{n},\frac
{2k}{n} ]$,
\[
\biggl\llvert\tilde{\rho} \biggl( r, \frac{2j}{n} \biggr)-\tilde{\rho}
\biggl( r, \frac{2j-1}{n} \biggr)\biggr\rrvert\le Cn^{-
1/2}(2j-2k-1)^{-{1/2}}.
\]
Then, using (\ref{Deltatheta}) and (\ref{thetadiff}), we have (\ref
{betanc}) bounded by
\[
C \bigl[ n^{-{1/2}}(2j-2k-1)^{-{1/2}} + n^{-1}
\bigr]n^{-1}.
\]
Hence, the contribution of (\ref{betanc}) to the sum of $\beta
_n(2k-1,2j-1)^2$ is bounded~by
\[
Cn^{-2}\sum_{j=1}^{ \lfloor{nt/2} \rfloor}\sum
_{k=1}^{j-1} \bigl[ n^{-1}(2j-2k-1)^{-1}
+ n^{-2} \bigr]\le Cn^{-1}.
\]

We now turn to component (\ref{betana}). By (\ref{rhosqr}),
\[
\theta(r)\frac{\partial}{\partial r}\tilde{\rho} \biggl(r, \frac
{2j}{n} \biggr) =
\frac12 p \biggl(\frac{2j}{n}-r,q \biggl(\frac
{2j}{n} \biggr)-q(r)
\biggr).
\]
To simplify notation, define
\[
\psi_n(j,r) = e^{-{(q({j}/{n})-q(r))^2}/({2({j}/{n}-r)})}.
\]
By (\ref{qdiff}), we have for the interval $I_{2k} = [\frac
{2k-1}{n}, \frac{2k}{n} ]$,
\[
\sup_{r\in I_{2k}} \biggl\{ \frac{ ((q ({2j}/{n}
)-q(r) )^2}{2({2j}/{n}-r)} \biggr\} \le
\frac{C(2j-2k+1)}{n}.\vadjust{\goodbreak}
\]
This implies that $\inf\{\psi_n(2j,r), r\in I_{2k} \} \ge
e^{-C({2j-2k+1})/{n}}$, and hence when $j,k$ are small compared to
$n$, $|\psi|$ is close to unity. We can write
%
\begin{eqnarray}
\label{betand}
&&
\int_{(2k-1)/{n}}^{{2k}/{n}}
\theta(r) \biggl(\partial_r\tilde{\rho} \biggl(r,
\frac{2j}{n} \biggr) - \partial_r\tilde{\rho} \biggl(r,
\frac{2j-1}{n} \biggr) \biggr)\,dr \nonumber\\[-8pt]\\[-8pt]
&&\qquad= \frac{1}{2\sqrt{2\pi}}\int
_{(2k-1)/{n}}^{{2k}/{n}} \frac
{1}{\sqrt{{2j}/{n} - r}} -
\frac{1}{\sqrt{({2j-1})/{n} -
r}}\,dr\nonumber
\\
\label{betane}
&&\qquad\quad{} -\frac{1}{2\sqrt{2\pi}}\int
_{(2k-1)/{n}}^{{2k}/{n}} \bigl(1-\psi_n(2j-1,r)
\bigr) \nonumber\\[-8pt]\\[-8pt]
&&\hspace*{83pt}\qquad\quad{}\times\biggl(\frac{1}{\sqrt{{2j}/{n} - r}} - \frac
{1}{\sqrt{({2j-1})/{n} - r}} \biggr)\,dr
\nonumber\\
\label{betanf}
&&\qquad\quad{} +\frac{1}{2\sqrt{2\pi}}\int
_{(2k-1)/{n}}^{{2k}/{n}} \frac
{\psi_n(2j,r)-\psi_n(2j-1,r)}{\sqrt{{2j}/{n}-r}}\,dr.
\end{eqnarray}
For component (\ref{betane}), by the above estimate for $\inf
\{
\psi_n(2j,r), r\in I_{2k} \}$, we have
\[
\sup_{r\in I_{2k}} \bigl\llvert1-\psi(2j,r)\bigr\rrvert\le
Cn^{-1}(2j-2k+1) \le1,
\]
hence (\ref{betane}) is bounded by
\[
Cn^{-{3/2}}(2j-2k+1) ( \sqrt{2j-2k+1} - 2\sqrt{2j-2k} + \sqrt{2j-2k-1}
).
\]
Given $\varepsilon> 0$, we can find an $M>1$ such that
\[
\sum_{m=M}^{\infty} ( \sqrt{2m+1} - 2
\sqrt{2m} + \sqrt{2m-1} )^2 < \varepsilon.
\]
The contribution of (\ref{betane}) to the sum of $\beta_n(2k-1,2j-1)^2$
is thus bounded by
\begin{eqnarray*}
&&(2\pi n)^{-1}\sum_{j=1}^{ \lfloor{nt/2} \rfloor}
\theta^2 \biggl(\frac{2j}{n} \biggr)\sum
_{k=1}^{j-1} \sup_{r\in I_{2k}}\bigl(1-
\psi_n(2j,r)\bigr)^2\\
&&\qquad\hspace*{96.5pt}{}\times ( \sqrt{2j-2k+1} - 2\sqrt{2j-2k} +
\sqrt{2j-2k-1} )^2
\\
&&\qquad\le Cn^{-1}\sum_{j=1}^{ \lfloor{nt/2} \rfloor}
\sum_{k=1}^{j-M-1} ( \sqrt{2j-2k+1} - 2
\sqrt{2j-2k} + \sqrt{2j-2k-1} )^2
\\
&&\qquad\quad{} + Cn^{-1}\sum_{j=1}^{ \lfloor{nt/2} \rfloor}
\sum_{k=j-M}^{j-1} Cn^{-1}(2j-2k+1)
\\
&&\qquad\quad\hspace*{93.7pt}{}\times( \sqrt{2j-2k+1} - 2\sqrt{2j-2k} +
\sqrt{2j-2k-1} )^2
\\
&&\qquad\le Cn^{-1}\sum_{j=1}^{ \lfloor{nt/2} \rfloor}
\varepsilon
+ Cn^{-1}\sum
_{j=1}^{ \lfloor{nt/2} \rfloor} \frac
{M^2}{n^2},
\end{eqnarray*}
which is less than $C\varepsilon$ as $n \to\infty$, since $\theta(t)$
is bounded.

For (\ref{betanf}), by we have $\sup\{\llvert\psi_n(2j,r)-\psi
_n(2j-1,r)\rrvert, r\in I_{2k} \} \le Cn^{-1}$, and hence (\ref
{betanf}) is bounded by $Cn^{-{3/2}}(2j-2k-1)^{-{1/2}}$.
Therefore the contribution of the term including (\ref{betanf}) to the
sum of $\beta_n(2k-1,2j-1)^2$ is bounded by
\[
Cn^{-3}\sum_{j=1}^{ \lfloor{nt/2} \rfloor} \sum
_{k=1}^{j-1}(2j-2k-1)^{-1} \le
Cn^{-2}\log(nt),
\]
because $\theta(t)$ is bounded.

It follows that the sum of $\beta_n(2k-1,2j-1)^2$ is dominated by
(\ref{betana}), and the significant term in (\ref{betana}) is
given by
(\ref{betand}). Hence, it is enough to consider
\[
\frac{2}{n\pi} \sum_{j \le k-1} \theta^2
\biggl(\frac{2j}{n} \biggr) ( \sqrt{2j-2k+1}-2\sqrt{2j-2k}+\sqrt{2j-2k-1}
)^2.
\]
Using the change of index $j = k+m$, this is
\[
\frac{2}{n\pi}\sum_{j=1}^{ \lfloor{nt/2} \rfloor}
\theta^2 \biggl(\frac{2j}{n} \biggr) \sum
_{m=1}^{j-1} ( \sqrt{2m+1}-2\sqrt{2m}+\sqrt{2m-1}
)^2.
\]
Taking $n\to\infty$, this behaves like
\[
\frac{a}{\pi} \int_0^t
\theta^2(s)\,ds,
\]
where
\[
a = \sum_{m=1}^\infty( \sqrt{2m+1}-2
\sqrt{2m}+\sqrt{2m-1} )^2.
\]
By similar computation,
\begin{eqnarray*}
\sum_{k \le j-1} \beta_n(2k-2,2j-2)^2
&\longrightarrow&\frac{a}{\pi} \int_0^t
\theta^2(s)\,ds,
\\
\sum_{k \le j-1} \beta_n(2k-2,2j-1)^2
&\longrightarrow&\frac
{b_1}{\pi} \int_0^t
\theta^2(s)\,ds
\end{eqnarray*}
and
\[
\sum_{k \le j-1} \beta_n(2k-1,2j-2)^2
\longrightarrow
\frac{b_2}{\pi} \int_0^t
\theta^2(s)\,ds,
\]
where
\begin{eqnarray*}
b_1 &=& \sum_{m=1}^\infty(
\sqrt{2m+2}-2\sqrt{2m+1}+\sqrt{2m} )^2,
\\
b_2 &=& \sum_{m=1}^\infty(
\sqrt{2m}-2\sqrt{2m-1}+\sqrt{2m-2} )^2.
\end{eqnarray*}
We have proved the following result:
%
\begin{proposition}\label{pr5.9} Under the above assumptions,
$\rho(r,t)$ satisfies
condition~\textup{(v)} of Section~\ref{sec4}, where
\[
\eta(t) = \frac{2+4a-2b_1-2b_2}{\pi} \int_0^t \bigl(u
\bigl(q(s),s\bigr) \bigr)^{-2}\,ds.
\]
\end{proposition}

The coefficient $2+ 4a - 2b_1 - 2b_2$ is approximately 1.3437, while
$u(q(t),t)$ depends on $f$ and $\alpha$.

\section{Proof of the technical lemmas}\label{sec6}
We begin with two technical lemmas. The first is a version of Corollary
\ref{co4.2} with disjoint intervals.

%
\begin{lemma}\label{le6.1}
For $0 \le t_0 < t_1 \le t_2 < t_3 \le T$,
\[
\lim_{n \to\infty} \sum_{j= \lfloor{nt_0}/{2}
\rfloor+1}^{\lfloor{nt_1}/{2}\rfloor}\sum
_{k=\lfloor{nt_2}/{2}\rfloor
+1}^{\lfloor{nt_3}/{2} \rfloor}\bigl\llvert\bigl\langle
\partial_{(2j-1)/{n}}^{\otimes2} - \partial_{({2j-2})/{n}}^{\otimes2},
\partial_{(2k-1)/{n}}^{\otimes2} - \partial_{(2k-2)/{n}}^{\otimes2} \bigr\rangle_{\hten^{\otimes2}} \bigr\rrvert=0.
\]
\end{lemma}
\begin{pf} We may assume $t_0 = 0$ and $t_1 = t_2$. Observe that
\begin{eqnarray*}
&&
\bigl\langle\partial_{({2j-1})/{n}}^{\otimes2} - \partial_{(2j-2)/{n}}^{\otimes2}, \partial_{(2k-1)/{n}}^{\otimes2} -
\partial_{(2k-2)/{n}}^{\otimes2} \bigr\rangle_{\hten^{\otimes2}}
\\
&&\qquad= \beta_n(2j-1,2k-1)^2 -\beta_n(2j-1,2k-2)^2
- \beta_n(2j-2,2k-1)^2\\
&&\qquad\quad{} +\beta_n(2j-2,2k-2)^2.
\end{eqnarray*}
Therefore, it is enough to show that
%
\begin{equation}
\label{betaX2} \sum_{j=0}^{\lfloor nt_2 \rfloor}
\sum_{k=\lfloor nt_2
\rfloor+1}^{\lfloor nt_3 \rfloor} \beta_n(j,k)^2
\le Cn^{-\varepsilon}
\end{equation}
for some $\varepsilon> 0$. We can decompose the sum in (\ref
{betaX2}) as
\[
\sum_{k=\lfloor nt_2\rfloor+1}^{\lfloor nt_3 \rfloor} \beta_n(0,k)^2
+ \sum_{k=\lfloor nt_2\rfloor+1}^{\lfloor nt_3 \rfloor} \beta_n
\bigl(\lfloor nt_2\rfloor,k\bigr)^2 + \sum
_{j=1}^{\lfloor nt_2 \rfloor-1} \sum_{k=\lfloor nt_2 \rfloor
+1}^{\lfloor nt_3 \rfloor}
\beta_n(j,k)^2.
\]
By condition (iv), for some $\gamma> 0$ we have
\begin{eqnarray*}
\sum_{k=\lfloor nt_2\rfloor+1}^{\lfloor nt_3 \rfloor} \beta_n(0,k)^2
&\le&\sup_{1\le j \le\lfloor nt_3 \rfloor} \bigl|\beta_n(0,k)\bigr| \sum
_{k=\lfloor nt_2\rfloor+1}^{\lfloor nt_3\rfloor} \bigl|\beta_n(0,k)\bigr|
\\[-2pt]
& \le &Cn^{-1}\sum_{k=\lfloor nt_2\rfloor+2}^{\lfloor nt_3\rfloor
}(k-1)^{-\gamma}
+Cn^{-1} \le Cn^{-\gamma}.
\end{eqnarray*}
By condition (ii), for some $1 < \alpha\le\frac{3}{2}$,
\begin{eqnarray*}
\sum_{k=\lfloor nt_2\rfloor+1}^{\lfloor nt_3 \rfloor} \beta_n\bigl(
\lfloor nt_2\rfloor,k\bigr)^2 &\le&\beta_n
\bigl(\lfloor nt_2 \rfloor, \lfloor nt_2 \rfloor+1
\bigr)^2 + Cn^{-1}\sum_{k=\lfloor nt_2\rfloor+2}^{\lfloor nt_3
\rfloor}
\beta_n\bigl(\lfloor nt_2\rfloor,k\bigr)
\\[-2pt]
&\le &Cn^{-1} +Cn^{-1}\sum_{k=\lfloor nt_2\rfloor+1}^{\lfloor nt_3
\rfloor}
\bigl(k-\lfloor nt_2\rfloor\bigr)^{-\alpha} \le
Cn^{-1}
\end{eqnarray*}
and again by condition (ii), for $\beta= \frac{3}{2}-\alpha$,
\begin{eqnarray*}
&&
\sum_{j=1}^{\lfloor nt_2 \rfloor-1} \sum
_{k=\lfloor nt_2 \rfloor
+1}^{\lfloor nt_3 \rfloor} \beta_n(j,k)^2 \\[-2pt]
&&\qquad
\le Cn^{-1}\sum_{j=1}^{\lfloor nt_2 \rfloor-1} \sum
_{k=\lfloor nt_2
\rfloor+1}^{\lfloor nt_3 \rfloor} \bigl[ \bigl(k-\lfloor
nt_2\rfloor\bigr)^{-\alpha}j^{-\beta} +
(k-j)^{-{3/2}} \bigr]
\\[-2pt]
&&\qquad
\le Cn^{-1} \Biggl(\sum_{k=1}^{\lfloor nt_3\rfloor}
k^{-\alpha
} \Biggr) \Biggl(\sum_{j=1}^{\lfloor nt_2 \rfloor}
j^{-\beta} \Biggr) +Cn^{-1}\sum_{j=1}^{\lfloor nt_2 \rfloor}
\bigl(\lfloor nt_2 \rfloor- j\bigr)^{-
{1/2}}
\\[-2pt]
&&\qquad\le Cn^{-\beta} + Cn^{-{1/2}};
\end{eqnarray*}
hence the sum is bounded by $Cn^{-\varepsilon}$ for $\varepsilon=
\min
\{\beta, \gamma, \frac{1}{2} \}$.
\end{pf}
%
\begin{lemma}\label{le6.2}
For $0 \le t \le T$ and integer $j \ge1$,
\[
\bigl\llvert\langle\varepsilon_t, \partial_{j/n}
\rangle_\hten\bigr\rrvert\le Cn^{-{1/2}}
\]
for a positive constant $C$ which depends on $T$.
\end{lemma}
\begin{pf} By conditions (i) and (ii), we have for $j\ge1$ and $t > 0$,
%
\begin{eqnarray}
\label{cond2cor}
\bigl\llvert\langle\varepsilon_t, \partial_{j/n}
\rangle_\hten\bigr\rrvert&\le&\sum_{k=0}^{\lfloor nt \rfloor-1 }
\bigl\llvert\langle\partial_{{k}/{n}}, \partial_{j/n}
\rangle_\hten\bigr\rrvert+ \bigl\llvert\langle\varepsilon_t
- \varepsilon_{\lfloor nt \rfloor}, \partial_{{j}/{n}} \rangle_\hten
\bigr\rrvert
\nonumber\\[-8pt]\\[-8pt]
&\le& C\sum_{k=0}^{\infty}
n^{-{1/2}} \bigl( |j-k|^{-\alpha} \wedge1 \bigr)+O
\bigl(n^{-{1/2}}\bigr) \le Cn^{-{1/2}}.\nonumber
\end{eqnarray}
\upqed
\end{pf}

\subsection{\texorpdfstring{Proof of Lemma \protect\ref{le4.4}}{Proof of Lemma 4.4}}\label{sec6.1}
By the Lagrange theorem for the Taylor expansion remainder, the terms
$R_0(W_{{2j}/{n}}), R_1(W_{({2j-2})/{n}})$ can be expressed in
integral form,
\[
R_0(W_{{2j}/{n}}) = \frac{1}{2} \int_{W_
{({2j-1})/{n}}}^{W_{{2j}/{n}}}
(W_{{2j}/{n}} - u)^2 f^{(3)}(u) \,du
\]
and
\[
R_1(W_{({2j-2})/{n}}) = -\frac{1}{2} \int
_{W_{(2j-2)/{n}}}^{W_{({2j-1})/{n}}} (W_{({2j-2})/{n}} - u)^2
f^{(3)}(u) \,du.
\]
After a change of variables, we obtain
\begin{eqnarray*}
&&
R_0(W_{{2j}/{n}}) \\
&&\qquad= \frac{1}{2} (W_{{2j}/{n}} -
W_{(2j-1)/{n}})^3 \int_0^1
v^2 f^{(3)}\bigl(vW_{({2j-1})/{n}} + (1-v)W_{{2j}/{n}}
\bigr) \,dv
\end{eqnarray*}
and
\begin{eqnarray*}
R_1(W_{({2j-2})/{n}})
&=& \frac{1}{2} (W_{({2j-2})/{n}} -
W_{(2j-1)/{n}})^3 \\
&&{}\times\int_0^1
v^2 f^{(3)}\bigl(vW_{({2j-1})/{n}} + (1-v)W_{(2j-2)/{n}}
\bigr) \, dv.
\end{eqnarray*}
Define
\[
G_0(2j) = \frac{1}{2} \int_0^1
v^2 f^{(3)}\bigl(vW_{({2j-1})/{n}} + (1-v)W_{{2j}/{n}}
\bigr) \, dv
\]
and
\[
G_1(2j-2) = \frac{1}{2} \int_0^1
v^2 f^{(3)}\bigl(vW_{({2j-1})/{n}} + (1-v)W_{({2j-2})/{n}}
\bigr) \, dv.
\]
%
We may assume $r=0$. Define $\Delta W_{{\ell}/{n}} = W_{({\ell
+1})/{n}} - W_{{\ell}/{n}}$. We want to show that
%
\begin{eqnarray}
\label{G0G1}
&&
{\mathbb E} \Biggl[ \Biggl(\sum_{j=1}^{ \lfloor{nt/2} \rfloor}
\bigl\{G_0(2j) \Delta W^3_{({2j-1})/{n}} +
G_1(2j-2)\Delta W^3_{({2j-2})/{n}} \bigr\}
\Biggr)^2 \Biggr] \nonumber\\[-8pt]\\[-8pt]
&&\qquad\le C \biggl\lfloor\frac{nt}{2} \biggr\rfloor
n^{-{3/2}}.\nonumber
\end{eqnarray}

This part of the proof was inspired by a computation in~\cite{NoRev};
see Lemma 4.2. Consider the Hermite polynomial identity $x^3 = H_3(x) +
3H_1(x)$. We use the map $\delta^q (h^{\otimes q}) = H_q(W(h))$ [see
(\ref{Hmap}) in Section~\ref{sec2}], for $h \in\hten$ with $\| h \|_\hten=1$.
For each~$j$, let $w_j:= \|\Delta W_{j/n} \|_\hten$, and note
that condition (i) implies $w_j \le Cn^{-{1/4}}$ for all $j$. Then
\begin{eqnarray*}
\frac{\Delta W_{j/n}^3}{w_j^3} &=& H_3 \biggl(\frac
{\Delta W_{j/n}}{w_j}
\biggr)+3H_1 \biggl(\frac{\Delta
W_{{j}/{n}}}{w_j} \biggr) \\
&=&
\delta^3 \biggl(\frac{\partial_{j/n}^{\otimes
3}}{w_j^3} \biggr) +3\delta\biggl(
\frac{\partial_{j/n}}{w_j} \biggr)
\end{eqnarray*}
so that
\[
\Delta W_{j/n}^3 =
\delta^3 \bigl(\partial_{{j}/{n}}^{\otimes3} \bigr) +
3w_j^2 \delta(\partial_{j/n}).
\]
It follows that we can write
\begin{eqnarray*}
&&G_0(2j) \Delta W^3_{({2j-1})/{n}}
-G_1(2j-2) \Delta W_{({2j-2})/{n}}^3
\\
&&\qquad = G_0(2j) \delta^3\bigl(\partial_{({2j-1})/{n}}^{\otimes
3}
\bigr)-G_1(2j-2) \delta^3\bigl(\partial_{({2j-2})/{n}}^{\otimes3}
\bigr)
\\
&&\qquad\quad{} +3w_{2j}^2 G_0(2j)\delta(
\partial_{({2j-1})/{n}}) - 3w_{2j-1}^2 G_1(2j-2)
\delta(\partial_{({2j-2})/{n}}).
\end{eqnarray*}
It is enough to verify the individual inequalities
%
\begin{eqnarray}
\label{triple2j} {\mathbb E} \Biggl[\Biggl\llvert\sum
_{j=1}^{ \lfloor{nt/2} \rfloor} G_0(2j)\delta^3
\bigl(\partial_{({2j-1})/{n}}^{\otimes3}\bigr)\Biggr\rrvert^2
\Biggr] &\le& C \biggl\lfloor\frac{nt}{2} \biggr\rfloor n^{-{3/2}},
\\
\label{triple2j-2} {\mathbb E} \Biggl[\Biggl\llvert\sum
_{j=1}^{ \lfloor{nt/2} \rfloor} G_1(2j-2)
\delta^3\bigl(\partial_{({2j-2})/{n}}^{\otimes3}\bigr)\Biggr
\rrvert^2 \Biggr] &\le& C \biggl\lfloor\frac{nt}{2} \biggr\rfloor
n^{-{3/2}},
\\
\label{single2j} {\mathbb E} \Biggl[\Biggl\llvert\sum
_{j=1}^{ \lfloor{nt/2} \rfloor} w_{2j}^2G_0(2j)
\delta(\partial_{({2j-1})/{n}})\Biggr\rrvert^2 \Biggr] &\le& C \biggl
\lfloor\frac{nt}{2} \biggr\rfloor n^{-{3/2}}
\end{eqnarray}
and
%
\begin{equation}
\label{single2j-2} {\mathbb E} \Biggl[\Biggl\llvert\sum
_{j=1}^{ \lfloor{nt/2} \rfloor} w_{2j-1}^2G_1(2j-2)
\delta(\partial_{({2j-2})/{n}})\Biggr\rrvert^2 \Biggr] \le C \biggl
\lfloor\frac{nt}{2} \biggr\rfloor n^{-{3/2}}.
\end{equation}
We will show (\ref{triple2j}) and (\ref{single2j}), with (\ref{triple2j-2})
and (\ref{single2j-2}) being essentially similar.\vspace*{5pt}

\textit{Proof of} (\ref{triple2j}). Using (\ref{multi}) and the
duality property,
\begin{eqnarray*}
&&{\mathbb E} \Biggl[ \Biggl(\sum_{j=1}^{ \lfloor{nt/2} \rfloor}
G_0(2j)\delta^3\bigl(\partial_{({2j-1})/{n}}^{\otimes3}
\bigr) \Biggr)^2 \Biggr]
\\
&&\qquad= {\mathbb E} \sum_{j,k=1}^{ \lfloor{nt/2} \rfloor} \Biggl[
G_0(2j)G_0(2k) \\
&&\qquad\quad\hspace*{36pt}{}\times\Biggl(\sum_{r=0}^3
\delta^{6-2r}\bigl(\partial_{(2j-1)/{n}}^{\otimes
3-r} \otimes
\partial_{(2k-1)/{n}}^{\otimes3-r}\bigr) \langle\partial_{({2j-1})/{n}},
\partial_{(2k-1)/{n}} \rangle_\hten^r \Biggr) \Biggr]
\\
&&\qquad\le\sum_{j,k=1}^{ \lfloor{nt/2} \rfloor}\sum
_{r=0}^3\bigl\llvert\langle\partial_{(2j-1)/{n}},
\partial_{(2k-1)/{n}} \rangle_\hten^r \bigr\rrvert
\\
&&\qquad\quad\hspace*{37pt}{}\times
{\mathbb E} \bigl[\bigl\llvert\bigl\langle D^{6-2r} \bigl(G_0(2j)G_0(2k)
\bigr), \partial_{({2j-1})/{n}}^{\otimes3-r} \otimes\partial_{(2k-1)/{n}}^{\otimes3-r}
\bigr\rangle_{\hten^{\otimes6-2r}}\bigr\rrvert\bigr].
\end{eqnarray*}
For integers $r \ge0$, we have
%
\begin{eqnarray}
\label{DrG0}
D^rG_0(2j) &=& D^r \int_0^1
\frac{1}{2} v^2 f^{(3)} \bigl(vW_{(2j-1)/{n}}+(1-v)W_{{2j}/{n}}
\bigr) \,dv
\nonumber\\
&=& \frac{1}{2}\int_0^1
v^2 f^{(3+r)} \bigl(vW_{(2j-1)/{n}}+(1-v)W_{{2j}/{n}}
\bigr) \\
&&\hspace*{24pt}{}\times
\bigl(v \varepsilon_{(2j-1)/{n}}^{\otimes r} + (1-v)
\varepsilon_{{2j}/{n}}^{\otimes r} \bigr) \,dv.\nonumber
\end{eqnarray}
By product rule and (\ref{DrG0}) we have
%
\begin{eqnarray}
\label{D6}
&&
{\mathbb E} \bigl[\bigl\llvert\bigl\langle D^{6-2r}
\bigl(G_0(2j)G_0(2k) \bigr), \partial_{({2j-1})/{n}}^{\otimes3-r}
\otimes\partial_{(2k-1)/{n}}^{\otimes3-r} \bigr\rangle_{\hten
^{\otimes
6-2r}}
\bigr\rrvert\bigr]\nonumber\\
&&\qquad\le C \sum_{a+b=6-2r}{\mathbb E} \Bigl[
\sup_{0\le
v,w\le
1}\bigl\llvert f^{(a)}\bigl(vW_{({2j-1})/{n}}+(1-v)W_{(2j-2)/{n}}
\bigr)\nonumber\\
&&\qquad\quad\hspace*{99pt}{}\times f^{(b)}\bigl(wW_{(2k-1)/{n}}+(1-w)W_{(2k-2)/{n}}\bigr
)\bigr
\rrvert\Bigr]
\nonumber\\[-8pt]\\[-8pt]
&&\hspace*{48.6pt}\qquad\quad{} \times\int_0^1 \int_0^1
\bigl\llvert\bigl\langle\bigl(v\varepsilon_{(2j-1)/{n}}^{\otimes a} +
(1-v)\varepsilon_{{2j}/{n}}^{\otimes
a} \bigr) \nonumber\\
&&\qquad\quad\hspace*{99pt}{}\otimes\bigl(w
\varepsilon_{(2k-1)/{n}}^{\otimes b} + (1-w)\varepsilon_{{2k}/{n}}^{\otimes b}
\bigr),\nonumber\\
&&\qquad\quad\hspace*{151.5pt} \partial_{(2j-1)/{n}}^{\otimes3-r} \otimes\partial_{(2k-1)/{n}}^{\otimes
3-r}
\bigr\rangle_{\hten^{\otimes6-2r}}\bigr\rrvert \,dv
\,dw.\nonumber
\end{eqnarray}
Notice that by condition (0), ${\mathbb E} [\sup\llvert f^{(3+r)}(\xi
)\rrvert^p ]< \infty$, where the supremum is taken over the random
variables
$ \{\xi= vW_{s_1} + (1-v)W_{s_2}, 0\le v\le1, 0\le s_1,s_2 \le
T \}$.
From Lemma~\ref{le6.2}, for integers $a$ and $b$ with $a+b = 6-2r$, we have the
following estimate:
%
\begin{eqnarray}
\label{dblint} &&\int_0^1 \int
_0^1 \bigl\llvert\bigl\langle\bigl(v
\varepsilon^{\otimes a}_{({2j-1})/{n}}+(1-v)\varepsilon^{\otimes
a}_{{2j}/{n}}
\bigr)\nonumber\\
&&\qquad\quad\hspace*{1.5pt}{} \otimes\bigl(w\varepsilon^{\otimes b}_{(2k-1)/{n}} + (1-w)
\varepsilon^{\otimes b}_{{2k}/{n}} \bigr),\nonumber\\[-8pt]\\[-8pt]
&&\qquad\quad\hspace*{51.5pt} \partial_{(2j-1)/{n}}^{\otimes3-r}
\otimes\partial_{(2k-1)/{n}}^{\otimes
3-r} \bigr\rangle_{\hten^{\otimes6-2r}}\bigr
\rrvert \,dv \,dw
\nonumber\\
&&\qquad\le Cn^{-(3-r)}.
\nonumber
\end{eqnarray}
It follows that if $r \neq0$, then by Lemma~\ref{le4.1}, equation (\ref{D6}) and
equation (\ref{dblint}),
\begin{eqnarray*}
&&
C \sum_{j,k=1}^{ \lfloor{nt/2} \rfloor} \bigl\llvert\langle
\partial_{(2j-1)/{n}}, \partial_{(2k-1)/{n}} \rangle_\hten^r\bigr\rrvert\\[-2pt]
&&\hspace*{20.2pt}\quad{}\times
{\mathbb E} \bigl[\bigl\llvert\bigl\langle D^{6-r}
\bigl(G_0(2j)G_0(2k) \bigr),\partial_{({2j-1})/{n}}^{\otimes3-r}
\otimes\partial_{(2k-1)/{n}}^{\otimes3-r} \bigr\rangle_{\hten^{\otimes6-2r}}\bigr
\rrvert\bigr]
\\[-2pt]
&&\qquad \le Cn^{r-3} \sum_{j,k=1}^{ \lfloor{nt/2} \rfloor}
\bigl\llvert\langle\partial_{(2j-1)/{n}}, \partial_{(2k-1)/{n}}
\rangle_\hten^r\bigr\rrvert
\\[-2pt]
&&\qquad \le C{ \biggl\lfloor\frac{nt}{2} \biggr\rfloor} n^{{r}/{2} -3},
\end{eqnarray*}
which satisfies (\ref{G0G1}) because $\frac{r}{2} - 3 \le-\frac{3}{2}$.
On the other hand, if $r=0$, then
\[
\sum_{j,k=1}^{ \lfloor{nt/2} \rfloor} Cn^{-3} \le C{
\biggl\lfloor\frac{nt}{2} \biggr\rfloor}n^{-2},
\]
and we are done with (\ref{triple2j}).\vspace*{5pt}

\textit{Proof of} (\ref{single2j}). Proceeding along the same lines
as above,
\begin{eqnarray*}
&&{\mathbb E} \Biggl[ \Biggl(\sum_{j=1}^{ \lfloor{nt/2} \rfloor}
w_{2j}^2G_0(2j)\delta(\partial_{({2j-1})/{n}}
) \Biggr)^2 \Biggr]
\\[-2pt]
&&\qquad = {\mathbb E} \Biggl[ \sum_{j,k=1}^{ \lfloor{nt/2} \rfloor}
w_{2j}^2w_{2k}^2
G_0(2j)G_0(2k) \\[-2pt]
&&\qquad\quad\hspace*{35pt}{}\times\bigl\{ \delta^2 (
\partial_{({2j-1})/{n}} \otimes\partial_{(2k-1)/{n}} ) + \langle
\partial_{(2j-1)/{n}},\partial_{(2k-1)/{n}} \rangle_\hten\bigr\}
\Biggr]
\\[-2pt]
&&\qquad \le Cn^{-1} \sum_{j,k=1}^{ \lfloor{nt/2} \rfloor} {
\mathbb E} \Bigl[ {\mathbb E}\sup_{0\le\ell\le\lfloor{nt}/{2}\rfloor
} \bigl\llvert G_0(
\ell)\bigr\rrvert^2 \bigl\llvert\langle\partial_{({2j-1})/{n}},
\partial_{(2k-1)/{n}} \rangle_\hten\bigr\rrvert\Bigr]
\\[-2pt]
&&\qquad\quad{} + Cn^{-1}\sum_{j,k=1}^{ \lfloor{nt/2} \rfloor} {
\mathbb E} \biggl[\sum_{a+b=2}{\mathbb E}\bigl\llvert
\bigl\langle D^aG_0(2j)D^bG_0(2k),\\[-2pt]
&&\qquad\quad\hspace*{119.6pt}
\delta^2 (\partial_{({2j-1})/{n}} \otimes\partial_{(2k-1)/{n}} )
\bigr\rangle_{\hten^{\otimes2}}\bigr\rrvert\biggr].
\end{eqnarray*}
By Lemma~\ref{le4.1} we have an estimate for the second term,
\[
Cn^{-1} \sum_{j,k=1}^{ \lfloor{nt/2} \rfloor} \bigl
\llvert\langle\partial_{(2j-1)/{n}},\partial_{(2k-1)/{n}}
\rangle_\hten\bigr\rrvert\le C{ \biggl\lfloor\frac{nt}{2} \biggr
\rfloor}n^{-{3/2}}.\vadjust{\goodbreak}
\]
Then the first term has the same estimate as (\ref{D6}) when $r=2$,
which proves (\ref{single2j}) and the lemma.\vspace*{-2pt}

\subsection{\texorpdfstring{Proof of Lemma \protect\ref{le4.5}}{Proof of Lemma 4.5}}\label{sec6.2}
As in Lemma~\ref{le4.4}, we may assume $r=0$. Start with $B_n(t)$. Define
\begin{eqnarray*}
\gamma_n(t):\!&=& \sum_{j=1}^{\lfloor{nt}/{2} \rfloor}
f^{(3)}(W_{({2j-1})/{n}}) \bigl\langle\varepsilon_{({2j-1})/{n}},
\partial_{({2j-1})/{n}}^{\otimes2}-\partial_{(2j-2)/{n}}^{\otimes2}
\bigr\rangle_\hten
\\
&=& \sum_{j=1}^{\lfloor{nt}/{2} \rfloor} f^{(3)}(W_{({2j-1})/{n}})
\langle\varepsilon_{({2j-1})/{n}}, \partial_{({2j-1})/{n}} -
\partial_{({2j-2})/{n}} \rangle_\hten\\
&&\hspace*{21.3pt}{}\times( \partial_{({2j-1})/{n}} -
\partial_{({2j-2})/{n}} ),
\end{eqnarray*}
so that $B_n(t) = 2\delta(\gamma_n(t))$. By Lemma~\ref{le2.1}, we have
\[
\bigl\|\delta\bigl(\gamma_n(t)\bigr)
\bigr\|_{L^2(\Omega)}^2 \le{\mathbb E}\bigl\|\gamma_n(t)\bigr\|
^2_{\hten} + {\mathbb E}\bigl\|D\gamma_n(t)\bigr\|_{\hten^{\otimes2}}^2.
\]
We can write
\begin{eqnarray*}
\bigl\| \gamma_n(t) \bigr\|_\hten^2 &=& \sum
_{j,k=1}^{\lfloor
{nt}/{2} \rfloor} f^{(3)}(W_{({2j-1})/{n}})f^{(3)}(W_{(2k-1)/{n}})\\
&&\hspace*{22pt}{}\times
\langle\varepsilon_{({2j-1})/{n}}, \partial_{(2j-1)/{n}} -
\partial_{({2j-2})/{n}} \rangle_\hten
\\
&&\hspace*{22pt}{}\times\langle\varepsilon_{(2k-1)/{n}}, \partial_{(2k-1)/{n}} -
\partial_{(2k-2)/{n}} \rangle_\hten\\
&&\hspace*{22pt}{}\times\langle\partial_{({2j-1})/{n}}
- \partial_{({2j-2})/{n}}, \partial_{(2k-1)/{n}} - \partial_{(2k-2)/{n}}
\rangle_\hten
\\
&\le&\sup_{0\le s\le t} \bigl\llvert f^{(3)}(W_s)\bigr
\rrvert^2 \sup_{1\le
j\le
\lfloor nt \rfloor} \langle\varepsilon_{({2j-1})/{n}},
\partial_{({2j-1})/{n}} - \partial_{({2j-2})/{n}} \rangle_\hten^2\\
&&{}\times
\sum_{j,k=1}^{ \lfloor{nt/2} \rfloor} \bigl\llvert\langle
\partial_{({2j-1})/{n}} -\partial_{({2j-2})/{n}}, \partial_{(2k-1)/{n}}-
\partial_{(2k-2)/{n}} \rangle_\hten\bigr\rrvert.
\end{eqnarray*}
Note that ${\mathbb E} [ \sup_{0\le s\le t}|f^{(3)}(W_s)|^2 ]
= C$ by condition (0), and by Lemma~\ref{le6.2}, $\llvert\langle
\varepsilon_t,
\partial_{({2j-1})/{n}} - \partial_{({2j-2})/{n}} \rangle_\hten
\rrvert\le C_2 n^{-{1/2}}$ for all $j, t$. By Corollary~\ref{co4.2} we know
\begin{eqnarray*}
&&\sum_{j,k=1}^{ \lfloor{nt/2} \rfloor} \bigl\llvert\langle
\partial_{({2j-1})/{n}} -\partial_{({2j-2})/{n}}, \partial_{(2k-1)/{n}}-
\partial_{(2k-2)/{n}} \rangle_\hten\bigr\rrvert\le C \biggl\lfloor
\frac{nt}{2} \biggr\rfloor n^{-{1/2}}.
\end{eqnarray*}
Hence, it follows that
${\mathbb E}\| \gamma_n(t) \|_\hten^2 \le C \lfloor\frac
{nt}{2} \rfloor n^{-1} n^{-{1/2}}
\le C \lfloor\frac{nt}{2} \rfloor n^{-{3/2}}$. Next,
\begin{eqnarray*}
D\gamma_n(t) &=& \sum_{j=1}^{\lfloor{nt}/{2} \rfloor}
f^{(4)}(W_{({2j-1})/{n}}) \langle\varepsilon_{({2j-1})/{n}},
\partial_{({2j-1})/{n}} - \partial_{({2j-2})/{n}} \rangle_\hten\\
&&\hspace*{21pt}{}\times
\bigl(
\varepsilon_{({2j-1})/{n}} \otimes(\partial_{(2j-1)/{n}} -
\partial_{({2j-2})/{n}} ) \bigr),
\end{eqnarray*}
and this implies
\begin{eqnarray*}
\bigl\| D\gamma_n(t) \bigr\|_{\hten^{\otimes2}}^2
&\le&
\sup_{0\le
s\le t}\bigl\llvert f^{(4)}(W_s)\bigr
\rrvert^2 \\
&&{}\times\Biggl\llvert\sum_{j,k=1}^{ \lfloor{nt/2} \rfloor}
\langle\varepsilon_{({2j-1})/{n}}, \partial_{({2j-1})/{n}} -
\partial_{({2j-2})/{n}} \rangle_\hten\\
&&\hspace*{36pt}{}\times\langle\varepsilon_{(2k-1)/{n}}, \partial_{(2k-1)/{n}} - \partial_{(2k-2)/{n}}
\rangle_\hten\Biggr\rrvert
\\
&&{} \times\bigl\llvert\bigl\langle\varepsilon_{({2j-1})/{n}} \otimes(
\partial_{({2j-1})/{n}} - \partial_{({2j-2})/{n}} ),\\
&&\hspace*{-31.4pt}\hspace*{55.5pt}\varepsilon
_{(2k-1)/{n}}
\otimes(\partial_{(2k-1)/{n}} - \partial_{(2k-2)/{n}} ) \bigr
\rangle_{\hten^{\otimes2}}\bigr\rrvert
\\
&\le&\sup_{0\le s\le t}\bigl\llvert f^{(4)}(W_t)\bigr\rrvert^2
\Bigl(\sup_j \langle\varepsilon_{({2j-1})/{n}},
\partial_{({2j-1})/{n}} - \partial_{({2j-2})/{n}} \rangle_\hten^2
\Bigr)
\\
&&{} \times\sup_{0\le s,r \le t}\bigl\llvert\langle\varepsilon_s,
\varepsilon_r \rangle_\hten\bigr\rrvert\\
&&
\hspace*{0pt}{}\times\sum
_{j,k=1}^{ \lfloor{nt/2} \rfloor} \bigl\llvert\langle
\partial_{({2j-1})/{n}} -\partial_{({2j-2})/{n}}, \partial_{(2k-1)/{n}}-
\partial_{(2k-2)/{n}} \rangle_\hten\bigr\rrvert.
\end{eqnarray*}
By condition (0), ${\mathbb E} [\sup_{0\le s \le t}
|f^{(4)}(W_s)| ]$ is bounded, and $\sup_{0\le s,r\le t}|
\langle
\varepsilon_r, \varepsilon_s \rangle_\hten|$ is bounded. Hence,
it can
be seen that ${\mathbb E}\| D\gamma_n(t) \|_{\hten^{\otimes2}}^2$
gives the same estimate as $\gamma_n(t)$.\

For $C_n(t)$, using condition (0) and the identity $a^2 - b^2 = (a-b)
(a+b)$, we can write
\begin{eqnarray*}
{\mathbb E} \bigl[C_n(t)^2 \bigr] &\le&{\mathbb E} \Bigl[
\sup_{0\le s
\le
t} \bigl|f^{(4)}(W_s)\bigr|^2 \Bigr]\\
&&
\hspace*{0pt}{}\times
\Biggl(\sup_{1\le j\le{nt}/{2}}\bigl\llvert\langle\varepsilon
_{({2j-1})/{n}},
\partial_{({2j-1})/{n}} - \partial_{({2j-2})/{n}} \rangle_\hten\bigr
\rrvert\\
&&
\hspace*{54.6pt}{}\times
\sum_{j=1}^{ \lfloor{nt/2} \rfloor} \bigl\llvert
\langle\varepsilon_{({2j-1})/{n}}, \partial_{({2j-1})/{n}} +
\partial_{({2j-2})/{n}} \rangle_\hten\bigr\rrvert\Biggr)^2.
\end{eqnarray*}
By Lemma~\ref{le6.2}, $\llvert\langle\varepsilon_{({2j-1})/{n}},
\partial_{({2j-1})/{n}} - \partial_{({2j-2})/{n}} \rangle_\hten
\rrvert\le
C_2 n^{-{1/2}}$ and by condition~(iv),
\begin{eqnarray*}
\sum_{j=1}^{ \lfloor{nt/2} \rfloor} \bigl\llvert\langle
\varepsilon_{({2j-1})/{n}}, {\mathbf1}_{[({2j-2})/{n},
{2j}/{n})} \rangle_\hten
\bigr\rrvert&\le& Cn^{-{1/2}}+Cn^{-{1/2}}\sum
_{j=2}^{ \lfloor{nt/2} \rfloor} (2j-2)^{-{1/2}} \\
&\le& C \biggl\lfloor
\frac{nt}{2} \biggr\rfloor^{{1/2}} n^{-{1/2}}.
\end{eqnarray*}
Hence it follows that ${\mathbb E} [C_n(t)^2 ] \le C
\lfloor\frac{nt}{2} \rfloor
n^{-2}$ for some constant C, and the lemma is proved.

\subsection{\texorpdfstring{Proof of Lemma \protect\ref{le4.9}}{Proof of Lemma 4.9}}\label{sec6.3}
For $i = 1,\ldots, d$, set
\[
u_n^i = \sum_{j=\lfloor{nt_{i-1}}/{2}\rfloor+1}^{\lfloor
{nt_i}/{2}\rfloor}
f''(W_{({2j-1})/{n}}) \bigl( \partial_{(2j-1)/{n}}^{\otimes2}
- \partial_{({2j-2})/{n}}^{\otimes2} \bigr)
\]
and recall that $F_n^i = \delta^2(u_n^i)$.
As in Remark~\ref{re3.3}, we want to show:\vspace*{8pt}

Condition (a). For each $1 \le i \le d$, the following
converge to zero in $L^1(\Omega)$:

(a.1) $ \langle u_n^i, h_1 \otimes h_2 \rangle_{\hten
^{\otimes
2}} \mbox{ for all } h_1, h_2 \in\hten$ of the form
$\varepsilon_\tau$ (see Remark~\ref{re3.4}).

(a.2) $ \langle u_n^i, DF_n^j \otimes h \rangle_{\hten
^{\otimes
2}} \mbox{ for each } 1 \le j \le d \mbox{ and } h \in
\hten$.

(a.3) $ \langle u_n^i, DF_n^j \otimes DF_n^k \rangle_{\hten
^{\otimes2}} \mbox{ for each } 1 \le j, k \le d$.\vspace*{8pt}

Condition (b).

(b.1) $ \langle u_n^i, D^2F_n^j \rangle_{\hten
^{\otimes2}}
\longrightarrow0$ in $L^1$ if $i \neq j$.

(b.2) $ \langle u_n^i, D^2F_n^i \rangle_{\hten
^{\otimes2}}$
converges in $L^1$ to a random variable of the form
\[
F_\infty^j = c\int_{t_{i-1}}^{t_i}
f''(W_s)^2 \eta(ds).
\]

The proofs of (a.1) and (a.2) are essentially the same as those given
in~\cite{NoNu} (see Theorem 5.2), but the proof of (a.3) is new.

\textit{Proof of} (a.1).
We may assume $i=1$. Let $h_1 \otimes h_2 = \varepsilon_s \otimes
\varepsilon_\tau\in\hten^{\otimes2}$ for some values $s, \tau\in
[0,t]$. Then
\begin{eqnarray*}
\bigl\langle u_n^1, h_1 \otimes
h_2 \bigr\rangle_{\hten^{\otimes2}} &=& \sum_{j=1}^{\lfloor{nt_1}/{2}
\rfloor}
f''(W_{({2j-1})/{n}}) \langle\partial_{({2j-1})/{n}}
- \partial_{({2j-2})/{n}}, \varepsilon_s \rangle_\hten\\
&&\hspace*{24.5pt}{}\times
\langle\partial_{({2j-1})/{n}} - \partial_{(2j-2)/{n}}, \varepsilon_\tau
\rangle_\hten;
\end{eqnarray*}
so that
\begin{eqnarray*}
&&
\bigl\llvert\bigl\langle u_n^1, h_1
\otimes h_2 \bigr\rangle_{\hten^{\otimes
2}}\bigr\rrvert\\
&&\qquad\le
\sup_{0\le s \le t}\bigl|f''(W_s)\bigr|
\sup_{1\le j\le\lfloor{nt_1}/{2}
\rfloor} \sup_{0 \le s \le t_1} \bigl\llvert\langle
\partial_{({2j-1})/{n}} - \partial_{({2j-2})/{n}}, \varepsilon_s
\rangle_\hten\bigr\rrvert\\
&&\qquad\quad\hspace*{113pt}{}\times\sum_{j=1}^{\lfloor{nt_1}/{2} \rfloor}
\bigl\llvert\langle\partial_{(2j-1)/{n}} - \partial_{({2j-2})/{n}},
\varepsilon_\tau\rangle_\hten\bigr\rrvert.
\end{eqnarray*}
It follows from condition (iii) that for fixed $\tau\ge0$,
%
\begin{eqnarray}\label{deltaeps}
&&
\sum_{j=1}^{\lfloor{nt_1}/{2} \rfloor} \bigl\llvert\langle
\partial_{({2j-1})/{n}} - \partial_{({2j-2})/{n}}, \varepsilon_\tau
\rangle_\hten\bigr\rrvert\nonumber\\
&&\qquad = \sum_{j=1}^{\lfloor{nt_1}/{2} \rfloor}
\bigl\llvert{\mathbb E} \bigl[ W_\tau(W_{{2j}/{n}} -
2W_{({2j-1})/{n}}+W_{(2j-2)/{n}}) \bigr]\bigr\rrvert
\nonumber\\[-8pt]\\[-8pt]
&&\qquad \le Cn^{-{1/2}} + Cn^{-{1/2}}\sum_{j=2}^{\lfloor
{nt_1}/{2}\rfloor}
\bigl((2j-2)^{-{3/2}} + |\tau-2j|^{-{3/2}}\wedge1 \bigr)
\nonumber\\
&&\qquad\le Cn^{-{1/2}},\nonumber
\end{eqnarray}
and Lemma~\ref{le6.2} implies
\[
\sup_{1\le j\le\lfloor{nt_1}/{2} \rfloor} \sup_{0 \le s \le t} \bigl
\llvert\langle
\partial_{({2j-1})/{n}} - \partial_{(2j-2)/{n}}, \varepsilon_s
\rangle_\hten\bigr\rrvert\le Cn^{-{1/2}}
\]
so that
\[
{\mathbb E} \bigl( \bigl\llvert\bigl\langle u_n^1,
h_1 \otimes h_2 \bigr\rangle_{\hten
^{\otimes2}}\bigr\rrvert
\bigr) \le Ct_1n^{-1} \longrightarrow0.
\]

\textit{Proof of} (a.2). 
As in (a.1), assume $i=1$. Using the same technique as in (a.1), we can
write $DF_n^j \otimes h$ as $DF_n^j \otimes\varepsilon_\tau$ for some
$\tau\in[0,T]$. By Lemma~\ref{le2.1}, $DF^j_n = D\delta^2(u_n^j) = \delta
^2(Du_n^j)+\delta(u_n^j)$, which gives
\[
\bigl\langle u_n^1, DF_n^j
\otimes\varepsilon_\tau\bigr\rangle_{\hten
^{\otimes
2}} = \bigl\langle
u_n^1,\delta^2\bigl(Du_n^j
\bigr) \otimes\varepsilon_\tau\bigr\rangle_{\hten^{\otimes2}} + \bigl
\langle u_n^1, \delta\bigl(u_n^j
\bigr) \otimes\varepsilon_\tau\bigr\rangle_{\hten^{\otimes2}}.
\]
For the first term,
\begin{eqnarray*}
&&
{\mathbb E} \bigl\llvert\bigl\langle u_n^1,
\delta^2\bigl(Du_n^j\bigr) \otimes
\varepsilon_\tau\bigr\rangle_{\hten^{\otimes2}} \bigr\rrvert
\\
&&\qquad = \sum_{\ell=1}^{\lfloor{nt_1}/{2}\rfloor} {\mathbb E}\bigl
\llvert f''(W_{(2\ell-1)/{n}}) \bigl\langle
\partial_{(2\ell-1)/{n}} - \partial_{(2 \ell-2)/{n}}, \delta^2
\bigl(Du_n^j\bigr) \bigr\rangle_\hten\\
&&\qquad\quad\hspace*{114.5pt}{}\times\langle
\partial_{(2\ell-1)/{n}} - \partial_{(2 \ell-2)/{n}},
\varepsilon_\tau
\rangle_\hten\bigr\rrvert
\\
&&\qquad \le2{\mathbb E} \Bigl[ \sup_{0\le s \le t_1}\bigl|f''(W_s)\bigr|
\Bigr] {\mathbb E} \Bigl[ \sup_{1\le\ell\le\lfloor{
{nt_1}/{2}}\rfloor} \bigl\llvert\bigl\langle
\partial_{{\ell}/{n}}, \delta^2 \bigl(Du_n^j
\bigr) \bigr\rangle_\hten\bigr\rrvert\Bigr] \\
&&\qquad\quad\hspace*{0pt}{}\times\sum
_{\ell= 1}^{\lfloor{nt_1}/{2}\rfloor} \bigl\llvert\langle
\partial_{(2\ell-1)/{n}} -\partial_{(2\ell-2)/{n}}, \varepsilon
_\tau
\rangle_\hten\bigr\rrvert.
\end{eqnarray*}
By (\ref{deltaeps}), the sum has estimate $Cn^{-{1/2}}$, and for
the second term we can take
\[
\bigl\llvert\bigl\langle\partial_{{\ell}/{n}}, \delta^2
\bigl(Du_n^j\bigr) \bigr\rangle_\hten\bigr
\rrvert\le\sup_\ell\|\partial_{{\ell}/{n}} \|_\hten\bigl\|
\delta^2 \bigl(Du_n^j\bigr)
\bigr\|_\hten.
\]
It follows from condition (i) that $\|\partial_{{\ell}/{n}} \|_\hten
\le Cn^{-{1/4}}$. This leaves the $\| \delta^2 (Du_n^j)
\|_\hten$ term. By the Meyer inequality for a process taking values in
$\hten$,
%
\begin{equation}
\label{Meyer49} \hspace*{15pt}{\mathbb E} \bigl[\bigl\| \delta^2
\bigl(Du_n^j\bigr) \bigr\|_\hten^2
\bigr] \le c_1{\mathbb E}\bigl\|Du_n^j
\bigr\|_{\hten^{\otimes3}}^2 + c_2{\mathbb E} \bigl\|
D^2u_n^j\bigr\|_{\hten^{\otimes4}}^2 +
c_3{\mathbb E}\bigl\|D^3 u_n^j
\bigr\|_{\hten
^{\otimes5}}^2,
\end{equation}
so that by Lemma~\ref{le4.7}, ${\mathbb E} [\|\delta^2 (Du) \|_\hten^2
] \le C$, and we have
\[
{\mathbb E} \bigl\llvert\bigl\langle u_n^1,
\delta^2\bigl(Du_n^j\bigr) \otimes
\varepsilon_\tau\bigr\rangle_{\hten^{\otimes2}}\bigr\rrvert\le
Cn^{-{3/4}}.
\]
Then similarly,
\begin{eqnarray*}
&&
\bigl\llvert\bigl\langle u_n^1,\delta
\bigl(u_n^j\bigr) \otimes\varepsilon_t \bigr
\rangle_{\hten
^{\otimes2}}\bigr\rrvert\\
&&\qquad\le2 \biggl[\sup_{0\le s \le t_1}\bigl|f''(W_s)\bigr|
\sup_\ell\bigl\llvert\bigl\langle\partial_{{\ell}/{n}}, \delta
\bigl(u_n^j\bigr) \bigr\rangle_\hten\bigr
\rrvert\sum_{\ell}\bigl\llvert\langle
\partial_{(2\ell
-1)/{n}} -\partial_{(2\ell-2)/{n}}, \varepsilon_t
\rangle_\hten\bigr\rrvert\biggr].
\end{eqnarray*}
Similar to the above case, for each $1 \le\ell\le\lfloor
\frac
{nt_1}{2} \rfloor$,
\begin{eqnarray*}
{\mathbb E} \bigl[ \bigl\llvert\bigl\langle\partial_{
{\ell}/{n}}, \delta
\bigl(u_n^j\bigr) \bigr\rangle_\hten\bigr
\rrvert\bigr] &\le &{\mathbb E} \bigl[\| \partial_{{\ell}/{n}}\|
_\hten
\bigl\| \delta\bigl(u_n^j\bigr) \bigr\|_\hten\bigr]
\\
&\le& Cn^{-{1/4}} \bigl( {\mathbb E}\bigl\| u_n^j
\bigr\|_{\hten^{\otimes2}} + {\mathbb E} \bigl\| Du_n^j
\bigr\|_{\hten^{\otimes3}} \bigr) \\
&\le& Cn^{-{1/4}},
\end{eqnarray*}
and hence with (\ref{deltaeps}) we have
\[
{\mathbb E} \bigl[ \bigl\llvert\bigl\langle u_n^1,
\delta\bigl(u_n^j\bigr) \otimes\varepsilon_\tau
\bigr\rangle_{\hten^{\otimes2}} \bigr\rrvert\bigr] \le Cn^{-{3/4}}.
\]

\textit{Proof of} (a.3). 
For this term we consider the product $ \langle u_n^i, DF_n^j
\otimes
DF_n^k \rangle_{\hten^{\otimes2}}$. Lemma~\ref{le6.1} shows that scalar
products of this kind are small in absolute value when the time
intervals are disjoint, and therefore it is enough to consider the
worst case $ \langle u_n^1, DF_n^1 \otimes DF_n^1 \rangle_{\hten
^{\otimes
2}}$, and assume $t_1 = t$.
We have
\begin{eqnarray*}
&&
{\mathbb E} \bigl[\bigl\llvert\bigl\langle u_n^1,
DF_n^1 \otimes DF_n^1 \bigr
\rangle_{\hten
^{\otimes2}}\bigr\rrvert\bigr] \\
&&\qquad\le \sum_{\ell=1}^{ \lfloor{nt/2} \rfloor}
\bigl\llvert{\mathbb E} \bigl[ \bigl\langle f''(W_{(2\ell-1)/{n}})
\bigl(\partial_{(2\ell-1)/{n}}^{\otimes
2}-\partial_{(2\ell-2)/{n}}^{\otimes2}
\bigr),DF_n^1 \otimes DF_n^1
\bigr\rangle_{\hten^{\otimes2}} \bigr]\bigr\rrvert
\\
&&\qquad\le C\sum_{\ell=1}^{ \lfloor{nt/2} \rfloor} {\mathbb E} \bigl[
\bigl\llvert\bigl\langle\partial_{(2\ell-1)/{n}}, DF_n^1
\bigr\rangle_\hten^2 - \bigl\langle\partial_{(2\ell-2)/{n}},
DF_n^1 \bigr\rangle_\hten^2\bigr
\rrvert\bigr]
\\
&&\qquad\le C\sum_{\ell=1}^{ \lfloor{nt/2} \rfloor} {\mathbb E} \bigl[
\bigl\llvert\bigl\langle\partial_{(2\ell-1)/{n}}-\partial_{(2\ell-2)/{n}},
DF_n^1 \bigr\rangle_\hten\bigr\rrvert\bigl
\llvert\bigl\langle{\mathbf1}_{[({2\ell-2})/{n}, {2\ell
}/{n}]},DF_n^1
\bigr\rangle_\hten\bigr\rrvert\bigr].
\end{eqnarray*}
Using the decomposition $DF_n^1 = \delta^2(Du_n^1) + \delta(u_n^1)$,
the above summand expands into four terms:
\begin{eqnarray*}
&&\mbox{(1)\quad} \bigl\llvert\bigl\langle\partial_{(2\ell-1)/{n}} -
\partial_{(2\ell-2)/{n}}, \delta^2\bigl(Du_n^1
\bigr) \bigr\rangle_\hten\bigr\rrvert\bigl\llvert\bigl\langle{
\mathbf1}_{[({2\ell-2})/{n},{2\ell}/{n}]}, \delta^2\bigl(Du_n^1
\bigr) \bigr\rangle_\hten\bigr\rrvert;
\\
&&\mbox{(2)\quad}\bigl\llvert\bigl\langle\partial_{(2\ell-1)/{n}} -
\partial_{(2\ell-2)/{n}}, \delta^2\bigl(Du_n^1\bigr) \bigr
\rangle_\hten\bigr\rrvert\bigl\llvert\bigl\langle{
\mathbf1}_{[({2\ell-2})/{n},{2\ell}/{n}]}, \delta\bigl(u_n^1\bigr
) \bigr\rangle_\hten\bigr\rrvert;
\\
&&\mbox{(3)\quad}\bigl\llvert\bigl\langle\partial_{(2\ell-1)/{n}} -
\partial_{(2\ell-2)/{n}}, \delta\bigl(u_n^1\bigr)
\bigr\rangle_\hten\bigr\rrvert
\bigl\llvert\bigl\langle{\mathbf1}_{[({2\ell
-2})/{n},{2\ell}/{n}]}, \delta^2\bigl(Du_n^1\bigr) \bigr
\rangle_\hten\bigr\rrvert;
\\
&&\mbox{(4)\quad}\bigl\llvert\bigl\langle\partial_{(2\ell-1)/{n}} -
\partial_{(2\ell-2)/{n}}, \delta\bigl(u_n^1\bigr)
\bigr\rangle_\hten\bigr\rrvert
\bigl\llvert\bigl\langle{\mathbf1}_{[({2\ell
-2})/{n},{2\ell}/{n}]}, \delta\bigl(u_n^1\bigr)
\bigr\rangle_\hten\bigr\rrvert.
\end{eqnarray*}
We will show computations for the terms (1) and (4) only, with the
others being similar. For (1) we have
\begin{eqnarray*}
&&
C\sum_{\ell=1}^{ \lfloor{nt/2} \rfloor} {\mathbb E} \bigl[\bigl
\llvert\bigl\langle\partial_{(2\ell-1)/{n}} - \partial_{(2\ell-2)/{n}},
\delta^2\bigl(Du_n^1\bigr) \bigr
\rangle_\hten\bigr\rrvert\bigl\llvert\bigl\langle{
\mathbf1}_{[({2\ell-2})/{n},{2\ell}/{n}]}, \delta^2\bigl(Du_n^1
\bigr) \bigr\rangle_\hten\bigr\rrvert\bigr]
\\
&&\qquad = C\sum_{\ell,m,m' =1}^{ \lfloor{nt/2} \rfloor} {\mathbb E}\bigl
\llvert\bigl\langle\partial_{(2\ell-1)/{n}} -
\partial_{(2\ell-2)/{n}},\\
&&\qquad\hspace*{72pt}
\delta^2 \bigl(f^{(3)}(W_{(2m-1)/{n}})
\varepsilon_{(2m-1)/{n}} \bigl(\partial_{(2m-1)/{n}}^{\otimes2} -
\partial_{(2m -2)/{n}}^{\otimes2} \bigr) \bigr) \bigr\rangle_\hten
\bigr\rrvert
\\
&&\hspace*{43.4pt}\qquad\quad{} \times\bigl\llvert\bigl\langle{\mathbf1}_{[({2\ell
-2})/{n},
{2\ell}/{n}]},\\
&&\qquad\quad\hspace*{63.2pt}
\delta^2 \bigl(f^{(3)}(W_{({2m'-1})/{n}})
\varepsilon_{({2m'-1})/{n}} \bigl(\partial_{({2m'-1})/{n}}^{\otimes
2} -
\partial_{({2m' -2})/{n}}^{\otimes2} \bigr) \bigr) \bigr\rangle_\hten
\bigr\rrvert
\\
&&\qquad \le C\sup_{1\le k\le\lfloor{nt}/{2} \rfloor
} \bigl( {\mathbb E} \bigl[ \bigl\llVert
\delta^2 \bigl( f^{(3)}(W_{(2k-1)/{n}}) \bigl(
\partial_{(2k-1)/{n}}^{\otimes2} - \partial_{(2k -2)/{n}}^{\otimes2}
\bigr) \bigr)\bigr\rrVert_{\hten^{\otimes2}} \bigr] \bigr)^2
\\
&&\qquad\quad{} \times\sum_{\ell,m,m' =1}^{ \lfloor{nt/2} \rfloor}
\bigl\llvert\langle\partial_{(2\ell-1)/{n}} -
\partial_{(2\ell-2)/{n}},\varepsilon_{(2m-1)/{n}}
\rangle_\hten\bigr\rrvert\\
&&\qquad\quad\hspace*{46.4pt}{}\times\bigl\llvert\langle{\mathbf1}_{[({2\ell
-2})/{n},{2\ell}/{n}]},\varepsilon_{({2m'-1})/{n}} \rangle_\hten
\bigr\rrvert.
\end{eqnarray*}
By Lemmas~\ref{le2.1} and~\ref{le4.7}, the Skorohod integral term is bounded by
$Cn^{-{1/2}}$, and we use conditions (iii) and (iv) for the
scalar products to obtain an estimate of the form
\begin{eqnarray*}
&&
Cn^{-2}\sum_{\ell, m, m'=1}^{ \lfloor{nt/2} \rfloor} \bigl(
(2m-1)^{-{3/2}} + |2\ell-2m|^{-{3/2}} \bigr)\\
&&\hspace*{60pt}{}\times \bigl((2
\ell-2)^{-{1/2}} + \bigl|2\ell- 2m'\bigr|^{-{1/2}} \bigr)\\
&&\qquad\le
Cn^{-{1/2}}.
\end{eqnarray*}
For term (4), we have by a computation similar to the proof of Lemma~\ref{le4.7},
\[
{\mathbb E} \bigl[ \bigl\llVert\delta\bigl( f^{(3)}(W_{(2k-1)/{n}})
(\partial_{(2k-1)/{n}} - \partial_{(2k -2)/{n}} ) \bigr)\bigr
\rrVert_{\hten} \bigr] \le Cn^{-{1/4}},
\]
and by conditions (i) and (ii) we have
\begin{eqnarray*}
&&
Cn^{-{3/2}}\sum_{\ell, m, m'=1}^{ \lfloor{nt/2} \rfloor} \bigl
\llvert\langle\partial_{(2\ell-1)/{n}} - \partial_{(2\ell-2)/{n}},
\partial_{(2m-1)/{n}} - \partial_{(2m -2)/{n}} \rangle_\hten\bigr
\rrvert\\
&&\qquad\quad\hspace*{36pt}{}\times\bigl\llvert\langle{\mathbf1}_{[({2\ell-2})/{n},{2\ell
}/{n}]},\partial_{({2m'-1})/{n}} - \partial_{({2m'-2})/{n}}
\rangle_\hten\bigr\rrvert
\\
&&\qquad \le Cn^{-{3/2}}\sum_{\ell, m, m'=1}^{ \lfloor{nt/2} \rfloor}
\bigl(|2\ell- 2m|^{-\alpha} \bigr) \bigl(\bigl|2\ell- 2m'\bigr|^{-\alpha}
\bigr)
\\
&&\qquad \le Cn^{-{1/2}}.
\end{eqnarray*}

\textit{Proof of} (b.1). By Lemma~\ref{le2.1}, we can expand $D^2F_n$ as follows:
%
\begin{eqnarray}
\label{D2Fi} \bigl\langle u_n^i, D^2
F_n^j \bigr\rangle_{\hten
^{\otimes2}} &=& \bigl\langle
u_n^i, \delta^2 \bigl(D^2
u_n^j\bigr) \bigr\rangle_{\hten^{\otimes2}} + 4 \bigl
\langle u_n^i, \delta\bigl(Du_n^j
\bigr) \bigr\rangle_{\hten^{\otimes2}}\nonumber\\[-8pt]\\[-8pt]
&&{} + 2 \bigl\langle u_n^i,u_n^j
\bigr\rangle_{\hten^{\otimes2}}.\nonumber
\end{eqnarray}

Without loss of generality, we may assume that $u_n^i$ is defined on
$[0, t_1]$, and $F_n^j$ is defined on $[t_1, t_2]$ for $t_1 < t_2$, so
that the sums are over
\[
u_n^i = \sum_{\ell= 1}^{\lfloor{nt_1}/{2} \rfloor}
f''(W_{(2\ell-1)/{n}}) \bigl(\partial_{(2\ell-1)/{n}}^{\otimes2}
-\partial_{(2\ell-2)/{n}}^{\otimes2} \bigr)
\]
and
\[
u_n^j = \sum_{m = \lfloor{nt_1}/{2} \rfloor+1}^{\lfloor
{nt_2}/{2} \rfloor}
f''(W_{(2m-1)/{n}}) \bigl(\partial_{(2m-1)/{n}}^{\otimes2}
-\partial_{(2m-2)/{n}}^{\otimes2} \bigr).
\]

\textit{First term.}
\begin{eqnarray*}
&&
{\mathbb E}\bigl\llvert\bigl\langle u_n^i,
\delta^2 \bigl(D^2 u_n^j\bigr)
\bigr\rangle_{\hten^{\otimes2}}\bigr\rrvert
\\
&&\qquad = {\mathbb E} \Biggl\llvert\Biggl\langle\sum_{\ell=1}^{\lfloor{nt_1}/{2}
\rfloor}
f'' (W_{(2\ell-1)/{n}}) \bigl(\partial_{(2\ell
-1)/{n}}^{\otimes2}
- \partial_{(2\ell-2)/{n}}^{\otimes2} \bigr),\\
&&\qquad\hspace*{29pt} \delta^2 \Biggl(
\sum_{m=\lfloor{nt_1}/{2} \rfloor+1}^{\lfloor
{nt_2}/{2} \rfloor} f^{(4)}
(W_{(2m-1)/{n}}) \varepsilon_{(2m-1)/{n}}^{\otimes2} \\
&&\qquad\quad\hspace*{85pt}{}\otimes\bigl(
\partial_{(2m-1)/{n}}^{\otimes
2} - \partial_{(2m-2)/{n}}^{\otimes2}
\bigr) \Biggr) \Biggr\rangle_{\hten
^{\otimes2}}\Biggr\rrvert
\\
&&\qquad \le{\mathbb E} \Bigl[\sup_{0\le s \le t}\bigl|f''(W_s)\bigr| \Bigr] \\
&&\qquad\quad{}\times{\mathbb E}
\biggl[\sum_\ell\sum_m \bigl\llvert\bigl\langle\varepsilon_{(2m-1)/{n}}^{\otimes2}, \partial_{(2\ell-1)/{n}}^{\otimes2} -
\partial_{(2\ell-2)/{n}}^{\otimes2}
\bigr\rangle_{\hten^{\otimes2}} \bigr\rrvert\\
&&\qquad\quad\hspace*{54pt}{}\times\bigl\llvert\delta^2 \bigl(
f^{(4)}(W_{(2m-1)/{n}}) \bigl(\partial_{(2m-1)/{n}}^{\otimes2}
- \partial_{(2m-2)/{n}}^{\otimes2} \bigr) \bigr)\bigr\rrvert\biggr]
\\
&&\qquad \le{\mathbb E} \Bigl[\sup_{0\le s \le t}\bigl|f''(W_s)\bigr|
\Bigr] \\
&&\qquad\quad\hspace*{0pt}{}\times\sup_m \bigl\|\delta^2 (g_4)
\bigr\|_{L^2(\Omega)} \sum_{\ell=1}^{\lfloor{nt_2}/{2}
\rfloor} \sum
_{m=1}^{\lfloor{nt_2}/{2} \rfloor} \bigl[ \langle
\varepsilon_{(2m-1)/{n}}, \partial_{(2\ell-1)/{n}} \rangle
_\hten^2\\[-2pt]
&&\qquad\hspace*{160pt}{}
- \langle\varepsilon_{(2m-1)/{n}}, \partial_{(2\ell
-2)/{n}}
\rangle_\hten^2 \bigr].
\end{eqnarray*}

First we need an estimate for the $\delta^2 (g_4)$ term, where in the
notation of Lem\-ma~\ref{le4.7},
\[
g_4:= f^{(4)}(W_{(2m-1)/{n}}) \bigl(
\partial_{(2m-1)/{n}}^{\otimes2} - \partial_{(2m-2)/{n}}^{\otimes2}
\bigr).
\]
By Lemma~\ref{le2.1},
$\| \delta^2 (g_4) \|_{L^2(\Omega)} \le c_1{\mathbb E}\| g_4 \|_{\hten
^{\otimes2}} + c_2{\mathbb E} \| Dg_4 \|_{\hten^{\otimes3}} +
c_3{\mathbb E}\| D^2 g_4 \|_{\hten^{\otimes4}}$,
and so $\| \delta^2 (g_4) \|_{L^2(\Omega)} \le Cn^{-{1/2}}$ for
each $\lfloor\frac{nt_1}{2} \rfloor< m \le\lfloor\frac{nt_2}{2}
\rfloor$. We can write
\begin{eqnarray*}
&&
{\mathbb E}\bigl\llvert\bigl\langle u_n^i,
\delta^2 \bigl(D^2 u_n^i\bigr)
\bigr\rangle_{\hten
^{\otimes2}}\bigr\rrvert\\[-2pt]
&&\qquad\le Cn^{-{1/2}} \sum
_{\ell, m
=1}^{\lfloor
{nt_2}/{2} \rfloor} \bigl\llvert\langle
\varepsilon_{(2m-1)/{n}}, \partial_{(2\ell-1)/{n}} \rangle_\hten^2
 - \langle\varepsilon_{(2m-1)/{n}}, \partial_{(2\ell-2)/{n}}
\rangle_\hten^2 \bigr\rrvert.
\end{eqnarray*}

We need an estimate for the double sum. We have by condition (iii),
\begin{eqnarray*}
&&
\sum_{\ell, m=1}^{\lfloor{nt_2}/{2} \rfloor}  \bigl[ \bigl\llvert
\langle\varepsilon_{(2m-1)/{n}}, \partial_{(2\ell
-1)/{n}}
\rangle_\hten^2 - \langle\varepsilon_{(2m-1)/{n}},
\partial_{(2\ell-2)/{n}} \rangle_\hten^2\bigr\rrvert\bigr]
\\[-2pt]
&&\qquad \le\sup_{\ell,m} \bigl\llvert\langle\varepsilon_{(2m-1)/{n}}, {
\mathbf1}_{[({2\ell-2})/{n},{{2\ell}/{n}}]} \rangle_\hten\bigr
\rrvert\\[-2pt]
&&\quad\qquad\hspace*{14pt}{}\times\sum
_{\ell,m=1}^{\lfloor{nt_2}/{2} \rfloor} \bigl\llvert\langle
\varepsilon_{(2m-1)/{n}}, \partial_{(2\ell-1)/{n}} - \partial
_{(2\ell-2)/{n}}
\rangle_\hten\bigr\rrvert
\\[-2pt]
&&\qquad \le Cn^{-{1/2}} \sum_{\ell, m=1}^{\lfloor{nt_2}/{2}
\rfloor}
C_2 n^{-{1/2}} \bigl[ \bigl( |\ell-m|^{-{3/2}} + (
\ell-1)^{-{3/2}} \bigr)\wedge1 \bigr]
\\[-2pt]
&&\qquad \le Cn^{-1} \sum_{\ell=1}^{\lfloor{nt_2}/{2} \rfloor}
\sum_{p=1}^\infty p^{-{3/2}} \le C.
\end{eqnarray*}

This provides an upper bound for the double sum, and hence the first
term of (\ref{D2Fi}) is $O(n^{-{1/2}})$. Note that in the above
estimate the double sum is taken over $1 \le\ell, m\le\lfloor
\frac{nt_2}{2} \rfloor$. It follows that this estimate also
holds for the case $i=j$, that is, ${\mathbb E}\llvert\langle u_n^i,
\delta^2(D^2 u_n^i) \rangle_{\hten^{\otimes2}}\rrvert\le
Cn^{-{1/2}}$.\vspace*{8pt}

\textit{Second term.}
Using $t_1 < t_2$ as above,
\begin{eqnarray*}
\hspace*{-3pt}&&
{\mathbb E}\bigl\llvert\bigl\langle u_n^i, \delta
\bigl(D u_n^j\bigr) \bigr\rangle_{\hten^{\otimes2}}\bigr
\rrvert
\\[-2pt]
\hspace*{-3pt}&&\quad = {\mathbb E}\Biggl\llvert\Biggl\langle\sum_{j=1}^{\lfloor
{nt_1}/{2} \rfloor}
f'' (W_{({2j-1})/{n}}) \bigl(\partial_{({2j-1})/{n}}^{\otimes2}
- \partial_{(2j-2)/{n}}^{\otimes2} \bigr),\\
\hspace*{-3pt}&&\quad\hspace*{29pt} \delta\Biggl( \sum
_{k=\lfloor
{nt_1}/{2} \rfloor}^{\lfloor{nt_2}/{2}\rfloor} f^{(3)} (W_{(2k-1)/{n}})
\varepsilon_{(2k-1)/{n}} \otimes\bigl(\partial_{(2k-1)/{n}}^{\otimes2}
- \partial_{(2k-2)/{n}}^{\otimes2} \bigr) \Biggr) \Biggr
\rangle_{\hten^{\otimes2}}\Biggr\rrvert
\\
\hspace*{-3pt}&&\quad ={\mathbb E}\biggl\llvert\sum_j
\sum_k f''(W_{(2j-1)/{n}})
\bigl\langle\varepsilon_{(2k-1)/{n}}\otimes(\partial_{(2k-1)/{n}} -
\partial_{(2k-2)/{n}} ),\\
\hspace*{-3pt}&&\quad\hspace*{166pt} ( \partial_{(2j-1)/{n}} - \partial_{({2j-2})/{n}}
)^{\otimes2} \bigr\rangle_{\hten
^{\otimes
2}}\biggr\rrvert
\\
\hspace*{-3pt}&&\qquad{} \times\bigl\llvert\delta\bigl( f^{(3)}(W_{(2k-1)/{n}}) (
\partial_{(2k-1)/{n}} - \partial_{(2k-2)/{n}} ) \bigr)\bigr
\rrvert
\\
\hspace*{-3pt}&&\quad \le C {\mathbb E} \Bigl[\sup_{0\le s \le t}\bigl|f''(W_s)\bigr|
\Bigr] \Bigl(\sup_{s,j}\bigl\llvert\langle\varepsilon_s,
\partial_{{j}/{n}} \rangle_\hten\bigr\rrvert\Bigr) \Bigl(
\sup_k \bigl\|\delta(g_3)\bigr\|_{L^2(\Omega
)} \Bigr) \\
\hspace*{-3pt}&&\qquad{}\times\sum
_{j=0}^{\lfloor nt_2 \rfloor} \sum
_{k=0}^{\lfloor nt_2
\rfloor} \bigl\llvert\langle
\partial_{j/n},\partial_{{k/n}} \rangle_\hten\bigr
\rrvert,
\end{eqnarray*}
where in this case, $g_3$ corresponds to the term including
$f^{(3)}(W_t)$. It follows from Lemma~\ref{le6.2} that $\sup| \langle
\varepsilon_s, \partial_{k/n} \rangle_\hten| \le Cn^{-{1/2}};$
and the double
sum is bounded by $Cn^{{1/2}}$ by Corollary~\ref{co4.2}. This leaves an
estimate for $\| \delta(g_3)\|_{L^2(\Omega)}$. By Lemma~\ref{le2.1},
$\llVert\delta(g_3) \rrVert_{L^2(\Omega)} \le c_1 \| g_3 \|_\hten+
c_2 \| Dg_3 \|_{\hten^{\otimes2}}$. For this case,
\[
\| g_3 \|_{\hten}^2 \le{\mathbb E} \Bigl[
\sup_{0\le s \le
t}\bigl|f^{(3)}(W_s)\bigr|^2 \Bigr]
\llVert\partial_{(2k-1)/{n}} -\partial_{(2k-2)/{n}} \rrVert
_\hten^2
\le Cn^{-{1/2}},
\]
hence $\| g_3 \|_{\hten} \le C n^{-{1/4}}$. Similarly,
\begin{eqnarray*}
\llVert Dg_3 \rrVert_{\hten^{\otimes2}} &\le&{\mathbb E} \Bigl[
\sup_{0\le s \le t}\bigl|f^{(4)}(W_s)\bigr| \Bigr]
\sup_{0\le s \le t} \| \varepsilon_s \|_\hten\llVert
\partial_{(2k-1)/{n}} -\partial_{(2k-2)/{n}} \rrVert_\hten\\
&\le&
Cn^{-{1/4}},
\end{eqnarray*}
hence the second term is $O(n^{-{1/4}})$. As in the first term,
the double sum estimate shows that this result also holds for $
\langle
u_n^i, \delta( DF_n^i) \rangle_{\hten^{\otimes2}}$.\vspace*{8pt}

\textit{Third term.}
We can write
\begin{eqnarray*}
&&
\bigl\llvert\bigl\langle u_n^i, u_n^j
\bigr\rangle_{\hten^{\otimes2}}\bigr\rrvert\le\sup_{0\le s \le t}\bigl|f''(W_s)\bigr|^2
\sum_{\ell=1}^{\lfloor
{nt_1}/{2}\rfloor
} \sum
_{m=\lfloor{nt_1}/{2}\rfloor+1}^{\lfloor{nt_2}/{2}
\rfloor} \bigl\llvert\bigl\langle
\partial_{(2\ell-1)/{n}}^{\otimes2} - \partial_{(2\ell
-2)/{n}}^{\otimes2},\\
&&\qquad\hspace*{201pt}\partial_{(2m-1)/{n}}^{\otimes2} - \partial_{(2m-2)/{n}}^{\otimes2}
\bigr\rangle_{\hten^{\otimes2}}\bigr\rrvert,
\end{eqnarray*}
and it follows from Lemma~\ref{le6.1} that ${\mathbb E}\llvert\langle u_n^i,
u_n^j \rangle_{\hten^{\otimes2}}\rrvert\le Cn^{-\varepsilon}$, for
some $\varepsilon> 0$.

\textit{Proof of} (b.2).
As in case (b.1), this has the expansion (\ref{D2Fi}). From remarks in
the proof of (b.1), the first two terms have the same estimate as the
$i \neq j$ case, hence only the term $ \langle u_n^i, u_n^i
\rangle_{\hten
^{\otimes2}}$ is significant.\vspace*{8pt}

\textit{Third term.}
Assume for the summation terms that the indices run over $\lfloor\frac
{nt_{i-1}}{2} \rfloor+1 \le j, k \le\lfloor\frac{nt_i}{2} \rfloor$.
We have
\begin{eqnarray*}
&&
\bigl\langle u_n^i, u_n^i \bigr\rangle_{\hten^{\otimes2}} = \sum_{j,k}
f''(W_{({2j-1})/{n}}) f''(W_{(2k-1)/{n}})\\[-2pt]
&&\qquad\hspace*{51.5pt}{}\times
\bigl\langle\partial_{({2j-1})/{n}}^{\otimes2} -
\partial_{({2j-2})/{n}}^{\otimes2},
\partial_{(2k-1)/{n}}^{\otimes2} - \partial_{(2k-2)/{n}}^{\otimes2} \bigr\rangle_{\hten^{\otimes2}}.
\end{eqnarray*}
Expanding the product, observe that
\begin{eqnarray*}
&&
\bigl\langle\partial_{({2j-1})/{n}}^{\otimes2} - \partial_{(2j-2)/{n}}^{\otimes2}, \partial_{(2k-1)/{n}}^{\otimes2} -
\partial_{(2k-2)/{n}}^{\otimes2} \bigr\rangle_{\hten^{\otimes2}}\\[-2pt]
&&\qquad=
\beta_n(2j-1,2k-1)^2 - \beta_n(2j-1,2k-2)^2
\\[-2pt]
&&\qquad\quad{} -\beta_n(2j-2,2k-1)^2 +\beta_n(2j-2,2k-2)^2,
\end{eqnarray*}
where $\beta_n(\ell, m)$ is as defined for condition (v). For each $n$,
define discrete measures on $\{ 1, 2, \ldots\}^{\otimes2}$ by
\begin{eqnarray*}
\mu_n^+&:=& \sum_{j,k=1}^\infty
\beta_n(2j-1,2k-1)^2+\beta_n(2j-2,2k-2)^2
\delta_{jk};
\\[-2pt]
\mu_n^-&:=& \sum_{j,k=1}^\infty
\beta_n(2j-1,2k-2)^2+\beta_n(2j-2,2k-1)^2
\delta_{jk},
\end{eqnarray*}
where in this case $\delta_{jk}$ denotes the Kronecker delta. In the
following, we show only~$\eta_n^+$, with $\eta_n^-$ being similar. It
follows from condition (v) that for each $t > 0$,
\begin{eqnarray*}
\mu^+ \bigl( [0,t]^2 \bigr) :\!&=& \lim_{n \to\infty
}
\mu_n \biggl( \biggl\lfloor\frac{nt}{2} \biggr\rfloor, \biggl
\lfloor\frac
{nt}{2} \biggr\rfloor\biggr)
\\[-2pt]
& =& \lim_n \sum_{j,k=1}^{ \lfloor{nt/2} \rfloor}
\beta_n(2j-1,2k-1)^2+\beta_n(2j-2,2k-2)^2\\[-2pt]
&=& \eta^+(t).
\end{eqnarray*}
Moreover, if $0 < s < t$, then
\begin{eqnarray*}
&&
\mu_n \biggl( \biggl\lfloor\frac{ns}{2} \biggr\rfloor, \biggl
\lfloor\frac{nt}{2} \biggr\rfloor\biggr) \\[-2pt]
&&\qquad= \mu_n \biggl( \biggl
\lfloor\frac{ns}{2} \biggr\rfloor, \biggl\lfloor\frac
{ns}{2} \biggr
\rfloor\biggr)\\[-2pt]
&&\qquad\quad{} + \sum_{j=1}^{\lfloor{ns}/{2}\rfloor} \sum
_{k = \lfloor{ns}/{2}\rfloor+1}^{\lfloor{nt}/{2} \rfloor} \beta_n(2j-1,2k-1)^2
+\beta_n(2j-2,2k-2)^2,
\end{eqnarray*}
which converges to $\mu^+([0,s]^2)$ because the disjoint sum vanishes
by Lem\-ma~\ref{le6.1}. Hence, we can conclude that $\mu_n$ converges weakly to
the measure given by $\mu^+ ([0,s]\times[0,t]) = \eta^+ (s \wedge t)$.
It follows by continuity of $f''(W_t)$ and Portmanteau theorem that
\begin{eqnarray*}
&&\sum_{j,k=1}^{ \lfloor{nt/2} \rfloor} f''(W_{(2j-1)/{n}})
f''(W_{(2k-1)/{n}}) \bigl( \beta_n(2j-1,2k-1)^2
+\beta_n(2j-2,2k-2)^2 \bigr)
\\
&&\qquad=\int_{{\mathbb R}^2} f''(W_s)
f''(W_u) {\mathbf1}_{s<t} {
\mathbf1}_{u<t} \mu_n^+(ds,du)
\end{eqnarray*}
converges to
\[
\int_0^t f''(W_s)^2
\eta^+(ds).
\]
Combining this result with a similar integral defined for $\mu^-$, we
have for $t >0$,
\begin{eqnarray*}
&&\lim_{n\to\infty} \sum_{j,k = 1}^{ \lfloor{nt/2} \rfloor}
f''(W_{({2j-1})/{n}}) f''(W_{(2k-1)/{n}})\\
&&\qquad\hspace*{23.2pt}{}\times
\bigl\langle\partial_{({2j-1})/{n}}^{\otimes2} -
\partial_{({2j-2})/{n}}^{\otimes2}, \partial_{(2k-1)/{n}}^{\otimes2} - \partial_{(2k-2)/{n}}^{\otimes2}
\bigr\rangle_{\hten^{\otimes2}}
\\
&&\qquad = \int_0^t f''(W_s)
\mu^+(ds) -\int_0^t f''(W_s)
\mu^-(ds) = \int_0^t f''(W_s)
\eta(ds),
\end{eqnarray*}
where we define $\eta(t) = \eta^+(t)-\eta^-(t)$. It follows that on the
subinterval $[t_{i-1}, t_i]$ we have
the result
\[
\bigl\langle u_n^i, u_n^i \bigr
\rangle_{\hten^{\otimes2}} \longrightarrow\int_{t_{i-1}}^{t_i}
f''(W_s)^2 \eta(ds)
\]
in $L^1(\Omega)$ as $n \to\infty$. 

\section*{Acknowledgments}

The authors wish to thank two referees for a careful reading and
helpful comments, especially with respect to the examples section.



\printaddresses

\end{document}